\documentclass[a4paper, 10pt, twoside, notitlepage]{amsart}

\usepackage[utf8]{inputenc}
\usepackage{color}
\usepackage{amsmath} 
\usepackage{amssymb} 
\usepackage{amsthm}
\usepackage{mathdots}
\usepackage{geometry}
\usepackage{graphicx}
\graphicspath{{figures/}}
\usepackage{esint}
\usepackage[colorlinks=true,linkcolor=blue]{hyperref}
\usepackage{comment}

\usepackage{combelow}
\usepackage{mathrsfs}

\usepackage[capitalise, noabbrev, nameinlink]{cleveref}
\crefname{equation}{}{} 
\usepackage{tikz, tikz-3dplot}
\usetikzlibrary{patterns}
\usetikzlibrary{calc}

\theoremstyle{plain}
\newtheorem{thm}{Theorem}
\newtheorem{prop}{Proposition}[section]

\newtheorem{lem}[prop]{Lemma}
\newtheorem{cor}[prop]{Corollary}

\newtheorem{rmk}[prop]{Remark}

\newtheorem{example}[prop]{Example}

\newcommand {\R} {\mathbb{R}} \newcommand {\Z} {\mathbb{Z}}
\newcommand {\T} {\mathbb{T}} \newcommand {\N} {\mathbb{N}}
 
\newcommand {\p} {\partial}

\newcommand {\supp} {\text{supp}}

\newcommand {\rank} {\text{rank}}

\newcommand{\K}{\mathcal{K}}
\newcommand{\A}{\mathcal{A}}
\renewcommand{\AA}{\mathbb{A}}
\newcommand{\ran}{\text{ran}}

\DeclareMathOperator{\di}{div}

\DeclareMathOperator {\dist} {dist}

\DeclareMathOperator {\Per} {Per}

\DeclareMathOperator{\Id} {Id}
\DeclareMathOperator{\F} {\mathcal{F}}

\DeclareMathOperator{\vspan}{span}

\title[Scaling for the Two-State Problem and a $T_3$ Structure]{On Scaling Properties for Two-State Problems and for a Singularly Perturbed $T_3$ Structure}

\author{Bodgan Rai\cb{t}\u a}
\address{Scuola Normale Superiore, Centro di Ricerca Matematica Ennio De Giorgi,
P.za dei Cavalieri, 3, 56126 Pisa, Italy}
\email{bogdanraita@gmail.com}

\author{Angkana Rüland}
\address{Institut f\"ur Angewandte Mathematik, Ruprecht-Karls-Universit\"at Heidelberg, Im Neuenheimer Feld 205, 69120 Heidelberg, Germany}
\email{Angkana.Rueland@uni-heidelberg.de}

\author{Camillo Tissot}
\address{Institut f\"ur Angewandte Mathematik, Ruprecht-Karls-Universit\"at Heidelberg, Im Neuenheimer Feld 205, 69120 Heidelberg, Germany}
\email{Camillo.Tissot@uni-heidelberg.de}

\definecolor{HDred}{RGB}{198, 24, 38}
\definecolor{HDredDark}{RGB}{89, 13, 8}
\definecolor{PewterBlue}{RGB}{124, 158, 178}
\definecolor{PurpleNavy}{RGB}{82, 82, 140}
\definecolor{BudGreen}{RGB}{112, 174, 110}

\begin{document}

\begin{abstract}
In this article we study quantitative rigidity properties for the compatible and incompatible two-state problems for suitable classes of $\A$-free operators and for a singularly perturbed $T_3$ structure for the divergence operator. In particular, in the compatible setting of the two-state problem we prove that all homogeneous, first order, linear operators with affine boundary data which enforce oscillations yield the typical $\epsilon^{\frac{2}{3}}$-lower scaling bounds. As observed in \cite{CC15} for higher order operators this may no longer be the case. Revisiting the example from \cite{CC15}, we show that this is reflected in the structure of the associated symbols and that this can be exploited for a new Fourier based proof of the lower scaling bound. Moreover, building on \cite{Ruland22, Garroni04, Palombaro04}, we discuss the scaling behaviour of a $T_3$ structure for the divergence operator. We prove that as in \cite{Ruland22} this yields a non-algebraic scaling law.
\end{abstract}

\maketitle

\section{Introduction}
Rigidity and flexibility properties associated with (nonlinear) differential inclusions for the gradient have been objects of intensive study. They arise in a variety of applications, including the analysis of PDEs, e.g. the study of regularity of elliptic systems \cite{MS03, KMS03,S04}, fluid dynamics \cite{DLS09, DLS12}, geometry \cite{K55, N54,CDLS12, DLS15} and various settings in the materials sciences, e.g. the study of patterns in shape-memory alloys \cite{BJ89, BJ92, BFJK94, MS99, M99b,Rue16b}. Motivated by applications of microstructures in composites \cite{P10, PS09}, optimal design problems \cite{KW14, KW16,PW21} and micromagnetics \cite{KDOM06, DK06} as well as by recent developments on more general differential inclusion problems \cite{DPPR18,ST21,SW19}, in this article we study two instances of quantitative rigidity and flexibility properties of differential inclusions for more general operators. On the one hand, we consider \emph{constant-coefficient, homogeneous, linear differential operators} for which we discuss quantitative versions of the compatible and incompatible two-state problems.
On the other hand, we investigate quantitative properties of a $T_3$ structure for the \emph{divergence operator}.

\subsection{On quantitative results for the two-well problem for $\mathcal{A}$-free differential inclusions}

$\mathcal{A}$-free differential inclusions arise in many different settings, including linearized elasticity \cite{B03, B93},
 liquid crystal elastomers \cite{W07,CDPRZZ20} and the study of the Aviles-Giga functional \cite{LLP20, LLP22} to name just a few examples.
They have been systematically investigated in the context of compensated compactness theory in classical works such as \cite{T79, DP85, T83, MT18} but also in more recent literature on compensated compactness theory \cite{M99, FS08, GRS20, ARDPR20, R19, GRS21, R21}, in truncation results \cite{BGS21}, in classical minimization and regularity questions in the calculus of variations \cite{FM99,KR20, CG20, GLN22} and in the context of fine properties of such operators in borderline spaces \cite{DPR16,DPPR18, ARDPHR19, BDG20}. 
In the recent articles \cite{DPPR18, ST21} general $\mathcal{A}$-free versions of the incompatible two-well problem in borderline spaces and the study of $T_N$ structures have been initiated.
Motivated by these applications, in the first part of this article, we seek to study \emph{quantitative} versions of the compatible and incompatible two-state problems.

Before turning to the setting of general $\A$-free differential operators, let us recall the analogous ``classical'' setting for the gradient: Inspired by problems from materials science and phase transformations, the exact and approximate rigidity properties of differential inclusion problems for the gradient \cite{BJ89} (see also \cite{D07, DM12, P97, M99b}) with and without gauge invariance have been considered. For two energy wells without gauge invariances this amounts to the study of the differential inclusion 
\begin{align}
\label{eq:grad}
\nabla v \in \{A,B\} \mbox{ in } \Omega
\end{align}
for $A,B \in \R^{d\times d}, A \neq B$ and \(\Omega \subset \R^d\) a bounded Lipschitz domain. It is well-known that depending on the \emph{compatibility} of the wells, a dichotomy arises:
\begin{itemize} 
\item On the one hand, for \emph{incompatible} wells, i.e. if \(B-A \in \R^{d \times d}\) is not a rank-one matrix, the differential inclusion \eqref{eq:grad} is \emph{rigid} both for exact and approximate solutions:
Indeed, if $A,B$ are \emph{incompatible}, \eqref{eq:grad} only permits solutions with constant deformation gradient, which is in the following referred to as the rigidity of the two-state problem for \emph{exact} solutions. Moreover, in this setting one has that for sequences $\nabla u_k$ with $\dist(\nabla u_k, \{A,B\}) \rightarrow 0$ in measure, it necessarily holds that along a subsequence $\nabla u_k \rightarrow A$ or $\nabla u_k \rightarrow B$ in measure. We will refer to this as rigidity of the two-state problem for \emph{approximate} solutions.

We remark that various far-reaching generalizations of these results have been obtained: For settings with $SO(d)$ symmetries a quantitative version of such a result was deduced in \cite{CM04}; a further proof was found in \cite{DLS06}. The one-state problem with continuous symmetry group was studied in \cite{FJM02,LLP22b}.
 
 The article \cite{DPPR18} investigates a similar problem for two incompatible wells for a \emph{general constant coefficient, homogeneous, linear differential operator} $\mathcal{A}(D)$, providing qualitative rigidity results for the associated exact and approximate differential inclusions, including $L^1$-based frameworks.
\item On the other hand, if the wells are \emph{compatible}, i.e. if \(B-A \in \R^{d \times d}\) is a rank-one matrix, then \emph{simple laminate solutions} of \eqref{eq:grad} exist, in which the deformation gradient oscillates between the two fixed values $A,B$ and is a one-dimensional function depending only on the direction determined by the difference $B-A$. Due to the failure of rigidity on the exact level, also rigidity on the approximate level cannot be expected without additional regularization terms.
\end{itemize}

In the first part of this article we seek to consider \emph{quantitative}, $L^2$-based variants of these type of results for more general, linear differential operators $\A(D)$. In this context, we will consider the following two guiding questions:

\begin{itemize}
\item \textbf{Quantitative incompatible rigidity.} For a \emph{constant coefficient, homogeneous, linear differential operator} $\mathcal{A}(D)$ and two \emph{incompatible} wells, i.e. \(A,B \in \R^n\) such that \(B-A \notin \Lambda_\A\), cf. \eqref{eq:wave_cone} for the definition of the wave cone, do we have a quantitative rigidity result in terms of domain scaling for prescribed boundary data which are a convex combination of the two states?
Here the dimension $n$ depends on the operator $\A(D)$; for $\A(D) = \mbox{curl}$ (which corresponds to the gradient setting from \eqref{eq:grad} above) we would for instance consider $n = d \times d$.

More precisely, we seek to study the following question: Let $A,B \in \R^{n}$ be such that $B-A \notin \Lambda_{\mathcal{A}}$. Is it true that  
\begin{align*}
E_{el}(u,\chi)= \int\limits_{\Omega} |u - \chi_A A - \chi_B B|^2 dx \geq C(A,B,\lambda)|\Omega|,
\end{align*}
if  $u: \R^d \rightarrow \R^n$, $\mathcal{A(D)}u =0$ in $\R^d$, \(\chi := A \chi_A + B \chi_B \in \{A,B\}\) in \(\Omega\) and $\chi_{A}, \chi_B \in \{0,1\}$ with $\chi_{A}+\chi_B =1$ and if for some $\lambda \in (0,1)$ we have that $u=F_{\lambda} := \lambda A + (1-\lambda) B  $ in $\R^d\setminus \overline{\Omega}$?
\item \textbf{Quantitative compatible rigidity.} Let us next consider two \emph{compatible} wells $A,B \in \R^{n}$, i.e. let $A,B \in \R^{n}$ be such that $B-A \in \Lambda_{\mathcal{A}}$, see \eqref{eq:wave_cone} below, and let us again consider boundary data $F_{\lambda}:= \lambda A + (1-\lambda) B$ for some $\lambda \in (0,1)$ as above and with the set of admissible functions given by \(\mathcal{D}_{F_\lambda}\) in \eqref{eq:admissible}.
For a singularly perturbed energy similarly as in \eqref{eq:energy} is it true that as in \cite{KM92, KM94} also in the setting of a \emph{more general constant coefficient, homogeneous, linear differential operator} $\mathcal{A}(D)$ the following bound holds
\begin{align*}
\inf\limits_{\chi \in BV(\Omega;\{A,B\})}\inf\limits_{u\in \mathcal{D}_{F_{\lambda}}} (E_{el}(u,\chi) + \epsilon \int\limits_{\Omega}|\nabla \chi|) \geq C \epsilon^{2/3}?
\end{align*} 
\end{itemize}

In what follows, we will formulate the set-up, the relevant operator classes and our results on these questions.

\subsection{Formulation of the two-state problem for $\mathcal{A}$-free operators in bounded domains} \label{sec:formulation}

Following \cite{FM99,DPPR18,ST21}, we consider a particular class of \emph{linear, homogeneous, constant-coefficient operators}. The operator \(\A(D):  C^\infty(\R^d;\R^n) \to C^\infty(\R^d;\R^m)\) of order \(k \in \N\) is given in the form
\begin{align}
\label{eq:operator}
\mathcal{A}(D):= \sum\limits_{|\alpha|=k} A_{\alpha} \partial^{\alpha},
\end{align}
where $\alpha \in \N^{d}$ denotes a multi-index of length $|\alpha|:=\sum\limits_{j=1}^{d} \alpha_j$ and $A_{\alpha} \in \R^{m \times n}$ are constant matrices.
Seeking to study microstructure, in the sequel we are particularly interested in \emph{non-elliptic} operators. Here the operator $\mathcal{A}(D)$ is said to be \emph{elliptic} if its symbol 
\begin{align} \label{eq:Symbol}
\AA(\xi):= \sum\limits_{ |\alpha| = k} A_{\alpha} \xi^{\alpha}
\end{align}
is injective for all $\xi \neq 0$. If $\A(D)$ is not elliptic, there exist vectors $\xi \in \R^d \setminus \{0\}$ and $\mu \in \R^n \setminus \{0\}$ such that 
\begin{align}
\label{eq:lxi}
\AA(\xi) \mu= 0. 
\end{align}
The collection of these vectors $\mu \in \R^n \setminus \{0\}$ forms the \emph{wave cone} associated with the operator $\A(D)$: 
\begin{align}
\label{eq:wave_cone}
\Lambda_{\A} := \bigcup\limits_{\xi \in \mathbb{S}^{d-1}} \ker(\AA(\xi)).
\end{align}
The relevance of the wave cone $\Lambda_{\A}$ for compensated compactness and the existence of microstructure is well-known. For instance, for any pair $(\mu,\xi)$ as in \eqref{eq:lxi} it is possible to obtain $\A$-free \emph{simple laminate solutions}. These are one-dimensional functions $u(x):=\mu h(x \cdot \xi)$, where $h: \R \rightarrow \{0,1\}$, which obey the differential constraint \(\A(D) u = 0\) due to the choice of $\mu, \xi$ and \(u \in \{0,\mu\}\). More generally, if $k=1$, and for $\mu \in \Lambda_{\mathcal{A}}$ (see, for instance, \cite{BMS17}) it holds that 
\begin{align*}
u(x):= \mu h(x\cdot \xi_1, \dots, x\cdot \xi_{\ell})
\end{align*}
 is a solution to the differential equation $\mathcal{A}(D) u = 0$ for vectors $\xi_1, \dots, \xi_{\ell} \in \R^d\setminus \{0\}$ forming a basis of the vectorspace  
\begin{align} \label{eq:laminatDir}
V_{\mathcal{A}, \mu}:=\left\{
\xi \in \R^d: \  \AA(\xi) \mu = 0 \right\}.
\end{align}
For the row-wise curl operator (\(n=d \times d\)) this is an at most one-dimensional space, while for the row-wise divergence operator (\(n = m \times d\)), it is a space of possibly higher dimension as \(V_{\di,\mu} = \ker \mu\), leading to substantially more flexible solutions of the associated differential inclusions than for the curl.

In order to study microstructures arising as solutions to the two-state problem, for the above specified class of operators, analogously as in the gradient setting, we consider the following \(\A\)-free differential inclusion with prescribed boundary values:
\begin{align}
\label{eq:A-free}
\begin{split}
u & \in \mathcal{K} \mbox{ in } \Omega,\\
\A(D)  u &= 0 \mbox{ in } \R^d,\\
\end{split}
\end{align}
with $\mathcal{K} \subset \R^{n}$, and $\Omega \subset \R^d$ an open, bounded, simply connected domain, with appropriately prescribed boundary data.
For the two-state problem we consider \(\K = \{A,B\} \subset \R^n\).

Now, in analogy to the gradient setting, on the one hand, we call the differential inclusion for the two-state problem \eqref{eq:A-free} \emph{incompatible} if it is elliptic in the sense that $B-A \notin \Lambda_{\A}$. In this case it is proved in \cite{DPPR18} that both the exact and approximate differential inclusion \eqref{eq:A-free} are rigid.
We emphasize that incompatibility in particular excludes the presence of simple laminates.
On the other hand, the differential inclusion \eqref{eq:A-free} is said to be \emph{compatible} if $B-A \in \Lambda_{\A}$. In this case, also in the setting of more general operators, a particular class of solutions to \eqref{eq:A-free} consists of (generalized) simple laminates. Moreover, in the compatible setting, we further distinguish a particular case: We consider the subspace
\begin{align}\label{eq:I_A}
I_{\mathcal{A}} := \bigcap\limits_{\xi \in \R^d} \mbox{ker} (\mathbb{A}(\xi)) = \bigcap\limits_{|\alpha| = k} \mbox{ker} (A_{\alpha}).
\end{align}
This is the space of values that are (algebraically) unconstrained by $\mathcal{A}(D)$, meaning that for
all $u \in L^2(\R^d; I_{\mathcal{A}})$, we have that $\mathcal{A}(D)u = 0$ without taking any regularity constraints on $u$.
In the case of $I_{\mathcal{A}} = \{0\}$, the operator $\A(D)$ belongs to the class of \emph{cocanceling} operators, introduced in \cite{VS13}.

We seek to study both settings and the resulting microstructures \emph{quantitatively} in the spirit of scaling results as, for instance, in the following non-exhaustive list involving different physical applications \cite{KM92,KM94,CKO99, CDMZ20, K07, CKM22,KK11,KKO13,KW14,R16, CO12,CO09, RT21,RT22,Ruland22,RTZ19}.
To this end, for $\Omega \subset \R^d$ an open, bounded, Lipschitz set, we introduce elastic and surface energies and consider their minimization for prescribed, not globally compatible boundary data \(F_\lambda = \lambda A + (1-\lambda)B\) for some \(\lambda \in (0,1)\), where again the set of states is given by \(\K = \{A,B\}\).

Motivated by the applications from materials science, we study the following ``elastic energy''
\begin{align}
\label{eq:E_el}
E_{el}(u,\chi) : = \int\limits_{\Omega} |u - \chi|^2 dx,
\end{align}
which we minimize in the following admissible class of deformations
\begin{align}
\label{eq:admissible}
u \in \mathcal{D}_{F_{\lambda}}:= \big\{u \in L^2_{loc}( \R^d; \R^n): \ \mathcal{A}(D)u = 0 \mbox{ in } \R^d, \ u = F_{\lambda} \mbox{ in } \R^d \setminus \overline{\Omega}\big\}, \ \chi \in L^2(\Omega;\K),
\end{align}
where $k \in \N$ denotes the order of the operator $\A(D)$.
For ease of notation, here and in what follows, we often use the convention that \(\chi := \chi_A A + \chi_B B\) with \(\chi_A, \chi_B \in L^2(\Omega;\{0,1\})\) and \(\chi_A + \chi_B = 1\) in \(\Omega\).
Moreover, we further use the notation
\begin{align*}
 E_{el}(\chi;F_\lambda) := \inf_{u \in \mathcal{D}_{F_\lambda}} E_{el}(u,\chi).
\end{align*}

In addition to the ``elastic'' energy contributions, we also introduce a surface energy contribution of the form
\begin{align}
\label{eq:E_surf}
E_{surf}(\chi):= \int\limits_{\Omega}|\nabla \chi|, \quad \chi \in BV(\Omega;\K)
\end{align}
and consider the following singularly perturbed elastic energy for \(\epsilon >0\)
\begin{align} \label{eq:E-Total}
 E_\epsilon(u,\chi) := E_{el}(u,\chi) + \epsilon E_{surf}(\chi),
\end{align}
and correspondingly
\begin{align*}
 E_\epsilon(\chi;F_\lambda) := E_{el}(\chi;F_\lambda) + \epsilon E_{surf}(\chi) = \inf_{u \in \mathcal{D}_{F_\lambda}} E_\epsilon(u,\chi).
\end{align*}
We note that this can be defined for an arbitrary set of states \(\mathcal{K}\) and suitable boundary data \(F_{\lambda}\).

With these quantities in hand, we can formulate the following quantitative rigidity results for the two-state problems: 

\begin{thm} \label{thm:TwoWell}
Let \(d,n,m,k \in \N, n > 1\).
Let $\Omega \subset \R^d$ be a bounded Lipschitz domain, and let \(\A(D)\) be as in \eqref{eq:operator} with the wave cone \(\Lambda_\A\) given in \eqref{eq:wave_cone}.
Let $A,B \in \R^{n}$ and let $\chi = \chi_A A + \chi_B B \in L^2(\Omega;\{A,B\})$.
Further, let the elastic and surface energies \(E_{el}\), \(E_{surf}\) be given as in \eqref{eq:E_el} and \eqref{eq:E_surf}, respectively.
For \(\lambda \in (0,1) \) set \(F_\lambda = \lambda A + (1-\lambda)B \in \R^n\) and consider \(\mathcal{D}_{F_\lambda}\) as in \eqref{eq:admissible}. The following results hold:
\begin{itemize}
\item [(i)]\emph{Incompatible case:} Assume that $B-A \notin \Lambda_{\A}$. Then there is a constant \(C= C(A,B) >0\) such that,
\begin{align*}
\inf\limits_{\chi \in L^2(\Omega;\{A,B\})}\inf\limits_{u\in \mathcal{D}_{F_{\lambda}}} E_{el}(u,\chi) \geq C (\min\{\lambda, 1-\lambda\})^2|\Omega|.
\end{align*}
\item[(ii)] \emph{Compatible case:}
Assume that \(\A(D)\) is one-homogeneous, i.e. \(k=1\) in \eqref{eq:operator}, and that $B-A \in \Lambda_{\A} \setminus I_{\mathcal{A}}$. Then, there exist \(C= C(\A(D),A,B,\Omega, d, \lambda) >0\) and $\epsilon_0=\epsilon_0(\A(D),A,B,\Omega, d, \lambda)>0$ such that for \(\epsilon \in (0,\epsilon_0)\)
\begin{align*}
\inf\limits_{\chi \in BV(\Omega;\{A,B\})}\inf\limits_{u\in \mathcal{D}_{F_{\lambda}}} (E_{el}(u,\chi) + \epsilon E_{surf}(\chi)) \geq C \epsilon^{2/3}.
\end{align*}
Furthermore, if we assume \(\A(D) = \di\) and \(\Omega = [0,1]^d\), then there exists a constant \(c = c(A,B,\lambda) > 0\) such that for \(\epsilon > 0 \) we also have the matching upper bound
\begin{align*}
 \inf_{\chi \in BV(\Omega;\{A,B\})} \inf_{u \in \mathcal{D}_{F_\lambda}} (E_{el}(u,\chi) + \epsilon E_{surf}(\chi)) \leq c \epsilon^{2/3}.
\end{align*}
\item[(iii)] \emph{Super-compatible case:} Assume that $A-B \in I_{\mathcal{A}}$. Then,
\begin{align*}
\inf\limits_{\chi \in BV(\Omega;\{A,B\})}\inf\limits_{u\in \mathcal{D}_{F_{\lambda}}} (E_{el}(u,\chi) + \epsilon E_{surf}(\chi)) = 0.
\end{align*}
\end{itemize}
\end{thm}

We highlight that in our discussion of the compatible case, we have restricted ourselves to operators of order one. This is due to the fact that for higher order operators it is expected that more complicated microstructures may arise.
This is also reflected in the Fourier space properties of the symbol \(\AA\). We refer to \cref{sec:higher}, see \cref{prop:symm_grad}, for a brief discussion of this, illustrating that the scaling may, in general, be no longer of the order $\epsilon^{\frac{2}{3}}$ in the higher order setting.

Let us discuss a prototypical example of the above results:

\begin{example}
\label{ex:div}
As an example of the above differential inclusion, we consider the case in which $\mathcal{A}(D) = \di: C^\infty(\R^d;\R^{m \times d}) \to C^\infty(\R^d;\R^m)$ row-wise for matrix fields \(u: \R^d \to \R^{m \times d}\). In this case the boundary value problem under consideration turns into the following differential inclusion:
\begin{align*}
\begin{split}
u & \in \{A,B\} \mbox{ in } \Omega, \\
\di  u &= 0 \mbox{ in } \R^d,\\
u & = F_{\lambda} \mbox{ in } \R^d \setminus \overline{\Omega},
\end{split}
\end{align*}
for some $F_{\lambda}:= \lambda A + (1-\lambda) B$, $\lambda \in (0,1)$.
Such differential inclusions are related to applications in shape-optimization as, for instance, in \cite{KW14}.

The divergence is applied row-wise to matrix fields \(u : \R^d \to \R^{m \times d}\), i.e. \(\A(D)u = \sum_{j=1}^d (\p_j u) e_j \) with the matrix vector product.
Hence \(\AA(\xi)M = M \xi\) for \(M \in \R^{m \times d}\) and thus the wave cone is given by
\begin{align*}
 \Lambda_{\di} = \{ M \in \R^{m \times d}: \text{there is } \xi \in \R^d \mbox{ with } M \xi = 0\}.
\end{align*}
Moreover the divergence is a cocanceling operator as $I_{\di} = \{0\}$.
\end{example}

We emphasize that \cref{ex:div} is indeed a prototypical example and plays a central role in the study of first order operators in that all first order operators can be reduced to this model operator by a suitable linear transformation, see \cite[Appendix]{ST21} and also \cref{sec:RoleDivergence} below. We emphasize that this reduction is particularly useful if the differential inclusion is \emph{incompatible} or if the boundary data are in $\Lambda_{\A} \setminus I_{\A}$. As a consequence, quantitative \emph{lower bound estimates} for \emph{incompatible} differential inclusions for first order operators, e.g. for $T_N$ structures as qualitatively studied in \cite{ST21}, or for \emph{compatible, but not super-compatible boundary data} can be deduced from the ones of the divergence operator (see \cref{prop:div_reduction} and, in general, the discussion in \cref{sec:RoleDivergence}). A reduction to an equivalent problem for a modified operator and modified boundary data to the setting involving a cocanceling operator will be discussed in \cref{sec:super}, see  \cref{prop:reduction_cocancelling} and \cref{cor:two-well_reduction}.

\subsection{Quantitative rigidity of a $T_3$ structure for the divergence operator} \label{sec:IntroT3}
In the second part of the article, building on the works \cite{Ruland22, Garroni04, Palombaro04} and motivated by the highlighted considerations on the role of the divergence operator, we study the quantitative rigidity of the $T_3$ configuration
\begin{align}
\label{eq:problem}
u \in \{A_1,A_2,A_3\} \mbox{ a.e. in  } \Omega, \quad  \di u = 0 \mbox{ in } \R^3,
\end{align}
where $\Omega = [0,1]^3$, $u: \R^3 \rightarrow \R^{3\times 3}$ and
\begin{align}
\label{eq:problem1}
A_1 = 0_{3 \times 3}, A_2 = \begin{pmatrix} - \frac{1}{2} & 0 & 0 \\ 0 & \frac{2}{3} & 0 \\ 0 & 0 & 3 \end{pmatrix}, A_3 = Id_{3 \times 3}.
\end{align}
By virtue of the results from \cite{Garroni04,Palombaro04} this problem is flexible for approximate solutions but rigid on the level of exact solutions: 
\begin{itemize}
\item More precisely, on the level of \emph{exact solutions} to \eqref{eq:problem}, only constant solutions $u \equiv A_j$ for $j\in \{1,2,3\}$ obey the differential inclusion. 
\item Considering however \emph{approximate solutions}, i.e. sequences $(u_k)_{k\in \N}$ such that
\begin{align*}
\dist(u_k, \{A_1,A_2,A_3\}) \rightarrow 0 \mbox{ in measure as } k \rightarrow \infty, \ \di u_k = 0 \mbox{ for all } k \in \N,
\end{align*}
there exists a sequence of approximate solutions $(u_k)_{k\in \N}$ for \eqref{eq:problem} such that there is no subsequence which converges in measure to one of the constant deformations $\{A_1,A_2,A_3\}$.
\end{itemize}
 Compared to the setting of the gradient, for the divergence operator rigidity for approximate solutions is already lost for the three-state problem (while this arises only for four or more states for the gradient \cite{T93, CK02}, see also \cite{M99b} for further instances in which the Tartar square was found and used).

As in \cite{Ruland22} we here study a \emph{quantitative} version of the dichotomy between rigidity and flexibility: We consider the singularly perturbed variant of \eqref{eq:problem} as in \cref{sec:formulation}
\begin{align}
\label{eq:energy}
E_{\epsilon}(u,\chi):= \int\limits_{\Omega}|u - \chi|^2 dx + \epsilon \int\limits_{\Omega} |\nabla \chi|
\end{align}
under the constraint $u\in \mathcal{D}_F$, cf. \eqref{eq:admissible}, with $\A(D) = \di$ and where $F\in \{A_1,A_2,A_3\}^{qc}$ (see \cref{sec:not} for the definition of the $\A$-quasi-convexification of a compact set).
Here the function $\chi \in BV(\Omega; \{A_1,A_2,A_3\})$ denotes the phase indicator of the ``phases'' $A_1,A_2,A_3$, respectively.

Adapting the ideas from \cite{Ruland22} to the divergence operator in three dimensions, we prove the following scaling result:

\begin{thm}
\label{thm:scaling_T3}
Let $\Omega = [0,1]^3$ and let \(\K = \{A_1,A_2,A_3\}\) with $A_j$, $j\in \{1,2,3\}$ given in \eqref{eq:problem1}, $F\in \K^{qc}\setminus \K$, $E_{\epsilon}$ be as in \eqref{eq:energy} above for the divergence operator \(\A(D) = \di\), and consider \(\mathcal{D}_F\) given analogously as \eqref{eq:admissible}. Then, there exist constants $c = c(F) >0$ and \(C = C(F) > 1\) such that for any $\nu \in (0,\frac{1}{2})$ there is $\epsilon_0 = \epsilon_0(\nu,F)>0$ and $c_{\nu}>0$ such that for $\epsilon \in (0,\epsilon_0)$ we have
\begin{align*}
C^{-1} \exp(-c_{\nu}|\log(\epsilon)|^{\frac{1}{2}+\nu}) \leq  \inf\limits_{\chi 
\in BV(\Omega;\K)} \inf\limits_{u \in \mathcal{D}_F} E_{\epsilon}(u,\chi) \leq C \exp(- c|\log(\epsilon)|^{\frac{1}{2}}).
\end{align*}
\end{thm} 

Let us comment on this result: As in \cite{Ruland22} we obtain essentially matching upper and lower scaling bounds with less than algebraic decay behaviour as $\epsilon \rightarrow 0$, reflecting the infinite order laminates underlying the $T_3$ structure and the fact that the problem is ``nearly'' rigid. As in \cite{Ruland22} a key step is the analysis of a ``quantitative chain rule in a negative Sobolev space'' which results from the interaction of Riesz type transforms and a nonlinearity originating from the ``ellipticity'' of the differential inclusion. Both in the upper and the lower bound, these estimates for the divergence operator however require additional care due to the three-dimensionality of the problem. In the upper bound construction this is manifested in the use of careful cut-off arguments; in the lower bound, a more involved iterative scheme has to be used to reduce the possible regions of concentration in Fourier space.

Similarly, as in \cite{RT22} the scaling law from \cref{thm:scaling_T3} is obtained as a consequence of a rigidity estimate encoding both the \emph{rigidity} and \emph{flexibility} of the $T_3$ differential inclusion (in analogy to the $T_4$ case from \cite[Proposition 3]{RT22}).

\begin{prop}
\label{prop:rig_est}
Let $\Omega = [0,1]^3$ and let \(\K = \{A_1,A_2,A_3\}\) with $A_j$, $j\in \{1,2,3\}$ given in \eqref{eq:problem1}, $F\in \K^{qc}\setminus \K$. Let $\chi_{jj}$ denote the diagonal entries of the matrix $\chi \in BV(\Omega;\K)$ and denote by $E_{\epsilon}^{per}(\chi;F)$ the periodic singularly perturbed energy (see \eqref{eq:en_total_periodic} and \cref{sec:aux_lower}).
Then, there exists $\epsilon_0>0$ such that for any $\nu \in (0,1)$ there is a constant $c_{\nu}>0$ such that for $\epsilon \in (0,\epsilon_0)$ we have
\begin{align*}
\sum\limits_{j=1}^{3}\|\chi_{jj}-\langle \chi_{jj} \rangle\|_{L^2([0,1]^3)}^2 \leq   \exp( c_{\nu} |\log(\epsilon)|^{\frac{1}{2} + \nu}) E_{\epsilon}^{per}(\chi;F)^{\frac{1}{2}}.
\end{align*}
\end{prop}

We emphasize that in parallel to the setting of the Tartar square, this estimate quantitatively encodes both rigidity and flexibility of the differential inclusion, as it measures the distance to the constant state  (and thus reflects \emph{rigidity} of the exact differential inclusion) but also quantifies the ``price'' for this in terms of a ``high energy'' scaling law (and thus reflects the underlying flexibility of the approximate problem). Moreover, due to the \emph{flexibility} of the differential inclusion, we stress that such an estimate can only be inferred for a combination of elastic \emph{and} surface energies.

\subsection{Outline of the article}
The remainder of the article is structured as follows: After briefly recalling relevant notation and facts on convex hulls related to the operator $\A(D)$, we first discuss the compatible and incompatible two-state problems in \cref{sec:two-state}. Here we begin by discussing the incompatible setting, which is a consequence of direct elliptic estimates in \cref{sec:incompat} and then turn to the lower bounds for the compatible setting in \cref{sec:compat}. The super-compatible case is then treated in \cref{sec:super}. It is also in this section that we discuss a reduction to cocanceling operators. We revisit the scaling of a prototypical higher order operator from \cite{CC15} in \cref{sec:higher} and explain how our scheme of deducing lower bounds also yields a quick Fourier based proof of the lower scaling bound from \cite{CC15}.

In \cref{sec:T3} we then turn to the $T_3$ differential inclusion, for which we first prove upper bounds in \cref{sec:upper} and then adapt the ideas from \cite{Ruland22} to infer essentially matching lower bounds in  \cref{sec:lower}.

In the Appendix, we complement the general lower bounds for first order differential operators with upper bounds for the specific case of the divergence operator (\cref{sec:construc}). Moreover, in \cref{sec:ComparisonDivergenceProblem} we discuss the reduction to the divergence operator.

\section{Notation}
\label{sec:not}

In this section we collect the notation which is used throughout the article.

\begin{itemize}
\item For a set \(U\), we denote by \(d_U\) the (possibly smoothed-out) distance to this set: \(d_U(x) = \operatorname{dist}(x,U) = \inf_{y \in U}|x-y|\) and denote by \(\chi_U\) the (in some places smoothed-out) indicator function of this set.
\item For a finite set \(U\) and a function \(f: U \to \R^n\) we denote the mean by \(\langle f \rangle = \frac{1}{|U|} \int_U f dx\).
 \item For a function \(f \in L^2(\T^d)\) or \(f \in L^2(\R^d)\) we denote the Fourier transform by \(\mathcal{F}(f)(k) =  \hat{f}(k) =  (2 \pi)^{-\frac{d}{2}}\int_{\T^d} e^{- i k \cdot x} f(x) dx\) (\(k \in \Z^d\)) or \(\mathcal{F}(f)(\xi) = \hat{f}(\xi) =(2 \pi)^{-\frac{d}{2}} \int_{\R^d} e^{-i \xi \cdot x} f(x) dx\) (\(\xi \in \R^d\)).
 For a function \(h \in L^{\infty}(\R^d) \), we denote by \(h(D)\) the corresponding Fourier multiplier
 \begin{align*}
  h(D)f = \mathcal{F}^{-1}(h(\cdot)\hat{f}(\cdot)).
 \end{align*}
 \item We usually denote the phase indicator by \( \chi \in L^2(\Omega;\K)\), with the component functions given by \( \chi = (\chi_{i})\), $i\in\{1,\dots,n\}$. Moreover, we use \(F \in \K^{qc}\) (where $\K^{qc}$ is introduced below) as the exterior data.
 \item The set of admissible 'deformations' \(u\) for an operator of order $k\geq 1$ is given by \(\mathcal{D}_F := \mathcal{D}_F^{\A} := \{u \in L^2_{loc}(\R^d ; \R^n): \A(D) u = 0 \text{ in } \R^d, \ u = F \text{ in } \R^d \setminus \bar{\Omega}\}\). Here the equation $\A(D)u=0$ is considered in a distributional sense.
 \item As introduced in \cref{sec:formulation}, for \(F \in \K^{qc}\) we consider the elastic energy with \(u \in \mathcal{D}_F, \chi \in L^2(\Omega;\K)\):
 \begin{align} \label{eq:E-Elastic}
                                    E_{el}(u,\chi) = \int_\Omega |u-\chi|^2 dx, \ E_{el}(\chi;F) = \inf_{u \in \mathcal{D}_F} E_{el}(u,\chi),
                                   \end{align}

 and, for $\chi \in BV(\Omega;\mathcal{K})$, the surface energy as the total variation norm:
 \begin{align} \label{eq:E-Surface}
                        E_{surf}(\chi) = \Vert \nabla \chi \Vert_{TV(\Omega)} = \int_\Omega |\nabla \chi|.
                       \end{align}
 \item The total energy with \(\epsilon >0\) is given by \(E_{\epsilon}(u,\chi) = E_{el}(u,\chi) + \epsilon E_{surf}(\chi)\).
 \item We write \(f \sim g\) if there are constants \(c,C>0\) such that \(c f \leq g \leq C f\).
 \item We use the notation \(\Vert \cdot \Vert_{\dot{H}^{-1}}\) for the homogeneous \(H^{-1}\) semi-norm for \(f \in H^{-1}(\T^3;\R)\):
 \begin{align*}
  \Vert f \Vert_{\dot{H}^{-1}}^2 = \sum_{k \in \Z^3\setminus \{0\}} \frac{1}{|k|^2} |\hat{f}(k)|^2.
 \end{align*}

\end{itemize}
We further recall the notions of the $\Lambda_{\A}$-convex hull of a set (see \cite{BMS17}) and the \(\A\)-quasi-convex hull of a set:
\begin{itemize}
\item Let $\Lambda_{\A} \subset \R^{n}$ be the wave cone from \eqref{eq:wave_cone} and let $K\subset \R^{n}$ be a compact set. For $j\in \N$, we then define $K^{(j)}$ as follows:
\begin{align*}
K^{(0)} &:= K, \\ 
K^{(j)} &:=\{M \in \R^{n}: \ M = \lambda A + (1-\lambda) B: \ A,B \in K^{(j-1)}, \ B-A \in \Lambda_{\A}, \ \lambda \in [0,1]\}. 
\end{align*}
Moreover, we define the \(\Lambda_{\A}\)-convex hull \(K^{lc}\):
\begin{align*}
K^{lc}:= \bigcup\limits_{j=0}^{\infty} K^{(j)}.
\end{align*}
We define the \emph{order of lamination} of a matrix $M \in \R^{n}$ to be the minimal $j\in \N$ such that $M\in K^{(j)}$. 
In analogy to the gradient case, we will also refer to $K^{lc}$ as the \emph{laminar convex hull}.
\item We recall that the \emph{$\A$-quasi-convex hull} \(K^{qc}\) of a compact set \(K \subset \R^n\) is defined by duality to $\A$-quasi-convex functions \cite{FM99}. We recall that the $\di$-quasi-convex hull of the $T_3$ matrices from \eqref{eq:problem}, \eqref{eq:problem1} has been explicitly characterized in \cite[Theorem 2]{PS09} to consist of the union of the closed triangle formed by the matrices $S_1S_2S_3$ and the three ``legs'' formed by the line segments $A_j S_j$. The matrices $S_1,S_2,S_3 \in \R^{3\times 3}$ are introduced in \cref{sec:T3} below.
Given the set $\{A_1,A_2,A_3\} \subset \R^{3\times 3}$ from \eqref{eq:problem}, \eqref{eq:problem1}, as in the gradient case, we denote its $\A$-quasi-convex hull by $\{A_1,A_2,A_3\}^{qc}$.
\end{itemize}

\section{Quantitative Results on the Two-State Problem}
\label{sec:two-state}

In this section, we study quantitative versions of the two-state problem for general, linear, constant coefficient, homogeneous operators, always considering the divergence operator as a particular model case.

Building on the precise formulation of the problem from \eqref{eq:A-free}, in \cref{sec:elastic_energy}, we first characterize the elastic energies in terms of the operator $\mathcal{A}(D)$. With this characterization in hand, using ellipticity, we next prove the quantitative $L^2$ bounds in the \emph{incompatible} two-well case (\cref{sec:incompat}) and \emph{lower} $\epsilon^{\frac{2}{3}}$-scaling bounds for the \emph{compatible} case with first order, linear operators (\cref{sec:compat}). In \cref{sec:super}, we prove \cref{thm:TwoWell}(iii) and deduce a reduction to cocanceling operators.
In \cref{sec:higher} we briefly discuss the role of degeneracies in the symbol of the elastic energies which may arise for higher order, linear differential operators and which may thus lead to an alternative scaling behaviour different from the $\epsilon^{\frac{2}{3}}$-scaling bound.

Finally, in  \cref{sec:construc} we complement the lower bounds from this section with matching \emph{upper} bounds for the special case of $\A(D) = \di$. Similar constructions are also known for the gradient, the symmetrized gradient and lower dimensional problems from micromagnetics \cite{CC15,CKO99, OV10, CO09,CO12}.

\subsection{Elastic energy characterization}

\label{sec:elastic_energy}

We begin by recalling an explicit lower bound for the elastic energy in terms of the operator $\mathcal{A}(D)$ (see, for instance also \cite[discussion before Lemma 1.17]{KR20}) in an, for us, convenient form:

\begin{lem}[Fourier characterization of the elastic energy]
\label{lem:AFreeElast}
Let \(d,n,k \in \N\).
Let $\Omega \subset \R^d$ be a bounded Lipschitz domain with associated indicator function $\chi_{\Omega}$ and let \(\K \subset \R^n\).
Let \(\A(D) = \sum_{|\alpha|=k} A_\alpha \p^\alpha\) be as in \eqref{eq:operator} with the symbol \(\AA\), cf. \eqref{eq:Symbol}, and \(E_{el}\) be as in \eqref{eq:E-Elastic} and \(F \in \K^{qc}\).
Then there is a constant $c=c(\mathcal{A})>0$ such that for any \(\chi \in L^2(\Omega;\K)\)
\begin{align*}
 E_{el}(\chi;F) &\geq \inf_{\A(D) u = 0} \int_{\R^d} \Big| (u-F) - (\chi-F\chi_\Omega) \Big|^2 dx \\
 & = \int_{\R^d} \Big| \AA(\xi)^* \big(\AA(\xi) \AA(\xi)^*)^{-1} \AA(\xi) (\hat{\chi}-F\hat{\chi}_{\Omega}) \Big|^2 d\xi \geq c \int_{\R^d} \left| \AA(\frac{\xi}{|\xi|})(\hat{\chi}-F\hat{\chi}_\Omega) \right|^2 d\xi,
 \end{align*}
 where the infimum is taken over all \(u \in L^2_{loc}(\R^d ; \R^n)\) that fulfill \(\A(D) u = 0\) in \(\R^d\).
\end{lem}

\begin{proof}
The proof follows from a projection argument in Fourier space.
 
\emph{Step 1: Whole space extension, Fourier and pseudoinverse of the differential operator.}
We begin by transforming our problem to one on the whole space \(\R^d\), introducing the whole space extension \(w := u - F\):
 \begin{align*}
  E_{el}(u,\chi) & = \int_\Omega \Big| u - \chi \Big|^2 dx = \int_\Omega \Big| u - F - (\chi-F) \Big|^2 dx \\
  & = \int_{\R^d} \Big| w - (\chi - F \chi_\Omega) \Big|^2 dx,
 \end{align*}
 where we have extended all functions in the integrand by zero outside of $\Omega$.

 In the following,  we write \(\tilde{\chi} := \chi-F\chi_{\Omega}\). Fourier transforming the expression for the elastic energy then leads to
 \begin{align*}
  E_{el}(u,\chi) = \int_{\R^d} \Big| \hat{w} - \hat{\tilde\chi} \Big|^2 d\xi.
 \end{align*}
Further, seeking to deduce a lower bound, we neglect the boundary data for $w$, obtaining 
\begin{align*}
 E_{el}(\chi;F) = \inf_{u \in \mathcal{D}_{F}} \int_{\Omega} |u-\chi|^2 dx = \inf_{w \in  \mathcal{D}_{0}} \int_{\R^d} |w-\tilde{\chi}|^2 dx \geq \inf_{\A(D) w = 0} \int_{\R^d} |\hat{w} - \hat{\tilde{\chi}}|^2 d\xi.
\end{align*}
Minimizing the integrand \(\hat{w}\) (and still denoting the minimizer by \(\hat{w}\)) for each fixed mode $\xi\in \R^d \setminus \{0\}$, we infer that
\begin{align*}
 \Big| \hat{w}(\xi) - \hat{\tilde\chi}(\xi) \Big| = \Big|  \Pi_{\ker \AA(\xi)} \hat{\tilde\chi}(\xi) - \hat{\tilde\chi}(\xi) \Big| = \Big| \Pi_{\ran \AA(\xi)^*} \hat{\tilde\chi}(\xi) \Big| =  \Big| \AA(\xi)^* \big(\AA(\xi) \AA(\xi)^*)^{-1} \AA(\xi) \hat{\tilde\chi}(\xi) \Big|,
\end{align*}
where we view \(\AA(\xi) \AA(\xi)^*: \ran \AA(\xi) \to \ran\AA(\xi)\) as an isomorphism.
This directly implies the claimed lower bound for $E_{el}(\chi;F)$ in terms of the symbol $\AA$ and its pseudoinverse.

\emph{Step 2: Proof of the final estimate.}
Now to show the final estimate for the elastic energy, we seek to bound $|\mathbb{A}(\xi)^\ast (\mathbb{A}(\xi)\mathbb{A}(\xi)^\ast)^{-1} \mathbb{A}(\xi) x|$ for every $\xi \in \mathbb{R}^d$, $x\in \mathbb{R}^n$ from below in terms of $|\mathbb{A}(\xi) x|$.
As the projection operator is zero-homogeneous, we can reduce to $\xi \in \mathbb{S}^{d-1}$, and thus can use the continuity of $ \mathbb{S}^{d-1} \ni \xi \mapsto \mathbb{A}(\xi)$ and the compactness of $ \mathbb{S}^{d-1}$ for the desired bound: For any $x\in \mathbb{R}^n, \xi \in \mathbb{S}^{d-1}$ it hence holds
\begin{align*}
 \Big| \mathbb{A}(\xi) x \Big| = \Big| \mathbb{A}(\xi) \mathbb{A}(\xi)^\ast (\mathbb{A}(\xi)\mathbb{A}(\xi)^\ast)^{-1} \mathbb{A}(\xi) x \Big| \leq \Big| \mathbb{A}(\xi)^\ast (\mathbb{A}(\xi)\mathbb{A}(\xi)^\ast)^{-1} \mathbb{A}(\xi) x \Big| \sup_{\xi \in \mathbb{S}^{d-1}}(|\mathbb{A}(\xi)|).
\end{align*}
As $\mathcal{A}(D) \neq 0$, we have that $0<\sup_{\zeta \in \mathbb{S}^{d-1}}|\mathbb{A}(\zeta)|\leq C < \infty$. Dividing by this and plugging this into the expression with the pseudoinverse, we obtain
\begin{align*}
E_{el}(u,\chi)&\geq \int\limits_{\R^d}   \Big| \AA(\xi)^* \big(\AA(\xi) \AA(\xi)^*)^{-1} \AA(\xi) \hat{\tilde\chi} \Big|^2 d\xi
\geq \frac{1}{\sup_{\zeta \in \mathbb{S}^{d-1}}|\mathbb{A}(\zeta)|^2} \int\limits_{\R^d}   \Big|\AA(\frac{\xi}{|\xi|}) \hat{\tilde\chi} \Big|^2 d\xi,
\end{align*}
which concludes the argument.
\end{proof}

We emphasize that we are neglecting the boundary conditions for \(\hat{w}(\xi) := \Pi_{\ker \AA(\xi)} \hat{\tilde\chi}(\xi)\) as we do not calculate the projection of \(\chi\) onto \(\mathcal{D}_F\), hence the above Fourier bounds only provide lower bounds for the elastic energy.

We apply the lower bound from \cref{lem:AFreeElast} to the two-well problem:

\begin{cor} \label{cor:ElasticEnergyTwoWell}
Let \(d,n \in \N\).
Let \(\Omega,\A(D),\AA, E_{el}\) be as in \cref{lem:AFreeElast}.
 Consider \(\K = \{A,B\} \subset \R^n\) with \(F_\lambda = \lambda A + (1-\lambda)B\) for $\lambda \in (0,1)$.
 Then there exists a constant \(C = C(\A(D)) > 0\) such that for any \(\chi = \chi_A A + \chi_B B \in L^2(\Omega;\K)\), extended to \(\R^d\) by zero, it holds
 \begin{align*}
  E_{el}(\chi;F_\lambda) \geq C \int_{\R^d} \left| ((1-\lambda)\hat{\chi}_A - \lambda \hat{\chi}_B) \AA(\frac{\xi}{|\xi|}) (A-B) \right|^2 d\xi.
 \end{align*}
\end{cor}

\begin{proof}
 Using the expression of $F_{\lambda}$ in terms of $A,B,\lambda$ yields \(A- F_\lambda = (1-\lambda)(A-B)\) and \(B-F_\lambda = - \lambda (A-B)\).
 Thus, the fact that \(\chi = \chi_A A + \chi_B B\) and \cref{lem:AFreeElast} imply
 \begin{align*}
  E_{el}(\chi;F_\lambda) \geq C\int_{\R^d} \left| \AA(\frac{\xi}{|\xi|}) (\hat{\chi}-F_\lambda \hat{\chi}_\Omega ) \right|^2 d\xi = C \int_{\R^d} \left|((1-\lambda) \hat{\chi}_A - \lambda \hat{\chi}_B) \AA(\frac{\xi}{|\xi|}) (A-B) \right|^2 d\xi.
 \end{align*}
\end{proof}

 \begin{rmk}[The divergence operator]
  As seen in \cref{ex:div}, in the case of \(\A(D) = \di\), we have for \(u \in C^\infty(\R^d;\R^{d \times d})\) (note that we chose square matrices out of simplicity)
  \begin{align*}
   \A(D) u = \sum_{i=1}^d (\p_i u)e_i \in C^\infty(\R^d;\R^d).
  \end{align*}
  With this we can calculate
  \begin{align*}
   \AA(\xi)M  = \sum_{i=1}^d \xi_i M e_i = M \xi . 
  \end{align*}
  This, in particular, shows that the adjoint operator is given by \(\AA(\xi)^*: \R^d \to \R^{d \times d}\) with
  \begin{align*}
   \AA(\xi)^* x = x \otimes \xi.
  \end{align*}
Therefore \( \AA(\xi)\AA(\xi)^* x = (x \otimes \xi) \xi = |\xi|^2 x \), and the projection in the lower bound for the elastic energy of \cref{lem:AFreeElast} takes the desired form
  \begin{align*}
   \AA(\xi)^* \Big( \AA(\xi) \AA(\xi)^* \Big)^{-1} \AA(\xi) M = \frac{1}{|\xi|^2} (M\xi \otimes \xi).
  \end{align*}
  Furthermore it holds
  \begin{align*}
   \Big| \AA(\xi)^* \Big( \AA(\xi) \AA(\xi)^* \Big)^{-1} \AA(\xi) M \Big| = \Big| \AA(\frac{\xi}{|\xi|})M \Big|.
  \end{align*}
 \end{rmk}

\subsection{The incompatible two-well problem and scaling}
\label{sec:incompat}

As a first application of the Fourier characterizations from the previous section, we prove a quantitative lower bound for the incompatible two-well problem. 
We emphasize that -- as in \cite{DPPR18} -- this argument is an elliptic argument and thus can be applied to all linear, constant coefficient homogeneous operators. Indeed, the following result holds:

 \begin{prop}
 Let \(d,n \in \N\).
  Let $\Omega \subset \R^d$ be a bounded Lipschitz domain, let \(\A(D)\) be given in \eqref{eq:operator} and \(A,B \in \R^n\) with \(B-A \notin \Lambda_\A\), cf. \eqref{eq:wave_cone}, further let \(F_\lambda = \lambda A+ (1-\lambda)B\) for some $\lambda \in (0,1)$, and let $E_{el}$ be as in \eqref{eq:E-Elastic} with $\mathcal{D}_{F_\lambda}$ given in \eqref{eq:admissible}. Then there is \(C = C(A,B,\lambda, \A(D)) > 0\), such that for any \(\chi \in L^2(\Omega;\{A,B\})\)
  \begin{align*}
   \inf\limits_{u \in \mathcal{D}_{F_{\lambda}}} E_{el}(u,\chi) \geq C |\Omega|.
  \end{align*}
 \end{prop}

 \begin{proof}
By virtue of \cref{cor:ElasticEnergyTwoWell}, we have the lower bound
  \begin{align*}
   E_{el}(\chi;F_\lambda) \geq C \int_{\R^d} \left| ((1-\lambda)\hat{\chi}_A - \lambda \hat{\chi}_B) \AA(\frac{\xi}{|\xi|}) (A-B) \right|^2 d\xi,
  \end{align*}
where the constant only depends on the operator \(\A(D)\).

As $(A-B) \notin \ker \AA (\xi)$ for all $\xi \in \mathbb{R}^d$ and thus $|\mathbb{A}(\xi)(A-B)| > 0$ for any $\xi \in \mathbb{R}^d$, by continuity of $\xi \mapsto \mathbb{A}(\xi)$ and compactness of $\mathbb{S}^{d-1}$, this implies $|\mathbb{A}(\frac{\xi}{|\xi|})(A-B)| \geq C(A,B) > 0$.
  Hence,
  \begin{align*}
   E_{el}(u,\chi) & \geq C^2 \int_{\R^d} |(1-\lambda)\hat{\chi}_A - \lambda \hat{\chi}_B|^2 d\xi = C^2 \int_{\R^d} |(1-\lambda)\chi_A - \lambda \chi_B|^2 dx \\
   & \geq C^2 ( \int_{\Omega_A} (1-\lambda)^2 dx + \int_{\Omega_B} \lambda^2 dx) \geq C^2 \min\{\lambda,1-\lambda\}^2 |\Omega|,
  \end{align*}
where $\Omega_A:=\supp(\chi_A)\subset \overline{\Omega}$ and $\Omega_B:=\supp(\chi_B)\subset \overline{\Omega}$.
 \end{proof}
 
\begin{rmk}
We emphasize that this result can be viewed as an incompatible nucleation bound.
\end{rmk}

\subsection{The compatible two-well case and scaling}
\label{sec:compat}

We next turn to the setting of two compatible wells and restrict our attention to \emph{first order operators}. In this case, we claim the following $\epsilon^{\frac{2}{3}}$-lower scaling bound. This is in analogy to the situation for the gradient which had first been derived in the seminal works \cite{KM92,KM94}.

\begin{prop}
\label{prop:compatible_2well_A}
Let \(d,n \in \N\).
Let $\Omega \subset \R^d$ be a bounded Lipschitz domain and let \(\A(D)\) be a first order operator as in \eqref{eq:operator}.
Let $A,B \in \R^{n}$ be such that $B-A\in \Lambda_{\mathcal{A}}\setminus I_\A$ where \(\Lambda_\A\) and $I_\A$ are given in \eqref{eq:wave_cone} and \eqref{eq:I_A}, and define $F_{\lambda} := \lambda A + (1-\lambda)B$ for some $\lambda \in (0,1)$. Let $E_\epsilon:= E_{el} + \epsilon E_{surf}$ be given in \eqref{eq:E-Total} with \(\mathcal{D}_{F_\lambda}\) defined in \eqref{eq:admissible}. Then, there exist $C = C(\A(D),A,B,\Omega, d, \lambda)>0$ and $\epsilon_0=\epsilon_0(\A(D),A,B,\Omega, d, \lambda)>0$ such that for any $\epsilon \in (0,\epsilon_0)$ we have
\begin{align*}
\epsilon^{2/3}\leq C \inf\limits_{\chi \in BV(\Omega;\{A,B\})} \inf\limits_{u \in \mathcal{D}_{F_{\lambda}}} E_\epsilon(u,\chi).
\end{align*}
\end{prop}

In order to deal with the compatible case, we invoke the following (slightly generalized) auxiliary results from \cite{Ruland22}, see also \cite{KW16, KKO13}, which we formulate for a general Fourier multiplier \(m\):

\begin{lem}[Elastic, surface and low frequency cut-off]
\label{lem:aux}
Let \(d,n,N\in \N\).
Let \(\Omega \subset \R^d\) be a bounded Lipschitz domain.
Let \(m: \R^d \to \R^{N}\) be a linear map and denote by \(V := \ker m\subsetneq \R^d\) its kernel and by \(\Pi_V: \R^d \to V\) the orthogonal projection onto \(V\). Let \(f \in BV(\R^d;\{-\lambda,0,1-\lambda\})\) for \(\lambda \in (0,1)\) with \(f = 0\) outside \(\Omega\) and \(f \in \{-\lambda,1-\lambda\}\) in \(\Omega\).
Consider the elastic and surface energies given by
\begin{align*}
 \tilde{E}_{el}(f) := \int_{\R^d} |m(\frac{\xi}{|\xi|}) \hat{f}(\xi)|^2 d\xi, \ \tilde{E}_{surf}(f) := \int_{\Omega} |\nabla f|.
\end{align*}
Then the following results hold:
\begin{itemize}
\item[(a)] \emph{Low frequency elastic energy control.} Let \(\mu > 1\), then there exists \(C=C(m,\Omega)>0\) with
\begin{align*}
\|\hat{f}\|_{L^2(\{\xi \in \R^d: \ |\Pi_{V}(\xi)|\leq \mu\})}^2 \leq C \mu^{2} \tilde{E}_{el}(f).
\end{align*}
\item[(b)] \emph{High frequency surface energy control.} There exists \(C=C(d,\lambda)>0\) such that for \(\mu > 0\) it holds
\begin{align*}
\|\hat{f}\|_{L^2(\{\xi \in \R^d: |\xi|\geq \mu\})}^2 \leq C \mu^{-1}( \tilde{E}_{surf}(f)  + \Per(\Omega)).
\end{align*}
\end{itemize}
\end{lem}

\begin{proof}[Proof of \cref{lem:aux}]
Since the property (b) is directly analogous to the one from \cite[Lemma 2]{Ruland22}, we only discuss the proof of (a) which requires some (slight) modifications with respect to \cite{Ruland22} (and \cite{KW16}). We thus present the argument for this for completeness.
We split \(\xi = \xi' + \xi'' \), where \( \xi' \in V^\perp \neq \{0\}, \xi'' \in V = \ker m\).
With this in hand and by the linearity of $m$ it holds
\begin{align*}
 m(\frac{\xi}{|\xi|}) = m(\frac{\xi'+\xi''}{|\xi|}) = m(\frac{\xi'}{|\xi|}) = M \frac{\xi'}{|\xi|}
\end{align*}
for some matrix representation \(M\) of \(m\).
Hence, using \(\xi' \perp \ker m\), there is \(c = c(m) >0\) such that
\begin{align*}
 \tilde{E}_{el}(f) = \int_{\R^d} |m(\frac{\xi}{|\xi|}) \hat{f}(\xi)|^2 d\xi \geq c \int_{\R^d} | \frac{\xi'}{|\xi|} \hat{f}(\xi)|^2 d\xi.
\end{align*}

With this in hand, we argue similarly as in \cite{KW16} and \cite{Ruland22}: For $a,b\in \R$, $a,b\geq 0$ and the orthogonal splitting \(\xi = \xi' + \xi''\) from above, it holds that
\begin{align}
\label{eq:lower_est}
\begin{split}
\frac{1}{c} \tilde{E}_{el}(f)
& \geq \int\limits_{\R^d}\left|\frac{\xi'}{|\xi|} \hat{f}\right|^2 d\xi = \int\limits_{\R^d} \frac{|\xi'|^2}{|\xi'|^2+|\xi''|^2} |\hat{f}|^2 d\xi \\
& \geq \frac{1}{\frac{a^2}{b^2} + 1} \int\limits_{\{ |\xi'|\geq \frac{1}{a}, \ |\xi^{''}|\leq \frac{1}{b} \}} |\hat{f}|^2 d\xi\\
& = \frac{1}{\frac{a^2}{b^2} + 1} \int\limits_{ \{ |\xi''|\leq \frac{1}{b} \}} \left(\int\limits_{V^\perp}|\hat{f}|^2 d\xi'- \int\limits_{\{ |\xi'| < \frac{1}{a} \}} |\hat{f}|^2 d\xi' \right) d\xi''\\
& \geq \frac{1}{\frac{a^2}{b^2} + 1} \int\limits_{ \{ |\xi''|\leq \frac{1}{b} \}} \left(\, \int\limits_{V^\perp}|\hat{f}|^2 d\xi'- \left( \frac{2}{a} \right)^{\dim V} \sup\limits_{\xi' \in V^\perp} |\hat{f}|^2  \right) d\xi''.
\end{split}
\end{align}
Using the notation \(f(\xi) = f(\xi',\xi'')\), setting
\begin{align*}
a^{\dim V}:= 2^{\dim V+1} \sup\limits_{\xi'' \in V} \frac{\sup\limits_{\xi' \in V^\perp} |\hat{f}(\xi', \xi'')|^2 }{\int\limits_{V^\perp} |\hat{f}(\xi', \xi'')|^2d\xi'},
\end{align*}
and using Plancherel's identity, the $L^{\infty}-L^1$ bounds for the Fourier transform and Hölder's inequality, we obtain that
\begin{align*}
a^{\dim V} \leq  2^{\dim V+1}\sup\limits_{\xi'' \in V} \frac{\|\F_{\xi''}f(\cdot, \xi'')\|_{L^1(V^\perp)}^2 }{\|\F_{\xi''}{f}(\cdot, \xi'')\|_{L^2(V^\perp)}^2} \leq C(\Omega) 2^{\dim V+1}.
\end{align*}

In particular, the constant $a$ is well-defined. Returning to \eqref{eq:lower_est}, we consequently deduce that for \(b \in (0, 1)\)
\begin{align*}
\tilde{E}_{el}(f) \geq C(\Omega,m) b^2 \int\limits_{ \{\xi : |\xi''|\leq \frac{1}{b} \}} |\hat{f}|^2 d\xi.
\end{align*}
Choosing $b= \mu^{-1} < 1$ and noting that $\{\xi \in \R^d: |\Pi_{V}(\xi)|\leq \mu\} = \{\xi \in \R^d: \ |\xi''| \leq \frac{1}{b}\}$ implies the claim.
\end{proof}

With \cref{lem:aux} in hand, we turn to the proof of the lower bound in  \cref{prop:compatible_2well_A}:

\begin{proof}[Proof of the lower bound in \cref{prop:compatible_2well_A}]
Since $B-A \in \Lambda_{\mathcal{A}}\setminus I_{\mathcal{A}}$, there exists $\xi \in \R^d\setminus \{0\}$ such that $\AA(\xi)(B-A)=0$. As $\mathcal{A}(D)$ is a first order operator, we have that $\AA(\xi)$ is linear in $\xi \in \R^d$. In particular, we have that the set, cf. \eqref{eq:laminatDir},
\begin{align*}
V_{B-A}:= V_{\A,B-A} = \{\xi \in \R^d: \ \AA(\xi)(B-A) = 0\} \neq \{0\}
\end{align*}
is a linear space.
Rewriting $\R^d \ni \xi = \xi' + \xi'' $ with $\xi' \in V_{B-A}^\perp$ and $\xi'' \in V_{B-A}$ as in the proof of \cref{lem:aux}, then \cref{cor:ElasticEnergyTwoWell} implies that
\begin{align} \label{eq:BoundEnergyMultiplier}
\begin{split}
 E_{el}(\chi; F_{\lambda}) & \geq C \int\limits_{\R^d}|((1-\lambda ) \hat{\chi}_{A} - \lambda \hat{\chi}_B)\AA(\frac{\xi}{|\xi|})(A-B)|^2 d\xi, \\
 E_{surf}(\chi) & = \int_{\Omega} |\nabla(\chi-F_\lambda)| = \int_\Omega |\nabla ((1-\lambda)\chi_A - \lambda \chi_B)(A-B)| dx \geq C \int\limits_{\Omega} |\nabla [(1-\lambda)\chi_A - \lambda \chi_B]|,
\end{split}
\end{align}
for a constant \(C\) depending on the operator \(\A(D)\) and on $A-B$.

Now, setting $m(\xi):= \AA(\xi)(A-B)$ and $f:=(1-\lambda ) \chi_{A} - \lambda \chi_B$, yields the applicability of \cref{lem:aux} with $V = V_{B-A}\subsetneq\R^d$. This is the only place where we use the assumption that $A-B\notin I_\A$. We deduce that by \eqref{eq:BoundEnergyMultiplier} and the decomposition of \(\R^d\) into the two regions from \cref{lem:aux} we have for \( \mu > 1\)
\begin{align*}
\|f\|_{L^2}^2 &\leq  \|\chi_{\{|\xi|\geq \mu\}}(D) f \|_{L^2}^2
+ \|\chi_{\{|\xi''|\leq \mu\}}(D)f\|_{L^2}^2\\
& \leq C \Big(\mu^{2}E_{el}(u,\chi) + (\mu^{-1} \epsilon^{-1}) \epsilon E_{surf}(\chi) + \mu^{-1} \Per(\Omega)\big) \Big),
\end{align*}
where the constant \(C\) depends on \(\A(D),A,B,\Omega, d, \lambda\).
Now choosing $\mu = \epsilon^{-\frac{1}{3}}>1$, noting that then $\mu^{-1}\epsilon^{-1} = \epsilon^{-\frac{2}{3}} $ for \(\epsilon < 1\), we obtain
\begin{align*}
\|f\|_{L^2}^2 \leq C \big(\epsilon^{-\frac{2}{3}} E_\epsilon(\chi;F_\lambda) + \epsilon^{\frac{1}{3}}\Per(\Omega)\big),
\end{align*}
where $E_\epsilon(\chi;F_\lambda):= E_{el}(\chi;F_{\lambda}) + \epsilon E_{surf}(\chi)$.
Using the lower bound
\begin{align*}
 \Vert f \Vert_{L^2}^2 = \int_{\R^d} |(1-\lambda)\hat{\chi}_A + \lambda \hat{\chi}_B|^2 d\xi = \int_{\R^d} |(1-\lambda) \chi_A + \lambda \chi_B|^2 dx \geq (\min\{\lambda,1-\lambda\})^2 |\Omega|,
\end{align*}
then implies that
\begin{align*}
\epsilon^{\frac{2}{3}}\leq C (E_{\epsilon}(\chi;F_{\lambda})+\epsilon \Per(\Omega)).
\end{align*}
Finally, for $\epsilon \in (0,\epsilon_0)$ and $\epsilon_0 = \epsilon_0(\A(D),A,B,\Omega, d, \lambda)>0$ sufficiently small, the perimeter contribution on the right hand side can be absorbed into the left hand side, which yields the desired result.
\end{proof}

\subsection{The super-compatible setting: Proof of \cref{thm:TwoWell}(iii) and reduction to cocanceling operators}
\label{sec:super}
In this subsection we will show that if we are in the setting in which the estimates of the previous subsection degenerate, i.e., $\mathbb A(\xi)(A-B)=0$ for all $\xi\in\R^d$, then in fact there can be no non-trivial bound from below. 
We will however also show that, in general for pairwise not super-compatible wells, it is possible to reduce to an equivalent minimization problem in the setting of cocanceling operators for suitably modified boundary data.

\begin{proof}[Proof of the super-compatible case in \cref{thm:TwoWell}]
It suffices to give an upper bound construction with zero total energy. To this end, we consider  $\chi=A\chi_\Omega $ and $u=A \chi_{\Omega} + F_{\lambda} \chi_{\R^d \setminus \overline{\Omega}}$ and observe that
\begin{align*}
\mathcal{A}(D)u = \mathcal{A}(D)(u-B) = \mathcal{A}(D)[(\chi_{\Omega} + \lambda \chi_{\R^d \setminus \overline{\Omega}})(A-B)]=0.
\end{align*} 
As a consequence, $u$ is admissible in the definition of the elastic energy and the elastic energy vanishes. Moreover, since $\chi = A$ in $\Omega$ we also have the vanishing of the surface energy. This concludes the argument.
\end{proof}

We will show that for two not super-compatible wells, we can always assume that $I_\A=\{0\}$, in which case we work in the class of \emph{cocanceling} operators introduced by Van Schaftingen in \cite{VS13}.

\begin{prop}
\label{prop:reduction_cocancelling}
Let $n,d,k \in \N$, let $\Omega \subset \R^d$ be a bounded Lipschitz domain, $\A(D)$ a differential operator of order $k$ as in \eqref{eq:operator} with $I_\A$ given in \eqref{eq:I_A}. For $\chi \in L^2(\Omega; \mathcal{K})$ for some compact set of states $\mathcal{K} \subset \R^n$ let $E_{el}(\chi; F)$ be as in \eqref{eq:E_el} for $F \in \mathcal{K}^{qc}$, cf. \cref{sec:not}.  Then for the restricted operator $\tilde{\A}(D): C^\infty(\R^d;I_{\A}^{\perp}) \to C^\infty(\R^d;\R^m)$ there exists a function $\chi_{\perp} \in L^2(\Omega, \Pi_{I_{\A}^{\perp}}\mathcal K)$ such that
\begin{align*}
 E_{el}(\chi;F) = E_{el}^{\tilde{\A}}(\chi_{\perp};F_\perp) = \inf_{u_{\perp} \in \mathcal{D}^{\tilde{A}}_{F_\perp}} \int_{\Omega} |u_\perp - \chi_\perp|^2 dx.
 \end{align*}
 Here we denote the orthogonal projection of $F$ onto $I_{\A}^\perp$ by $F_\perp$.
\end{prop}

\begin{proof}
 
    We use the orthogonal decomposition $\R^n = I_{\A}^\perp \oplus I_{\A}$ to write
    \begin{align*}
     u = u_{\perp} + u_{I}, \quad \chi = \chi_{\perp} + \chi_{I},
    \end{align*}
with $u_{\perp} : \R^d \to I_{\A}^\perp, u_{I}: \R^d \to I_{\A}$, $\chi_{\perp} : \Omega \to I_{\A}^\perp$, $\chi_{I}: \Omega \to I_{\A}$.
    By orthogonality we can also split the elastic energy
    \begin{align*}
     E_{el}(u,\chi) = \int_{\Omega} |u-\chi|^2 dx = \int_{\Omega} |u_{\perp} - \chi_{\perp} |^2 dx + \int_{\Omega} |u_I - \chi_I|^2 dx.
    \end{align*}

    Defining the restricted operator $\tilde{\A}(D): C^\infty(\R^d;I_{\A}^\perp) \to C^\infty(\R^d;\R^m), \tilde{\A}(D)u := \A(D)u$ and the restricted space of admissible functions as in \eqref{eq:admissible}
    \begin{align*}
     \mathcal{D}_{F_\perp}^{\tilde{A}} := \{u_{\perp} \in L^2_{loc}(\R^d;I_{\A}^\perp) : \tilde{\mathcal{A}}(D)u_{\perp} = 0 \text{ in } \R^d, u_{\perp} = F_{\perp} \text{ in } \R^d \setminus \bar\Omega\},
    \end{align*}
    we see that for $u \in \mathcal{D}_F$ it holds $u_\perp \in \mathcal{D}^{\tilde{\A}}_{F_\perp}$ with $F_\perp = \Pi_{I_\A^{\perp}} F$ and $u_{I} = F - F_{\perp}$ outside $\Omega$.

    Thus, after minimizing the elastic energy in $u$, it holds
    \begin{align*}
     \inf_{u \in \mathcal{D}_F} E_{el}(u,\chi) = \inf_{u_{\perp} \in \mathcal{D}^{\tilde{A}}_{F_\perp}} \inf_{u_{I} \in L^2_{loc}(\R^d;I_{\mathcal{A}}), u_I = F-F_\perp \text{ in } \bar\Omega^c} \int_{\Omega} |u_{\perp} - \chi_{\perp} |^2 dx + \int_{\Omega} |u_I - \chi_I|^2 dx.
    \end{align*}

    As we have seen in the proof of \cref{thm:TwoWell}(iii), the second term involving $u_I$ vanishes and hence,
    \begin{align*}
     E_{el}(\chi;F) = E_{el}^{\tilde{\A}}(\chi_\perp;F_{\perp}).
    \end{align*}

    This reduces the elastic energy to the case of a cocanceling operator as indeed $I_{\tilde{A}} = \{0\}$.
    \end{proof}
As the surface energy does not depend on the operator $\A(D)$, this result yields:
\begin{align*}
 E_{\epsilon}(\chi; F) = E_{el}^{\tilde{A}}(\chi_{\perp};F_{\perp}) + \epsilon E_{surf}(\chi).
\end{align*}

As a corollary, we apply this to the $N$-well problem:

\begin{cor}[Finitely many, pairwise not super-compatible wells]
\label{cor:two-well_reduction}
Under the same assumptions as in \cref{prop:reduction_cocancelling}, with $\chi \in BV(\Omega;\mathcal{K})$, for the special case that $ \mathcal{K}:=\{B_1,\dots,B_N \}\subset \R^n$ for $N \in \N$ such that $B_j$, $j\in \{1,\dots,N\}$, are pairwise not super-compatible, i.e. $B_i - B_j \notin I_\A$ for $i \neq j$, there exists a constant $C=C(B_1,\dots,B_n)>1$ such that
\begin{align*}
C^{-1} E_{surf}(\chi_{\perp}) \leq E_{surf}(\chi) \leq C E_{surf}(\chi_{\perp}).
\end{align*}
In particular, it hence holds that
\begin{align*}
 E_{\epsilon}(\chi;F) \sim E_{\epsilon}^{\tilde{\A}}(\chi_\perp;F_\perp).
\end{align*}
\end{cor}

\begin{proof}
Writing $\chi= \sum_{j=1}^N B_j \chi_{\Omega_j}$ with $\chi_{\Omega_j} \in BV(\Omega;\{0,1\}), \sum_{j=1}^N \chi_{\Omega_j} = 1$ in $\Omega$, we can calculate
\begin{align*}
 | \nabla \chi| = \sum_{i<j} |B_i - B_j| \mathcal{H}^{d-1}(\partial^\ast \Omega_i \cap \partial^\ast \Omega_j), \  | \nabla \chi_\perp| = \sum_{i<j} |(B_i - B_j)_\perp| \mathcal{H}^{d-1}(\partial^\ast \Omega_i \cap \partial^\ast \Omega_j),
\end{align*}
where we denote the reduced boundary of a set $E$ with finite perimeter by $\partial^\ast E$ and used the notation from above for $B \in \R^n$ to write $B_\perp = \Pi_{I_\A^\perp} B$.

By assumption, for $i < j$ it holds $B_i - B_j \notin I_\A$, and therefore also the projection satisfies $(B_i - B_j)_\perp \neq 0$.
This implies $|(B_i - B_j)_\perp| > 0$ for all tupels $(i,j)$ such that $i<j$ and hence there are constants $0 < c < \frac{|B_i-B_j|}{|(B_i-B_j)_\perp|} < C$ such that
\begin{align*}
 0 < c |\nabla \chi_\perp| \leq |\nabla \chi| \leq C |\nabla \chi_\perp|.
\end{align*}
This together with \cref{prop:reduction_cocancelling} concludes the proof.
\end{proof}

We emphasize that \cref{cor:two-well_reduction} in particular holds in the context of \cref{thm:TwoWell}. Therefore in the statement of \cref{thm:TwoWell}(ii) we can assume without loss of generality that $I_\A = \{0\}$. In fact, note that in the crucial bound of \cref{cor:ElasticEnergyTwoWell} we have $$\AA(\xi)(A-B)=\AA(\xi)(A_\perp-B_\perp)=\tilde\AA(\xi)(A_\perp-B_\perp)$$
for $|\xi|=1$.

\subsection{Some remarks on the compatible two-state problem for higher order operators}
\label{sec:higher}
We conclude our discussion of lower scaling bounds by commenting on the case of the compatible two-well problem for higher order operators. Here the situation is still less transparent, yet some remarks are possible.

Indeed, on the one hand, it is known that, in general, for operators of order $k\geq 2$ the two-well problem does \emph{not} have to scale with $\epsilon^{\frac{2}{3}}$. 
In order to illustrate this, we consider the specific operator $\A{(D)}:= \operatorname{curl} \operatorname{curl}$. This operator is the annihilator of the symmetrized gradient $e(v):= \frac{1}{2}(\nabla v + (\nabla v)^t)$.
We consider the following quantitative two-state problem (for \(d=2\))
\begin{align}
\label{eq:sym}
\mathcal{E}_{\epsilon}(v,\chi):= \mathcal{E}_{el}(v,\chi) + \epsilon \mathcal{E}_{surf}(\chi):=\int\limits_{[0,1]^2} \left|e(v)-\begin{pmatrix} 1 & 0 \\ 0 & 1 + \alpha(1-2\chi) \end{pmatrix} \right|^2 dx + \epsilon \int\limits_{[0,1]^2}|\nabla \chi|
\end{align} 
with $e(v) \in \mathcal{D}^{\A}_F, \chi \in BV([0,1]^2;\{0,1\})$, $\alpha \in (0,1)$ and study the corresponding minimization problem with prescribed boundary data $F:= \begin{pmatrix} 1 & 0 \\ 0 & 1 \end{pmatrix}$.

\begin{prop}
\label{prop:symm_grad}
Let $\mathcal{E}_{\epsilon}(v,\chi)$ and $F$ be as in \eqref{eq:sym}. Then there exists $\epsilon_0= \epsilon_0(\alpha)>0$ such that for $\epsilon \in (0,\epsilon_0)$ it holds that
\begin{align*}
\inf\limits_{\chi \in BV([0,1]^2;\{0,1\}) }\inf\limits_{e(v)\in \mathcal{D}^{\A}_{F}}  \mathcal{E}_{\epsilon}(v,\chi) \sim \epsilon^{\frac{4}{5}}.
\end{align*}
\end{prop}

We remark that this observation is not new; indeed, a geometrically nonlinear version of this had earlier been derived in \cite[Theorem 1.2]{CC15}. As observed in \cite{CC15} the reason for the different scaling in \cref{prop:symm_grad} and \cite[Theorem 1.2]{CC15}, compared to the more standard $\epsilon^{\frac{2}{3}}$ behaviour from \cref{thm:TwoWell}, consists of the higher degeneracy of the multiplier associated with the energy which is manifested in the presence of only one possible normal in the (symmetrized) rank-one condition. For convenience of the reader and in order to illustrate the robustness of the above approach within geometrically linear theories, we present an alternative short proof (of the lower bound) of \cref{prop:symm_grad} based on our Fourier theoretic framework. We note that in the geometrically linear setting this provides an alternative to the approach from \cite{CC15} in which the lower bound for the energy is deduced by a local ``averaging'' argument, considering the energy on representative domain patches with the expected scaling behaviour.

\begin{proof}
\emph{Step 1: Lower bound.}
We note that the lower bound for this setting directly follows from our arguments above: Indeed, for $A -B = 2\alpha e_2 \otimes e_2$, we obtain that
\begin{align*}
\AA(\xi)(B-A) = 2\alpha \xi \times \big(\xi \times (e_2 \otimes e_2)\big)^t = 2\alpha \xi_1^2.
\end{align*}
With this in hand, an analogous argument as in \cref{lem:aux} and, in particular, in \eqref{eq:lower_est} implies that
\begin{align}
\label{eq:aux_a}
\|\hat{\chi}\|_{L^2(\{\xi \in \R^2: \ |k_2|\leq \mu\})}^2 \leq C \mu^{4} \mathcal{E}_{el}(\chi;F),
\end{align}
where we have used that in this situation the multiplier is given by $m(\xi) = \AA(\xi)(B-A) \sim \xi_1^2$. The different exponent of $\mu$ in \eqref{eq:aux_a} (compared to the one from \cref{lem:aux}(a)) is a consequence of the degeneracy of the symbol $m(\xi)$ and the higher order of the operator $\operatorname{curl} \operatorname{curl}$ (or put, more concretely, the quadratic dependence $\xi_1^2$). 
Hence, replacing the bound from \cref{lem:aux}(a) by the one from \eqref{eq:aux_a} and carrying out the splitting as in the proof of \cref{thm:TwoWell}, we obtain the following optimization problem: For $f:=1 + \alpha(1-2\chi)$
\begin{align*}
(1-\alpha)^2 \leq \|f\|_{L^2}^2 &\leq  \|\chi_{\{|\xi|\geq \mu\}}(D) f \|_{L^2}^2
+ \|\chi_{\{|\xi''|\leq \mu\}}(D)f\|_{L^2}^2\\
& \leq C \big(\mu^{4}\mathcal{E}_{el}(u,\chi) + (\mu^{-1} \epsilon^{-1}) \epsilon \mathcal{E}_{surf}(\chi) + \mu^{-1} \Per(\Omega)\big).
\end{align*}
Choosing $\mu \sim \epsilon^{-\frac{1}{5}}$ and rearranging the estimates then imply the claim.\\

\emph{Step 2: Upper bound.}
The associated improved upper bound makes use of the vectorial structure of the problem in contrast to the essentially scalar ``standard $\epsilon^{2/3}$ construction'' (see the arguments below). Recalling that the nonlinear construction from the proof of \cite[Lemma 2.1]{CC15} also yields a construction with the desired scaling for the geometrically linearized problem, we do not carry out the details of this but refer to \cite[Lemma 2.1]{CC15} for these.
\end{proof}

On the other hand, the arguments from \cite{D13,CC15,CO09,CO12} show that still for $\A(D) = \operatorname{curl} \operatorname{curl}$ if $A -B = \gamma( e_1 \otimes e_2 + e_2 \otimes e_1)$ for $\gamma \in \R\setminus \{0\}$, then one recovers the $\epsilon^{2/3}$ scaling for the symmetrized gradient differential inclusion. In this case, the symbol reads $m(\xi) = \AA(\xi)(B-A) \sim \xi_1 \xi_2$ and the operator is ``less degenerate''.

We expect that the scaling behaviour of general higher order operators is in many interesting settings directly linked to the degeneracy of the symbol $\AA(\xi)(B-A)$. We plan to explore this in future work.

\section{Quantitative Rigidity of the $T_3$ Structure from \eqref{eq:problem}, \eqref{eq:problem1} for $\mathcal{A}(D)=\di$} \label{sec:T3}
In this section, we consider the \(T_3\) structure for the divergence operator introduced in \eqref{eq:problem}, \eqref{eq:problem1}.
The upper bound construction is given by an approximate solution of the type described in the introduction.
The lower bound is motivated by the rigidity of exact solutions as outlined in \cref{sec:stress_free}.

\subsection{The upper bound construction -- an infinite order laminate}
\label{sec:upper}

To begin with, we construct an infinite order laminate similar to the one for the Tartar square (cf. \cite{W97,C99,Ruland22} for quantitative versions of this). This is based on \cite{Garroni04} and will yield the upper bound estimate from \cref{thm:scaling_T3}.
We recall that for the divergence operator, instead of requiring rank-one connectedness for the existence of a laminate as for the curl, in our three-dimensional set-up we need rank-one or rank-two connectedness as can be seen from the wave cone for the divergence operator, cf. \eqref{eq:wave_cone}.
As a consequence, for two matrices \(A,B \in \R^{3 \times 3}\) such that \(\rank(B-A) \leq 2 \) there exists a piecewise constant map \(u: \R^3 \to \R^{3 \times 3}\) such that \( u \in \{A,B\}\) a.e. in \(\Omega\) and \(\di u = 0 \).
The lamination can be done in any direction of the kernel \(\ker(B-A) \neq \{0\}\).

Considering now the matrices $A_1,A_2,A_3$ given in \eqref{eq:problem1}, we observe that $\rank(A_i - A_j) = 3 $ for $i \neq j$.
Following \cite{Garroni04}, we introduce auxiliary matrices $S_1,S_2,S_3 \in \mathbb{R}^{3 \times 3}$.
\begin{align} \label{eq:AuxMat}
 S_1 = \begin{pmatrix}
        0 & 0 & 0 \\
        0 & \frac{2}{3} & 0 \\
        0 & 0 & 2
       \end{pmatrix}, \
S_2 = \begin{pmatrix}
       \frac{1}{2} & 0 & 0 \\
       0 & \frac{2}{3} & 0 \\
       0 & 0 & 1
      \end{pmatrix}, \
S_3 = \begin{pmatrix}
       0 & 0 & 0 \\
       0 & \frac{1}{3} & 0 \\
       0 & 0 & 1
      \end{pmatrix}.
\end{align}
It then holds for $i=1,2,3$ that
\begin{align*}
 \ker(S_i - A_i) = \vspan(e_i),\ S_{i} = \frac{1}{2} (A_{i+1} + S_{i+1}),
\end{align*}
where $A_4 = A_1, S_4 = S_1$.
As proved in \cite[Theorem 2]{PS09}, the $\mathcal{A}$-quasi-convex hull $\{A_1,A_2,A_3\}^{qc}$ can be explicitly characterized as the convex hull of the matrices $S_1,S_2,S_3$ together with the ``legs'' given by the line segments $A_j S_j$ for $j\in \{1,2,3\}$.

\begin{figure}
 \centering
 \begin{tikzpicture}[scale=0.5,x  = {(-45:3cm)}, y  = {(3cm,0cm)}, z  = {(0cm,3cm)}]
\coordinate[label=left:$A_1$] (A1) at (0,0,0);
\coordinate[label=left:$A_2$] (A2) at (-1/2,2/3,3);
\coordinate[label=below:$A_3$] (A3) at (1,1,1);
\coordinate[label=right:$S_1$] (S1) at (0,2/3,2);
\coordinate[label=below:$S_2$] (S2) at (1/2,2/3,1);
\coordinate[label=left:$S_3$] (S3) at (0,1/3,1);
\coordinate (M) at (-44/147,353/1176,1763/1176);

\draw[dashed] (A1) -- (S1);
\draw[dashed] (A2) -- (S2);
\draw[dashed] (A3) -- (S3);

\coordinate (M12) at (-29/588,-13/392,-1/1176); 
\coordinate (M13) at (235/294,235/1176,-1175/1176); 
\coordinate (M23) at (323/588,209/392,589/1176); 
\draw (M) --++ (M12) --++ (M12) --++ (M12)--++ (M12) --++ (M12);
\draw (M) --++ (M13) --++ (235/1176,235/4704,-1175/4704);
\draw (M) --++ (M23) --++ (323/1176,209/783,589/2352);
 \end{tikzpicture}
\caption{The diagonal matrices $A_1,A_2,A_3,S_1,S_2,S_3$ with the dashed lines depicting the connections in the wave cone for $\mathcal{A}(D)=\di$ and the Voronoi-regions of $A_i$ shown by the lines. As shown in \cite{PS09} the set $\mathcal{K}^{qc}$ is given the inner triangle formed by $S_1 S_2 S_3$ and the ``legs'' connecting $S_j$ and $A_j$ for $j\in \{1,2,3\}$.}
\end{figure}
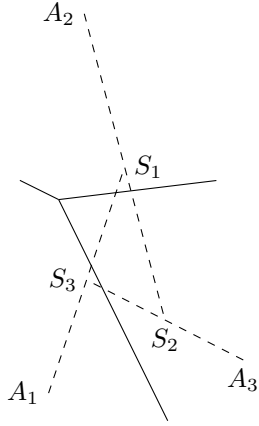

For simplicity and definiteness, we first assume, that $u = F = S_3$ outside $\Omega=[0,1]^3$.
In this setting we prove the following energy estimate:

\begin{prop}
\label{prop:energy_est_upper}
Let $\Omega = [0,1]^3$, $\mathcal{K} = \{A_1,A_2,A_3\}$ for the particular choice of matrices in \eqref{eq:problem1} and let $E_{\epsilon}$ be as in \eqref{eq:E-Total} and $\epsilon \in (0,1)$. Then for any $r\in (0,\frac{1}{4})$ with $r^{-1} \in 4\mathbb{N}$ there are sequences $u^{(m)}\in W^{ 1,\infty}(\mathbb{R}^3, \mathbb{R}^{3\times 3})$, such that $ \di u^{(m)} = 0$, $u^{(m)} = F := S_3$ outside $[0,1]^3$, and $\chi^{(m)} \in BV(\mathbb{R}^3, \mathcal{K})$, and a constant $C = C(F)>1$
with
 \begin{align*}
  E_{\epsilon}(u^{(m)},\chi^{(m)}) \leq C \left( 2^{-m} + \sum_{k=1}^m 2^{-k} r + r + \epsilon \frac{1}{r^m} \right).
 \end{align*}
\end{prop}

In order to achieve this, in the next subsections, we iteratively construct a higher and higher order laminate (depending on $\epsilon>0$). As in the setting of the Tartar square, we keep track of the surface and elastic energy contributions which arise in this process.

\subsection{Proof of the upper bound from \cref{thm:scaling_T3}}
\label{sec:lower}

We split the proof of the upper bound from \cref{thm:scaling_T3} into several steps which we will carry out in the next sections and then combine in \cref{sec:T3UpperCombiningEstimates}.

First, we start by a simple lamination of $A_1$ and $S_1$ to obtain regions in which $u \in \mathcal{K}$ holds and to satisfy the exterior data condition $u = F = S_3 = \frac{1}{2}(A_1+S_1)$ outside $\Omega$.
This is followed by a similar construction replacing $S_1$ by a lamination of $A_2$ and $S_2$ achieving a second order laminate.
Iterating the procedure of replacing $S_j$ by a lamination of $A_{j+1}$ and $S_{j+1}$ (with the convention that $A_4 = A_1, S_4 = S_1$) yields \cref{prop:energy_est_upper}.
Finally, we optimize the parameter $r$ and the number of iterations depending on $\epsilon$ in order to show the desired upper bound estimate in \cref{prop:upper}.

As we will use a potential for the laminates, we will define a ``profile-function'' once in a more general form and will then refer to this in our construction for the higher order laminates.
\begin{lem}\label{lem:LaminateProfile}
Let $R = [a_1,b_1] \times [a_2,b_2] \times [a_3,b_3] \subset \mathbb{R}^3$ be an axis-parallel cuboid, then for any direction $e_j$ with $j \in \{1,2,3\}$ and any scale $r>0$ such that $\frac{b_j-a_j}{r} \in \mathbb{N}$ there is a continuous function $f_j(\,\cdot\,;R,r):R \to \mathbb{R}^{3\times3}$ satisfying the following properties:
\begin{itemize}
 \item The function $f_j(\,\cdot\,;R,r)$ only depends on the $j$-th coordinate $x_j$ and is $r$-periodic.
 \item It holds $f_j(x;R,r) = 0$ if $x \in R$ is such that $x = (x_1,x_2,x_3)$ lies in one of the planes characterized by $x_j \in a_j + r\mathbb{N}_0$.
 \item The matrix-valued function $\operatorname{curl} f_j(x;R,r)$ only attains the values $\pm \frac{S_j-A_j}{2}$, where $A_j,S_j$ is given as in \eqref{eq:problem1} and in \eqref{eq:AuxMat}, respectively. Furthermore, the volumes of the level sets are equal, i.e. $|\{x \in R:f_j(x;R,r) = S_j\}| = |\{x \in R: f_j(x;R,r)=A_j\}| = \frac{1}{2}|R|$.
\end{itemize}
\end{lem}
\begin{proof}
Without loss of generality by a translation, we may assume that $R = [0,b_1] \times [0,b_2] \times [0,b_3]$ for $b_1,b_2,b_3 >0$.
 We consider the continuous one-periodic extension of the function
 \begin{align*}
  h:[0,1) \to \mathbb{R}, \ h(t) := \begin{cases} \frac{1}{2}t & t \in [0,\frac{1}{2}), \\
                                    \frac{1}{2}(1-t) & t \in [\frac{1}{2},1).
                                   \end{cases}
 \end{align*}
Furthermore, we define the matrices
\begin{align*}
 M_1 := \begin{pmatrix} 0 & 0 & 0 \\ 0 & 0 & -\frac{2}{3} \\ 0 & 2 & 0 \end{pmatrix}, \ 
 M_2 := \begin{pmatrix} 0 & 0 & 1 \\ 0 & 0 & 0 \\ 2 & 0 & 0 \end{pmatrix}, \
 M_3 := \begin{pmatrix} 0 & 1 & 0 \\ -\frac{2}{3} & 0 & 0 \\ 0 & 0 & 0 \end{pmatrix},
\end{align*}
satisfying $e_j \times M_j = S_j - A_j$.
With this at hand we define $f_j(\,\cdot\,;R,r): R \to \mathbb{R}^{3\times3}$:
\begin{align*}
 f_j(x_1,x_2,x_3; R,r) := r h(\frac{x_j}{r}) M_j. 
\end{align*} 
It follows directly, that $f_j(\,\cdot\,;R,r)$ is continuous, only depends on $x_j$, is $r$-periodic, and vanishes for $x_j \in r\mathbb{N}_0$.
Lastly, we note that $\operatorname{curl} f_j(x;R,r) = h'(\frac{x_j}{r}) e_j \times M_j \in \{\pm \frac{S_j-A_j}{2}\}$ and that indeed $|\{x \in R:f_j(x;R,r) = S_j\}| = |\{x \in R: f_j(x;R,r)=A_j\}| = \frac{1}{2}|R|$.
\end{proof}

\subsubsection{First order laminates}

We use a potential $v: \mathbb{R}^3 \to \mathbb{R}^{3 \times 3}$ to construct our laminates attaining the prescribed exterior data, i.e. we consider the row-wise curl: $u = \operatorname{curl} v$.

As $S_3 = \frac{1}{2} A_1 + \frac{1}{2} S_1 $ and $(S_1 - A_1) e_1 = 0$ the first order lamination is in the $e_1$-direction.

We seek to use \cref{lem:LaminateProfile} to construct $v$, but have to adapt the boundary condition.
For this we define a mapping $\tilde{S}_3: \mathbb{R}^3 \to \mathbb{R}^{3 \times 3}$ with $\operatorname{curl} \tilde{S}_3 = S_3 = F$. A possible choice for $\tilde{S}_3$ is given by the following matrix-valued function
\begin{align*}
\tilde{S}_3(x) := \begin{pmatrix}
  0 & 0 & 0 \\
  0 & 0 & - \frac{1}{3} x_1 \\
  0 & x_1 & 0
 \end{pmatrix}.
\end{align*}

Furthermore, we also define the cut-off function
\begin{align*}
 \phi : \mathbb{R} \to [0,1], \ \phi(t) = \begin{cases}
         0 & t < \frac{1}{8}, \\
         4t - \frac{1}{2} & t \in [\frac{1}{8},\frac{3}{8}], \\
         1 & t > \frac{3}{8},
        \end{cases}
\end{align*}

With this, we define the (continuous) potential for $r \in (0,\frac{1}{4}), r^{-1} \in 4 \mathbb{N}$ by using \cref{lem:LaminateProfile} for $j=1, R=[0,1]^3$:
\begin{align*}
 v^{(1)} & :\Omega \to \mathbb{R}^{3\times 3}, \\
 v^{(1)}(x) & = \phi(\frac{1}{r} d_{\partial\Omega}(x)) f_1(x;\Omega,r) + \tilde{S}_3(x),
\end{align*}
where $d_{\partial \Omega}(x)$ denotes a smoothed-out distance function to the boundary $\partial \Omega$.
Without change of notation, we consider the (continuous) extension of $v^{(1)}$ to $\mathbb{R}^3$ by $\tilde{S}_3(x)$, which is possible, as $v^{(1)}(x) = \tilde{S}_3(x)$ on $\partial\Omega$.

We then set
\begin{align*}
 u^{(1)} := \operatorname{curl} v^{(1)},
\end{align*}
and note that in $\Omega$ it holds
\begin{align*}
 u^{(1)}(x) & = \operatorname{curl} v^{(1)}(x) \\
 & = \phi'(\frac{d_{\partial\Omega}(x)}{r}) h(\frac{x_1}{r}) \nabla d_{\partial\Omega}(x) \times M_1 + \phi(\frac{d_{\partial\Omega}(x)}{r})  \operatorname{curl} f_1(x;\Omega,r) + F \\
 & = \phi'(\frac{d_{\partial\Omega}(x)}{r}) h(\frac{x_1}{r}) \nabla d_{\partial\Omega}(x) \times M_1 + \phi(\frac{d_{\partial\Omega}(x)}{r}) h'(\frac{x_1}{r}) (S_1-A_1) + \frac{1}{2}(S_1+A_1).
\end{align*}

With these considerations, we have obtained the following properties:
It holds $\di u^{(1)} = \di (\operatorname{curl} v^{(1)}) = 0$ in $\mathbb{R}^3$,
$u^{(1)}(x) = F$ in $\mathbb{R}^3\setminus\overline{\Omega}$ and for $x \in \Omega \setminus d_{\partial\Omega}^{-1}([0,\frac{3}{8}r])$ we have $u^{(1)} \in \{A_1,S_1\}$. In other words, our deformation $u^{(1)}$ is a divergence-free function satisfying the desired boundary conditions and which, outside of the cut-off region, is a solution to the differential inclusion $u^{(1)}\in \{A_1,S_1\}$.

With the higher order laminates in mind we rephrase this using the decomposition of $\Omega$ into the three disjoint parts consisting of the $S_1$-cells, the $A_1$-cells and the cut-off region.
To be more precise, we define
\begin{align*}
 R^{(1)} & := \{x \in \Omega: \phi(\frac{d_{\partial\Omega}(x)}{r}) = 1, u^{(1)} = S_1\},\\
 Q^{(1)} & := \{x \in \Omega: \phi(\frac{d_{\partial\Omega}(x)}{r}) = 1, u^{(1)} = A_1\}, \\
 C^{(1)} & := \{x \in \Omega: \phi(\frac{d_{\partial\Omega}(x)}{r}) < 1\}.
\end{align*}
Indeed it holds $\Omega = R^{(1)} \cup Q^{(1)} \cup C^{(1)}$ and $|C^{(1)}| = 6 \frac{3}{8}r$ and by \cref{lem:LaminateProfile} we know $|R^{(1)}| \leq \frac{1}{2} |\Omega| = \frac{1}{2}$.

\begin{figure}
 \centering
 \begin{tikzpicture}[scale = 5]


  \fill[draw=none,fill = blue!20] (0,0) -- (0.03125,0.03125) -- (0.03125,0.96875) -- (0,1) -- cycle;
  \fill[draw=none, fill = blue!20] (0,1) -- (1,1) -- (0.96875,0.96875) -- (0.03125,0.96875) -- cycle;
  \fill[draw=none, fill = blue!20] (1,1) -- (1,0) -- (0.96875,0.03125) -- (0.96875,0.96875) -- cycle;
  \fill[draw=none, fill = blue!20] (0,0) -- (1,0) -- (0.96875,0.03125) -- (0.03125,0.03125) -- cycle;
  \foreach \x in {0.125,0.375,0.625}
    \fill[draw=none, fill=blue!20] (\x,0.03125) -- (\x+0.125,0.03125) -- (\x+0.125,0.09375) -- (\x,0.09375) -- cycle;
  \foreach \x in {0.125,0.375,0.625}
    \fill[draw=none, fill=blue!20] (\x,0.96875) -- (\x+0.125,0.96875) -- (\x+0.125,0.90625) -- (\x,0.90625) -- cycle;
  \fill[draw=none, fill = blue!20] (0.875,0.03125) -- (0.96875,0.03125) -- (0.90625,0.09375) -- (0.875,0.09375) -- cycle;
  \fill[draw=none, fill = blue!20] (0.875,0.96875) -- (0.96875,0.96875) -- (0.90625,0.90625) -- (0.875,0.90625) -- cycle;
  \fill[draw=none,fill = blue!20] (0.90625,0.09375) -- (0.90625,0.90625) -- (0.96875,0.96875) -- (0.96875,0.03125) -- cycle;
  \foreach \x in {0.125,0.375,0.625}
    \fill[draw=none, fill=blue!20] (\x,0.09375) -- (\x+0.125,0.09375) -- (\x+0.125,0.90625) -- (\x,0.90625) -- cycle;
  \fill[draw=none, fill = blue!20] (0.875,0.09375) -- (0.90625,0.09375) -- (0.90625,0.90625) -- (0.875,0.90625) -- cycle;

  \fill[draw=none, fill = red!20] (0.03975183824,0.03975183824) -- (0.03975183824,0.9602481618) -- (0.09375,0.90625) -- (0.09375,0.09375) -- cycle;
  \foreach \x in {0.25,0.5,0.75}
    \fill[draw=none, fill = red!20] (\x,0.09375) -- (\x+0.125,0.09375) -- (\x+0.125,0.90625) -- (\x,0.90625) -- cycle;
  \fill[draw=none, fill = red!20] (0.09375,0.09375) -- (0.125,0.09375) -- (0.125,0.90625) -- (0.09375,0.90625) -- cycle;

  \foreach \x in {0.25,0.5,0.75}
    \fill[draw=none, fill = green!20] (\x,0.03125) -- (\x+0.125,0.03125) -- (\x+0.125,0.09375) -- (\x,0.09375) -- cycle;
  \foreach \x in {0.25,0.5,0.75}
    \fill[draw=none, fill = green!20] (\x,0.96875) -- (\x+0.125,0.96875) -- (\x+0.125,0.90625) -- (\x,0.90625) -- cycle;
  \fill[draw=none, fill = green!20] (0.03125,0.03125) -- (0.09375,0.09375) -- (0.125,0.09375) -- (0.125,0.03125) -- cycle;
  \fill[draw=none, fill = green!20] (0.03125,0.96875) -- (0.09375,0.90625) -- (0.125,0.90625) -- (0.125,0.96875) -- cycle;
  \fill[draw=none, fill = green!20] (0.03125,0.03125) -- (0.03125,0.96875) -- (0.03975183824,0.9602481618) -- (0.03975183824,0.03975183824) -- cycle;

  \draw (0,0) -- (1,0) -- (1,1) -- (0,1) -- cycle;
  \foreach \x in {0.125,0.25,0.375,0.5,0.625,0.75,0.875}
    \draw[opacity=0.5,dashed] (\x,0) -- (\x,1);

  \draw[gray] (0,0) -- (1,1);
  \draw[gray] (1,0) -- (0,1);

  \draw (0.03125,0.03125) -- (0.03125,0.96875) -- (0.96875,0.96875) -- (0.96875,0.03125) -- cycle;
  \draw (0.09375,0.09375) -- (0.09375,0.90625) -- (0.90625,0.90625) -- (0.90625,0.09375) -- cycle;
  \draw (0.03975183824,0.03975183824) -- (0.03975183824,0.9602481618);

 \end{tikzpicture}
\caption{The \(x_3 = \frac{1}{2}\) slice of the projection \(\chi^{(1)}\), blue represents \(A_1\), red \(A_2\) and green \(A_3\).}
\end{figure}
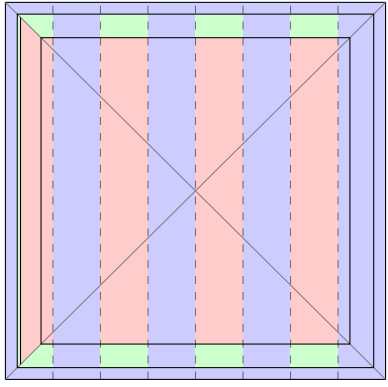

Choosing $\chi^{(1)} = \Pi_{\mathcal{K}} u^{(1)}$ as the pointwise orthogonal (with a fixed choice for the not uniquely defined points) projection of $u^{(1)}$ onto $\mathcal{K}$, up to a uniformly bounded constant, the elastic energy can be bounded by the measure of the region in which $u^{(1)}(x) = S_1$ and the cut-off region:
\begin{align*}
 E_{el}(u^{(1)},\chi^{(1)}) & = \int_\Omega |u^{(1)}-\chi^{(1)}|^2 \mathrm{d}x \leq C \Big( |R^{(1)}| + |C^{(1)}| \Big) \\
 & \leq C \Big( \frac{1}{2} |\Omega| + 6 \frac{3r}{8} \Big) \leq C \big( \frac{1}{2} + 3 r \big).
\end{align*}

Furthermore, the surface energy is bounded by counting the interfaces at which $\chi^{(1)}$ may jump. This consists of at most $\frac{2}{r}$ interfaces in the interior and at most $\frac{2}{r}+ 2 \cdot 6 \leq \frac{8}{r}$ new interfaces in the cut-off region. For $r \in (0,1)$, the surface area is thus controlled by 
\begin{align*}
 E_{surf}(\chi^{(1)}) = \int_\Omega |\nabla \chi^{(1)}| \leq C \frac{10}{r}.
\end{align*}

\subsubsection{Second order lamination}

After this first order lamination, the differential inclusion $u \in \mathcal{K}$ with $\di u = 0$ holds only in $Q^{(1)}$.
In order to further reduce the energy, we now replace each of the $\frac{1}{r}$ many cuboids in $R^{(1)}$ for which $u^{(1)} = S_1 = \frac{1}{2} A_2 + \frac{1}{2} S_2$ by a lamination in the $e_2$-direction.
For this, we modify the potential $v^{(1)}$ in these regions:

For $x \in R^{(1)}$ and for $r^2 \in (0,\frac{r}{2}), r^{-2}+\frac{3}{4r} \in \mathbb{N}$, we define with the help of \cref{lem:LaminateProfile}
\begin{align*}
 v^{(2)}(x) = \phi(\frac{d_{\partial R^{(1)}}(x)}{r^2}) f_2(x;R^{(1)},r^2) + f_1(x;\Omega,r) + \tilde{S}_3(x)
\end{align*}
 and $v^{(2)}(x) = v^{(1)}(x)$ else. 
Here, and in the following, we use the notation $f_j(x;R^{(1)},r)$ for $R^{(1)}$ which is not a cuboid but an union of disjoint cuboids and mean depending on $x$ the corresponding connected cuboid.
By this $v^{(2)}$ defines a continuous map, as inside $R^{(1)}$ the cut-off attains the constant value one and therefore $v^{(1)}(x) = f_1(x;\Omega,r) + \tilde{S}_3(x)$ for $x \in \partial R^{(1)}$.
 By construction, the map $u^{(2)} := \operatorname{curl} v^{(2)}$ is divergence free, i.e. the interfaces are compatible.

As in the construction for first order laminates we set $\chi^{(2)} = \Pi_\mathcal{K} u^{(2)}$, 
 and define the sets
 \begin{align*}
  R^{(2)} & := \{x \in R^{(1)}: \phi(\frac{d_{\partial R^{(1)}}(x)}{r^2}) = 1, u^{(2)} = S_2\}, \\
  Q^{(2)} & := \{x \in R^{(1)}: \phi(\frac{d_{\partial R^{(1)}}(x)}{r^2}) = 1, u^{(2)} = A_2\}, \\
  C^{(2)} & := \{x \in R^{(1)}: \phi(\frac{d_{\partial R^{(1)}}(x)}{r^2}) < 1\}.
 \end{align*}
 Then indeed again we have the decomposition $R^{(1)} = R^{(2)} \cup Q^{(2)} \cup C^{(2)}$ with $|R^{(2)}| \leq \frac{1}{2} |R^{(1)}|$ and $|C^{(2)}| \leq \frac{1}{r} \frac{3}{8} r^2 6$ as the volume of the cut-off region can be bounded by the number of cells times $\frac{3}{8}r^2$ times six times the area of the biggest face.

As the elastic energy vanishes in $Q^{(2)}$, we obtain
 \begin{align*}
  E_{el}(u^{(2)},\chi^{(2)}) & = \int_\Omega |u^{(2)} - \chi^{(2)}|^2 \mathrm{d}x \leq C \Big( |R^{(2)}| + |C^{(2)}| + |C^{(1)}|\Big)\\
  & \leq C \Big( \frac{1}{4} + 3 r + \frac{3}{2} r\Big).
  \end{align*}
 Indeed, this follows from the fact that we have improved our deformation in half the volume of the region in which $u^{(1)} \notin \mathcal{K}$ but have added a new cut-off region in each cuboid in which we do the second order lamination.
  For the surface energy it holds
  \begin{align*}
  E_{surf}(\chi^{(2)}) & \leq C(\frac{10}{r} + \frac{1}{r} \frac{10}{r^2} \frac{r}{2}) = C (\frac{10}{r} + \frac{5}{r^2}),
 \end{align*}
 as we add at most $\frac{4}{r^2}+12 \leq \frac{10}{r^2}$ many new faces in each one of the $\frac{1}{r}$ many cuboids and as each surface has a surface area of size at most $\frac{r}{2}$.

 \subsubsection{Iteration: (m+1)-th order}
Without loss of generality, we assume that the $(m+1)$-th order lamination will be in $e_1$-direction, i.e. $m=3j$ for some $j \in \mathbb{N}$.
Else, we only have to adapt the corresponding roles of the directions.

We define iteratively the $(m+1)$-th potential with the help of the sets $R^{(m)}$, for this we set (for given $v^{(m)}, u^{(m)} = \operatorname{curl} v^{(m)}$)
\begin{align*}
 R^{(m)} & := \{x \in R^{(m-1)} : \phi(\frac{d_{\partial R^{(m-1)}}(x)}{r^{m}})=1, u^{(m)} = S_{3}\}, \\
 Q^{(m)} & := \{x \in R^{(m-1)} : \phi(\frac{d_{\partial R^{(m-1)}}(x)}{r^m})=1, u^{(m)} = A_{3}\}, \\
 C^{(m)} & := \{x \in R^{(m-1)} : \phi(\frac{d_{\partial R^{(m-1)}}(x)}{r^m}) < 1\}.
\end{align*}
Inside $R^{(m)}$ we then define $v^{(m+1)}$ by
\begin{align*}
 v^{(m+1)}(x) &= \phi(\frac{d_{\partial R^{(m)}}(x)}{r^{m+1}}) f_1(x;R^{(m)},r^{m+1}) + \sum_{k=1}^m f_{[k]}(x;R^{(k-1)},r^k) + \tilde{S}_3(x), \\
 [k] & = \begin{cases} 1 & k \equiv 1 \operatorname{mod} 3, \\ 2 & k \equiv 2 \operatorname{mod} 3, \\ 3 & k \equiv 0 \operatorname{mod} 3.
         \end{cases}
\end{align*}
and $v^{(m+1)}(x) = v^{(m)}(x)$ else for $x \notin R^{(m)}$.
By induction we see that $v^{(m+1)}$ is continuous, that for $x \in R^{(m)}$ such that $\phi(\frac{d_{\partial R^{(m)}}(x)}{r^{m+1}})=1$ it holds that $u^{(m+1)} := \operatorname{curl} v^{(m+1)} \in \{A_1,S_1\}$ and that we have the decomposition $R^{(m-1)} = R^{(m)} \cup Q^{(m)} \cup C^{(m)}$ with $|R^{(m)}| \leq \frac{1}{2} |R^{(m-1)}|, |C^{(m)}| \leq C 2^{-m} \frac{r^{m}}{r^{m-1}} = C 2^{-m} r$ for a universal constant $C>0$ independent of $m,r$.
This follows from the fact, that the number of cuboids in $R^{(m)}$ is bounded by $C \frac{ 2^{-m}}{r^{m-1}r^{m-2}r^{m-3}}$ and that the biggest face of each cuboid has an area of at most $r^{m-2} r^{m-3}$.

\begin{figure}
 \centering
 \begin{tikzpicture}[thick, scale = 0.5]

 \foreach \x in {1,2,...,10,11}
    \draw[draw = none,fill = gray!20] (\x,0,2.25) -- (\x,5,2.25) -- (\x,5,0) -- (\x,0,0) -- cycle;
 \foreach \x in {1,2,...,10,11} 
    \draw[gray] (\x,0,2.25) -- (\x,5,2.25) -- (\x,5,0);
 \foreach \x in {1,2,...,10,11}
    \draw[gray, dashed, thin] (\x,5,0) -- (\x,0,0) -- (\x,0,2.25);
 
 \draw[draw=none, fill = blue, fill opacity = 0.05] (0,0,0.75) -- (12,0,0.75) -- (12,5,0.75) -- (0,5,0.75) -- cycle;
 \draw[draw=none, fill = blue, fill opacity = 0.1] (0,0,1.5) -- (12,0,1.5) -- (12,5,1.5) -- (0,5,1.5) -- cycle;
 
 \draw[draw=none, fill = red, fill opacity = 0.05] (0,0.75,0) -- (12,0.75,0) -- (12,0.75,2.25) -- (0,0.75,2.25) -- cycle;
 \draw[draw=none, fill = red, fill opacity = 0.1] (0,4.25,0) -- (12,4.25,0) -- (12,4.25,2.25) -- (0,4.25,2.25) -- cycle;
 
 \draw[draw=none, fill = green, fill opacity = 0.05] (0.75,0,0) -- (0.75,5,0) -- (0.75,5,2.25) -- (0.75,0,2.25) -- cycle;
 \draw[draw=none, fill = green, fill opacity = 0.1] (11.25,0,0) -- (11.25,5,0) -- (11.25,5,2.25) -- (11.25,0,2.25) -- cycle;

 \draw[dashed] (0,5,0) -- (0,0,0) -- (12,0,0);
 \draw[dashed] (0,0,0) -- (0,0,2.25);
 \draw[very thick] (0,0,2.25) -- (12,0,2.25) -- (12,5,2.25) -- (0,5,2.25) -- (0,0,2.25);
 \draw[very thick] (12,0,0) -- (12,0,2.25) -- (12,5,2.25) -- (12,5,0);
 \draw[very thick] (0,5,2.25) -- (0,5,0) -- (12,5,0) -- (12,0,0);

 \draw (0,-0.4,2.25) -- (1,-0.4,2.25);
 \draw (0,-0.45,2.25) -- (0,-0.35,2.25);
 \draw (1,-0.45,2.25) -- (1,-0.35,2.25);
 \node[below] at (0.5,-0.4,2.25) {$\frac{r^{m+1}}{2}$};
 \draw (0,-0.2,2.25) -- (12,-0.2,2.25);
 \draw (0,-0.25,2.25) -- (0,-0.15,2.25);
 \draw (12,-0.25,2.25) -- (12,-0.15,2.25);
 \node[below] at (6,-0.2,2.25) {$\frac{r^{m-2}}{2}-\frac{3}{4}r^{m+1}-\frac{3}{4}r^{m}$};
 \draw (-0.2,0,2.25) -- (-0.2,5,2.25);
 \draw (-0.25,0,2.25) -- (-0.15,0,2.25);
 \draw (-0.25,5,2.25) -- (-0.15,5,2.25);
 \node[left] at (-0.2,2.5,2.25) {$\frac{r^{m-1}}{2}-\frac{3}{4}r^{m}$};
 \draw (0,5.2,2.25) -- (0,5.2,0);
 \draw (0,5.15,2.25) -- (0,5.25,2.25);
 \draw (0,5.15,0) -- (0,5.25,0);
 \node[above left] at (0,5.2,1.125) {$\frac{r^{m}}{2}$};
 
 \draw[->] (-2,0,2.25) -- (-1,0,2.25) node[below] {$e_1$};
 \draw[->] (-2,0,2.25) -- (-2,1,2.25) node[left] {$e_2$};
 \draw[->] (-2,0,2.25) -- (-2,0,1.25) node[above right] {$e_3$};
\end{tikzpicture}
\caption{One $S_3$-cell with the corresponding lamination in $e_1$-direction. In green/red/orange the inner boundary of the cut-off region is depicted.}
\end{figure}
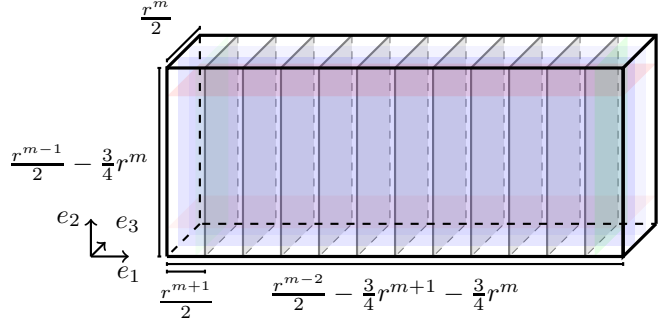

As the previous cut-offs still contribute to the total energy, we obtain the following elastic energy bound ($\chi^{(m+1)} := \Pi_\mathcal{K} u^{(m+1)}$)
\begin{align*}
 E_{el}(u^{(m+1)},\chi^{(m+1)}) & \leq C \big( |R^{(m+1)}| \cup \bigcup_{k=1}^{m+1} |C^{(k)}| \big) = C \left( 2^{-m} + r + \sum_{j=0}^{m+1} 2^{-j} r \right).
\end{align*}
This bound resembles the (iterative) decomposition of $\Omega = R^{(1)} \cup Q^{(1)} \cup C^{(1)} = R^{(m+1)} \cup \bigcup_{k=1}^{m+1} Q^{(k)} \cup \bigcup_{k=1}^{m+1} C^{(k)}$ and the fact that $u^{m+1} = \chi^{(m+1)}$ in $\bigcup_{k=1}^{m+1} Q^{(k)}$.

For the surface energy, we again calculate the new contribution of the next order lamination and then sum over all previous ones.
Each new face has a surface area of at most $\frac{r^m}{2} \frac{r^{m-1}}{2}$. In each $S_3$-cell, we add $\frac{r^{m-2}}{r^{m+1}}$ new faces in the lamination and at most $\frac{r^{m-2}}{r^{m+1}}+12$ ones for the cut-off.
Since we have at most $C \frac{2^{-m-1}}{r^m r^{m-1} r^{m-2}}$ new cells the surface energy is increased by
\begin{align*}
 C 2^{-m} \frac{8}{r^m r^{m-1} r^{m-2}} (\frac{r^{m-2}}{r^{m+1}} + (\text{surf. in cut-off })) \frac{r^m}{2} \frac{r^{m-1}}{2} \leq C 2^{-m} \frac{1}{r^{m+1}},
\end{align*}
yielding the following overall surface energy bound:
\begin{align*}
 E_{surf}(\chi^{(m+1)}) \leq C \left( \sum_{j=1}^{m+1} 2^{-j} r^{-j} \right) \leq \frac{C}{r^{m+1}}.
\end{align*}

For the total energy this implies
\begin{align*}
 E_{\epsilon}(u^{(m+1)},\chi^{(m+1)}) \leq C \left(2^{-m} + r + \sum_{j=1}^m 2^{-j} r + \epsilon \frac{1}{r^{m+1}}\right).
\end{align*}

With this we have shown the claimed upper bound and have thus concluded the proof of \cref{prop:energy_est_upper}.
\qed

\subsubsection{Combining the estimates: Proof of the upper bound in \cref{thm:scaling_T3}} \label{sec:T3UpperCombiningEstimates}

With the previous construction in hand, we conclude the upper estimate from \cref{thm:scaling_T3}:

\begin{prop}
\label{prop:upper}
Let $\Omega = [0,1]^3$, \(\K = \{A_1,A_2,A_3\}\) for \(A_1,A_2,A_3 \in \R^{3 \times 3}\) be given in \eqref{eq:problem1}, let $\epsilon \in (0,1)$, and let $E_{\epsilon}$ be as in \eqref{eq:E-Total}, $F\in \K^{qc} \setminus \K$ and \(\mathcal{D}_F\) given in \eqref{eq:admissible}.
 Then there are constants \(c>0\) and \(C>1\), only depending on the boundary data \(F\), such that
 \begin{align*}
  \inf_{u \in \mathcal{D}_F} \inf_{\chi \in BV(\Omega;\K)} E_{\epsilon}(u,\chi) \leq C \exp(-c |\log \epsilon|^{\frac{1}{2}}).
 \end{align*}
\end{prop}

\begin{proof}
\emph{Step 1: Conclusion of the argument for $F=S_3$.}
Using the sequences constructed in the construction from \cref{prop:energy_est_upper}, implies
\begin{align*}
 E_{\epsilon}(u^{(m)},\chi^{(m)}) \leq C \left( 2^{-m} + \sum_{k=1}^{m} 2^{-k}r + r + \epsilon r^{-m}\right) \leq C \left( 2^{-m} + r + \epsilon r^{-m}\right).
\end{align*}

Optimizing the value of $r$ depending on $\epsilon>0$, we require that \(r \sim \epsilon^{\frac{1}{m+1}}\).
Finally, we balance the resulting contributions and seek for the optimal number of iterations \(m\). This is given by \( 2^{-m} \sim \epsilon^{\frac{1}{m+1}}\), that is \(m \sim |\log\epsilon|^{\frac{1}{2}}\).

Plugging this into the upper bound results in
\begin{align*}
 E_{\epsilon}(u^{(m)},\chi^{(m)}) \leq C \exp\left(-\log(2)|\log(\epsilon)|^{\frac{1}{2}} \right).
\end{align*}
This concludes the proof for the special case $F = S_3$.

\emph{Step 2: Conclusion of the argument for a general boundary datum.}
The situation of other boundary data $F\in \mathcal{K}^{qc}$ can be reduced to the one from \cref{prop:upper} by at most two further iterations. 
Indeed, a general matrix $F\in \R^{3\times 3}$ in the convex hull of $S_1,S_2,S_3$, can be represented as
\begin{align*}
F= \lambda F_1 + (1-\lambda) F_2 = \lambda(\nu_1 A_j + (1-\nu_1) S_j) + (1-\lambda)(\nu_2 A_k + (1-\nu_2) S_k)
\end{align*}
with $\lambda, \nu_1, \nu_2 \in [0,1]$, $j,k\in \{1,2,3\}$ and $k\neq j$. Hence, after two additional iterations compared to the argument from above, we arrive at similar iterative procedures as in the previous subsections. 
In case that $F$ is an element of one of the legs $S_j A_j$ a single iteration suffices to reduce the situation to the above argument. This proves the result for a general boundary condition $F\in \K^{qc} \setminus \K$.

\end{proof}

\subsection{Proof of the lower bound}

In this section, we present the proof of the lower bound from \cref{thm:scaling_T3}. To this end, similarly as in \cite{Ruland22}, we mimic and quantify the analogous argument from the stress-free setting which we briefly recall in the following \cref{sec:stress_free} and for which we will provide a number of auxiliary results in \cref{sec:aux_lower}. The main argument, given in \cref{sec:comp_Fourier}, will then consist of a bootstrap strategy, similar to \cite{Ruland22}, in which we iteratively reduce the possible regions of mass concentration in Fourier space.

Contrary to the previous section, in what follows we will work in a periodic set-up. Since the energy contributions on periodic functions provides a lower bound on the energy contributions of functions with prescribed Dirichlet boundary conditions, we hence also obtain the desired lower bound for the setting of Dirichlet boundary conditions. Indeed, for the elastic energy this is immediate; for the surface energy there is at most an increase by a fixed factor (see the discussion in \cref{lem:per} in \cref{sec:aux_lower}).

\subsubsection{The stress-free argument}
\label{sec:stress_free}

We begin by recalling the argument for the rigidity of the exact inclusion, as we will mimic this on the energetic level.

\begin{prop}
\label{prop:rigid_exact}
Let $u: [0,1]^3 \rightarrow \R^{3 \times 3}$ be a solution to the differential inclusion \eqref{eq:problem}-\eqref{eq:problem1}. Then, there exists $j\in \{1,2,3\}$ such that $u \equiv A_j$ in $[0,1]^3$.
\end{prop}

\begin{proof}
In the exactly stress-free setting in which the differential inclusion is satisfied exactly, i.e. \(u \in \{A_1,A_2,A_3\}\), the observation that $\di u=0$ and that \(u\) is a diagonal matrix leads to the following three equations
\begin{align*}
\p_1 u_{11} = 0, \ \p_2 u_{22} =0 , \ \p_3 u_{33} =0,
\end{align*}
where $u_{jj}$ denote the diagonal components of the matrix $u$.
As a consequence,
\begin{align*}
u_{11} = f_{1}(x_2,x_3), \  u_{22} = f_{2}(x_1,x_3) , \  u_{33} = f_{3}(x_1,x_2).
\end{align*}
Next we note that the values of $u_{jj}$ determine the ones for $u_{kk}$ if $j\neq k$, i.e. there are functions \(h_{k,j}\) such that \(h_{k,j}(u_{jj}) = u_{kk}\). Hence, comparing the functions $u_{11}$ and $u_{22}$, we first obtain that $u_{11}$ and $u_{22}$ can only be functions of $x_3$.
Indeed it holds
\begin{align} \label{eq:ComparisonExact}
 \p_2 u_{11}(x) = \p_2 \big(h_{1,2}(u_{22}(x))\big) = \p_2 \big( h_{1,2}(f_{13}(x_1,x_3))\big) = 0,
\end{align}
and analogously \(\p_1 u_{22}(x) = 0\). Comparing this to $u_{33}$, we obtain that all three functions must be constant. Hence, any solution to the (exact) differential inclusion must be constant and $u$ is equal to one of the three matrices $A_1,A_2,A_3$ globally. The exact problem is hence rigid.
\end{proof}

Using the ideas from \cite{Ruland22}, we seek to turn this into a corresponding scaling result.
The main difference that arises can be seen in the qualitative rigidity argument above:
Instead of comparing only two diagonal entries like in \cite{Ruland22}, we have to compare twice to deduce that the map is constant.
This will be seen in the quantitative argument for the lower bound below.
Whereas in \cite{Ruland22} there are cones around a single axis (the diagonal entries only depend on one variable), we consider cones around a plane (the diagonal entries depend on two variables).
Furthermore, the bootstrap argument will be slightly modified as it resembles the comparison of the diagonal entries in the qualitative argument given above.

\subsubsection{Reduction to the periodic setting and auxiliary results for the elastic energy}
\label{sec:aux_lower}

In this subsection, we provide a number of auxiliary results which we will exploit in the following bootstrap arguments for deducing the lower bound.
As a first step, we reduce to the situation of periodic deformations.

\begin{lem}
\label{lem:per}
Let $\Omega = [0,1]^3$, let $\mathcal{K}:=\{A_1,A_2,A_3\}$ be as in \eqref{eq:problem1} and $F\in \mathcal{K}^{qc} \setminus \mathcal{K}$. Let
$E_{\epsilon}(u,\chi)$ be given by \eqref{eq:energy} and set
\begin{align*}
E_{\epsilon}^{per}(u,\chi):= \int\limits_{\mathbb{T}^3}|u-\chi|^2 \mathrm{d}x + \epsilon \int\limits_{\mathbb{T}^3} |\nabla \chi|.
\end{align*}
Let further $\mathcal{D}_F^{per}:=\{u:\mathbb{R}^3 \rightarrow \mathbb{R}^{3 \times 3}: \ \di u=0 \mbox{ in } \mathbb{R}^3, \ \langle  u \rangle = F \}$, where $\langle u \rangle := \int\limits_{\mathbb{T}^3} u(x) \mathrm{d}x$.
Assume that $E_{\epsilon}(u,\chi)\leq 1$ and that there is $\epsilon_0 >0$ such that for any $\nu\in (0,\frac{1}{2})$ there is $c_{\nu}>0$ such that for any $\epsilon \in (0,\epsilon_0)$ it holds that
\begin{align}
\label{eq:lower_per}
E_{\epsilon}^{per}(u,\chi) \geq \exp(-c_{\nu}|\log(\epsilon)|^{\frac{1}{2} + \nu}).
\end{align}
Then, there exists a constant $C>1$ such that for $\tilde\epsilon_0=\tilde\epsilon_0(\nu)>0$ sufficiently small
\begin{align*}
C^{-1} \exp(-c_{\nu}|\log(\epsilon)|^{\frac{1}{2} + \nu})
\leq  \inf\limits_{\chi \in BV([0,1]^3;\mathcal{K})} \inf\limits_{u \in \mathcal{D}_F} E_{\epsilon}(u,\chi)
\end{align*}
for all $\epsilon \in (0,\tilde\epsilon_0)$.
\end{lem}

\begin{proof}
In order to infer the lower bound, we show that any function $u:[0,1]^3 \rightarrow \R^3$ with constant boundary data can be associated with a suitable periodic function which has the boundary data of $u$ as its mean value and satisfies related energy estimates. Indeed, for given $u \in \mathcal{D}_F$, we view it as a function on $\T^3$ by restriction. By the prescribed boundary data it still satisfies the differential constraint and further the mean value property. Moreover,
\begin{align*}
\inf\limits_{u \in \mathcal{D}^{per}_F} E_{el}^{per}(u,\chi) := \inf_{u \in \mathcal{D}_F^{per}} \int_{\T^3} |u-\chi|^2 dx
\leq  \inf\limits_{u \in \mathcal{D}_F} E_{el}(u,\chi).
\end{align*}
Next, viewing $\chi:[0,1]^3 \rightarrow \{0,1\}$ as a periodic function $\tilde{\chi}:\T^3 \rightarrow \{0,1\}$, we infer that
\begin{align*}
\int\limits_{\T^3} |\nabla \tilde{\chi}| \leq \int\limits_{[0,1]^3} |\nabla \chi| + C.
\end{align*}
Now, due to \eqref{eq:lower_per}, we obtain that
\begin{align*}
\exp(-c_{\nu}|\log(\epsilon)|^{\frac{1}{2} + \nu}) \leq \inf\limits_{\chi \in BV(\T^3;\mathcal{K})} \inf\limits_{u \in \mathcal{D}^{per}_F} E_{\epsilon}^{per}(u,\chi)
\leq  \inf\limits_{\chi \in BV([0,1]^3;\mathcal{K})} \inf\limits_{u \in \mathcal{D}_F} E_{\epsilon}(u,\chi) + C \epsilon.
\end{align*}
For $\tilde\epsilon_0(\nu)>0$ sufficiently small, the last right hand side term may thus be absorbed into the left hand side, yielding the desired result.
\end{proof}

With this result in hand it suffices to consider the periodic set-up in the remainder of this section.
This will, in particular, allow us to rely on the periodic Fourier transform in deducing lower bounds for the elastic energy. 
In what follows all (semi-)norms will thus be considered on the torus. With a slight abuse of notation, we will often omit this dependence.

\begin{lem}\label{lem:FourierElastic}
Let \(F \in \K^{qc}\), where \(\K = \{A_1,A_2,A_3\}\) with \(A_j\) in \eqref{eq:problem1}, and \(E_{el}^{per}\) and $\mathcal{D}_F^{per}$ be as in \cref{lem:per} and as in \eqref{eq:admissible}. Then, it holds for any \(\chi \in L^2([0,1]^3;\K)\) and for \(E_{el}^{per}(\chi;F) := \inf_{u \in \mathcal{D}_F^{per}} \int_{\T^3} \big| u- \chi \big|^2 dx \)
 \begin{align*}
  E_{el}^{per}(\chi;F) = \sum_{k \in \Z^3 \setminus \{0\}} \sum_{i=1}^{3} \frac{k_i^2}{|k|^2} |\hat{\chi}_{i,i}|^2 + |\hat{\chi}(0)-F|^2,
 \end{align*}
 where \(\chi_{i,i}\) are the diagonal entries of \(\chi\) and \(\hat{\chi}\) is the (discrete) Fourier transform of \(\chi\).
\end{lem}

\begin{proof}
 We first calculate \(E_{el}^{per}(u,\chi)\) in Fourier space
  \begin{align*}
  E_{el}^{per}(u,\chi) = \int_{\T^3} |u-\chi|^2 dx = \sum_{k \in \Z^3} |\hat{u} - \hat{\chi}|^2,
 \end{align*}
which allows us to characterize minimizers of the elastic energy.

 In order to minimize this elastic energy in \(u \in \mathcal{D}_F^{per}\), \(\hat{u}\) has to be the (pointwise) orthogonal projection of \(\hat{\chi}\) onto the orthogonal complement of \(k\), as the differential constraint \(\di u = 0\) reads \(i \hat{u} k = 0\) in Fourier space.
Noting that the row-wise orthogonal projection of a matrix \(M\) onto \(\vspan(k)\) can be written as \(\Pi_k(M) = (M \frac{k}{|k|}) \otimes \frac{k}{|k|}\), the optimal \(\hat{u}\) is given by
 \begin{align*}
  \hat{u} = \Pi_{k^\perp} \hat{\chi} = (\Id - \Pi_{k}) \hat{\chi} = \hat{\chi} - ( \hat{\chi} \frac{k}{|k|}) \otimes \frac{k}{|k|}
 \end{align*}
for any \(k \in \Z^3 \setminus \{0\}\).

Returning to our energy, this yields
\begin{align*}
 E_{el}^{per}(\chi;F) &= E_{el}^{per}(u,\chi) = \sum_{k \in \Z^3\setminus\{0\}} |\hat{\chi} - (\hat{\chi} \frac{k}{|k|}) \otimes \frac{k}{|k|} - \hat{\chi}|^2 +|F - \hat{\chi}(0)|^2 \\
 & = \sum_{k \in \Z^3 \setminus \{0\}} |\hat{\chi} \frac{k}{|k|}|^2 + |\hat{\chi}(0)-F|^2 \\
 &=  \sum_{k \in \Z^3 \setminus \{0\}} \sum_{i=1}^{3} \frac{k_i^2}{|k|^2} |\hat{\chi}_{i,i}|^2 + |\hat{\chi}(0)-F|^2
\end{align*}
and shows the claim.
\end{proof}

Next, following the ideas from \cite{Ruland22}, for \(\mu,\lambda >0, j=1,2,3\), we introduce the cones
\begin{align} \label{eq:T3Cones}
 C_{j,\mu,\lambda} = \{k \in \Z^3 : |k_j| \leq \mu |k|, |k| \leq \lambda\},
\end{align}
and their corresponding cut-off functions \(m_{j,\mu,\lambda}(k) \in C^\infty(C_{j,2\mu,2\lambda}\setminus \{0\};[0,1])\) fulfilling \(m_{j,\mu,\lambda} = 1\) on \(C_{j,\mu,\lambda}\), $\supp(m_{j,\mu,\lambda}(k))\subset C_{j,2\mu,2\lambda}$ and the decay properties in Marcinkiewicz's multiplier theorem (see, for instance, \cite[Corollary 6.2.5]{G14}).
The corresponding cut-off multiplier is thus defined by
\begin{align} \label{eq:CutOffMult}
 m_{j,\mu,\lambda}(D) f = \mathcal{F}^{-1}(m_{j,\mu,\lambda}(\cdot) \hat{f}(\cdot)).
\end{align}

Furthermore, we use the following results which are shown in \cite[Lemma 2, Lemma 3, and Corollary 1]{Ruland22}. Following the conventions from \cite{RT21}, with slight abuse of notation compared to our setting in the first part of the article, in the whole following section, we now use $d$ to denote the degree of some suitable polynomials and no longer the dimension of the ambient space which in the whole section is simply fixed to be equal to three.

\begin{lem} \label{lem:ResultsRT21}
Let $\beta,\delta, \mu,\lambda>0$. Let $m_{i,\mu,\lambda}(D)$ denote the Fourier multipliers associated with the cones $C_{i,\mu,\lambda}$ for $i\in\{1,2,3\}$ as defined in \eqref{eq:T3Cones} with the corresponding multipliers given in \eqref{eq:CutOffMult}.
 Let \(f_i \in L^{\infty}(\T^3) \cap BV(\T^3)\) for \(i=1,2,3\) and let \(h_{j,i}: \R \to \R\) be nonlinear polynomials (of degree \(d\)) with \(h_{j,i}(0) = 0\) such that \(h_{j,i}(f_i) = f_j\) for \(i \neq j\).
 If
 \begin{align*}
  \sum_{i=1}^3 \Vert \p_i f_i \Vert_{\dot{H}^{-1}}^2 \leq \delta, \quad \sum_{i=1}^3 \Vert \nabla f_i \Vert_{TV} \leq \beta,
 \end{align*}
then there exist constants $C = C(h_{i,j},\Vert f_i \Vert_\infty),C' = C'(h_{i,j},\Vert f_i \Vert_\infty, d),C_0 = C_0(h_{i,j},\Vert f_i \Vert_\infty,d)>0$ such that for any $\gamma \in (0,1)$ we have for any \(i \neq j\)
\begin{align*}
 \sum_{k=1}^3 \Vert f_k - m_{k,\mu,\lambda}(D)f_k \Vert_{L^2}^2
 &\leq C(\mu^{-2}\delta + \lambda^{-1}\beta), \\
 \Vert f_i - h_{i,j}(m_{j,\mu,\lambda}(D)f_j) \Vert_{L^2}
 & \leq \frac{C'}{\gamma^{12 d}} \Vert f_j - m_{j,\mu,\lambda}(D)f_j \Vert_{L^2}^{1-\gamma}, \\
 \Vert h_{i,j}(m_{j,\mu,\lambda}(D)f_j) - m_{i,\mu,\lambda}(D)f_i\Vert_{L^2}^2
 & \leq \frac{C_0}{\gamma^{24 d}} \max\{(\mu^{-2}\delta + \lambda^{-1}\beta)^{1-\gamma},\mu^{-2}\delta+\lambda^{-1}\beta\}.
\end{align*}
Here we choose the constants such that \(C_0 > 2C+2C'^2+3\).
\end{lem}

Let us comment on these bounds: The functions \(f_i\) are representing the diagonal entries of the phase indicator \(\chi \in BV([0,1]^3;\K)\). Thus the first estimate corresponds to a first frequency localization by exploiting the surface energy control for the high frequencies and the ellipticity of the elastic energy away from the cones $C_{j,\mu,\lambda}$.
It can be viewed as a quantified version of the statement that \(u_{11}\) is a function only depending on \(x_2,x_3\) in \cref{prop:rigid_exact}.
The second estimate is a commutator bound that arises from the nonlinear relation \(h_{i,j}(f_j) = f_i\) for \(i \neq j\).
The third estimate combines the first two bounds.
The second and third estimate will form the core tool to iteratively decrease the Fourier support of the characteristic functions of our phase indicators. We will detail this in the remainder of the article.

\begin{rmk}
It is possible to make the mappings \(h_{j,i}\) for \(i,j \in \{1,2,3\}, i \neq j\) explicit for the choice of matrices \(A_1,A_2,A_3\), c.f. \eqref{eq:problem1}. To this end, we may, for instance, consider
 \begin{align*}
  h_{1,2}(x) & = \frac{14}{9}x^2 - \frac{5}{9}x, & h_{1,3}(x) & = \frac{14}{3}x^2 - \frac{11}{3}x,\\
  h_{2,1}(x) & = \frac{21}{4}x^2-\frac{17}{4}x, & h_{2,3}(x) & = -\frac{21}{2}x^2+\frac{23}{2}x,\\
  h_{3,1}(x) & = -\frac{7}{12}x^2+\frac{19}{12}x, & h_{3,2}(x) & = -\frac{7}{18}x^2+\frac{25}{18}x.
 \end{align*}
\end{rmk}

\subsubsection{Comparison argument in Fourier space}
\label{sec:comp_Fourier}

In this section, we carry out the iterative bootstrap argument which allows us to deduce the final rigidity result.

As a first step of the bootstrap argument, we invoke the results from above which allow us to decrease the region of potential Fourier concentration from the cone $C_{1,\mu,\lambda}$ to a cone $C_{1,\mu,\lambda'}$ with \(\lambda' < \lambda\).
This resembles \eqref{eq:ComparisonExact} in the exactly stress-free setting in a quanitified version. As \(f_1\) is determined by \(f_2\) and \(f_3\) with the help of \(h_{1,j}\) we can reduce \(\lambda\), similarly as in our reduction of the dependences of \(u_{11}\) in \cref{prop:rigid_exact}.

\begin{lem}\label{lem:FirstComparison}
 Under the same conditions as in \cref{lem:ResultsRT21} and for
 \begin{align*}
  \lambda > 0, \  \mu \in (0, \frac{1}{2 \sqrt{2d^2+1}}), \ \lambda' \in (\frac{1}{\sqrt{2}} \frac{4 d \lambda \mu}{\sqrt{1-4\mu^2}}, \lambda) 
 \end{align*}
 it holds
 \begin{align*}
    \Vert f_1 - m_{1,\mu,\lambda'}(D)f_1 \Vert_{L^2}^2 & +\Vert f_2 - m_{2,\mu,\lambda}(D)f_2 \Vert_{L^2}^2 +\Vert f_3 - m_{3,\mu,\lambda}(D)f_3 \Vert_{L^2}^2 \\
    & \leq 10 \frac{C_0}{\gamma^{24 d}} \max\{(\mu^{-2}\delta + \lambda^{-1}\beta)^{1-\gamma},\mu^{-2}\delta+\lambda^{-1}\beta\}.
 \end{align*}
 Here the constant \(C_0\) is chosen to be the same as in \cref{lem:ResultsRT21}.
\end{lem}

\begin{rmk}
We remark that the interval for \(\lambda'\) is chosen such that it holds \(C_{1,2\mu,2\lambda} \setminus C_{1,2\mu,2\lambda'} \subset \{\max\{|k_2|,|k_3|\} > 4 d \mu \lambda\} \cap C_{1,2\mu,2\lambda}\) and the one for \(\mu\) such that \(\frac{1}{\sqrt{2}} \frac{4d\lambda\mu}{\sqrt{1-4\mu^2}} < \lambda\), i.e. the interval for \(\lambda'\) is non-empty.
\end{rmk}

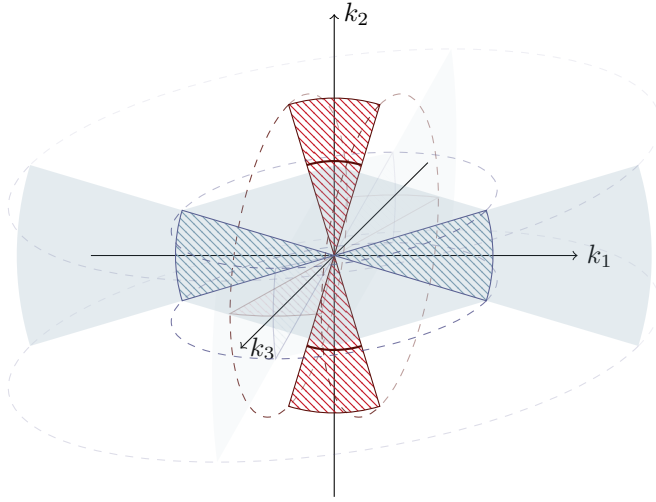
\begin{figure}
\begin{tikzpicture}[scale=0.4]

  \draw[->] (-8,0,0) -- (8,0,0) node[right] {$k_1$};
  \draw[->] (0,-8,0) -- (0,8,0) node[right] {$k_2$};
  \draw[->] (0,0,-8) -- (0,0,8) node[right] {$k_3$};

  \draw[color = HDredDark, canvas is xz plane at y=0, pattern=north west lines, pattern color=HDred, opacity=0.3] (0,0) -- (1.5,5) arc[start angle = 73.3, end angle = 106.7, radius = 5.22] -- (0,0);
  
  \draw[color = HDredDark, canvas is xz plane at y=0, pattern=north west lines, pattern color=HDred, opacity=0.1] (0,0) -- (1.5,-5) arc[start angle = -73.3, end angle = -106.7, radius = 5.22] -- (0,0);

  \draw[color = HDredDark, canvas is xy plane at z=0, pattern=north west lines, pattern color=HDred] (0,0) -- (1.5,5) arc[start angle = 73.3, end angle = 106.7, radius = 5.22] -- (0,0);
  
  \draw[color = HDredDark, canvas is xy plane at z=0, pattern=north west lines, pattern color=HDred] (0,0) -- (1.5,-5) arc[start angle = -73.3, end angle = -106.7, radius = 5.22] -- (0,0);

  \draw[canvas is yz plane at x=1.5, color = HDredDark, opacity=0.8, dashed] (5,0) arc[start angle = 0, end angle = 180, radius = 5];
  \draw[canvas is yz plane at x=-1.5, color = HDredDark, opacity=0.8, dashed] (5,0) arc[start angle = 0, end angle = 180, radius = 5];
  \draw[canvas is yz plane at x=1.5, color = HDredDark, opacity=0.4, dashed] (5,0) arc[start angle = 0, end angle = -180, radius = 5];
  \draw[canvas is yz plane at x=-1.5, color = HDredDark, opacity=0.4, dashed] (5,0) arc[start angle = 0, end angle = -180, radius = 5];

  \draw[color = PurpleNavy, canvas is xy plane at z = 0, pattern=north west lines, pattern color=PewterBlue] (0,0) -- (5,1.5) arc[start angle = 16.7, end angle = -16.7, radius = 5.22] -- (0,0);
  
  \draw[color = PurpleNavy, canvas is xy plane at z = 0, pattern=north west lines, pattern color=PewterBlue] (0,0) -- (-5,1.5) arc[start angle = 163.3, end angle = 196.7, radius = 5.22] -- (0,0);
  
  \draw[color = PurpleNavy, canvas is zy plane at x = 0, pattern=north west lines, pattern color=PewterBlue, opacity=0.3] (0,0) -- (5,1.5) arc[start angle = 16.7, end angle = -16.7, radius = 5.22] -- (0,0);
  
  \draw[color = PurpleNavy, canvas is zy plane at x = 0, pattern=north west lines, pattern color=PewterBlue, opacity=0.1] (0,0) -- (-5,1.5) arc[start angle = 163.3, end angle = 196.7, radius = 5.22] -- (0,0);
 
  \draw[canvas is xz plane at y=1.5, color = PurpleNavy, opacity=0.8, dashed] (5,0) arc[start angle = 0, end angle = 180, radius = 5];
  \draw[canvas is xz plane at y=-1.5, color = PurpleNavy, opacity=0.8, dashed] (5,0) arc[start angle = 0, end angle = 180, radius = 5];
  \draw[canvas is xz plane at y=1.5, color = PurpleNavy, opacity=0.4, dashed] (5,0) arc[start angle = 0, end angle = -180, radius = 5];
  \draw[canvas is xz plane at y=-1.5, color = PurpleNavy, opacity=0.4, dashed] (5,0) arc[start angle = 0, end angle = -180, radius = 5];

  \fill[canvas is xy plane at z=0, color = PewterBlue, opacity = 0.2] (0,0) -- (10,3) arc[start angle = 16.7, end angle = -16.7, radius = 10.44] -- (0,0);
  \fill[canvas is xy plane at z=0, color = PewterBlue, opacity = 0.2] (0,0) -- (-10,3) arc[start angle = 163.3, end angle = 196.7, radius = 10.44] -- (0,0);
  \fill[canvas is xy plane at z=0, color = PewterBlue, opacity = 0.2] (0,0) -- (5,1.5) -- (0,3) -- (-5,1.5) -- (0,0);
  \fill[canvas is xy plane at z=0, color = PewterBlue, opacity = 0.2] (0,0) -- (5,-1.5) -- (0,-3) -- (-5,-1.5) -- (0,0);
  
  \fill[canvas is zy plane at x=0, color = PewterBlue!50, opacity = 0.1] (0,0) -- (10,3) arc[start angle = 16.7, end angle = -16.7, radius = 10.44] -- (0,0);
  \fill[canvas is zy plane at x=0, color = PewterBlue!50, opacity = 0.1] (0,0) -- (-10,3) arc[start angle = 163.3, end angle = 196.7, radius = 10.44] -- (0,0);
  \fill[canvas is zy plane at x=0, color = PewterBlue!50, opacity = 0.1] (0,0) -- (5,1.5) -- (0,3) -- (-5,1.5) -- (0,0);
  \fill[canvas is zy plane at x=0, color = PewterBlue!50, opacity = 0.1] (0,0) -- (5,-1.5) -- (0,-3) -- (-5,-1.5) -- (0,0);
  
  \draw[canvas is xz plane at y=3, color = PurpleNavy, opacity=0.2, dashed] (10,0) arc[start angle = 0, end angle = 180, radius = 10];
  \draw[canvas is xz plane at y=-3, color = PurpleNavy, opacity=0.2, dashed] (10,0) arc[start angle = 0, end angle = 180, radius = 10];
  \draw[canvas is xz plane at y=3, color = PurpleNavy, opacity=0.1, dashed] (10,0) arc[start angle = 0, end angle = -180, radius = 10];
  \draw[canvas is xz plane at y=-3, color = PurpleNavy, opacity=0.1, dashed] (10,0) arc[start angle = 0, end angle = -180, radius = 10];
  
  \draw[thick,canvas is xy plane at z=0, color = HDredDark] (0.9,3) arc[start angle = 73.3, end angle = 106.7, radius = 3.132];\draw[thick,canvas is xy plane at z=0, color = HDredDark] (0.9,-3) arc[start angle = -73.3, end angle = -106.7, radius = 3.132];
  
  \end{tikzpicture}
 \caption{Illustration of the cones \(C_{1,\mu,\lambda}\) (red) and \(C_{2,\mu,\lambda}\) (blue), and also of \(C_{2,\mu,\lambda}+C_{2,\mu,\lambda}\) in Fourier space.}
\end{figure}

\begin{proof}
We first observe that the Fourier transform of \(h_{i,j}(m_{j,\mu,\lambda}(D)f_j)\) is given by a convolution of the functions $m_{j,\mu,\lambda}(D)f_j$. Hence, the support of \(h_{i,j}(m_{j,\mu,\lambda}(D)f_j)\) is contained in the \(d\)-fold Minkowski-sum of \(C_{j,2\mu,2\lambda}\) with itself.
For \(j=2,3\) it therefore holds that
 \begin{align*}
  \mathcal{F}\big(h_{1,j}(m_{j,\mu,\lambda}(D) f_j)\big)(k) = 0 \quad \text{for } |k_j| > 4 d \mu \lambda.
 \end{align*}
Further we introduce the sets \(K_2 = \{|k_2| > 4 d \mu \lambda\}, K_3 = \{|k_3| > 4 d \mu \lambda\}\) and consider the corresponding Fourier multipliers of the smoothed-out indicator functions  \(\chi_{K_2}(D),\chi_{K_3}(D),\chi_{K_2 \cup K_3}(D)\).
This implies for \(i=2,3\)
\begin{align}
\label{eq:zero}
 \chi_{K_i}(D) h_{1,i}(m_{i,\mu,\lambda}(D) f_i) = 0.
\end{align}
With this we can show that the Fourier mass concentrates in a cone with smaller truncation parameter: Indeed, in Fourier space \(\chi_{K_2 \cup K_3} \leq \chi_{K_2} + \chi_{K_3}\). Thus,
\begin{align}
\label{eq:triangle}
\begin{split}
 \Vert \chi_{K_2 \cup K_3}(D) m_{1,\mu,\lambda}(D) f_1 \Vert_{L^2}^2 
 &\leq \Vert (\chi_{K_2}(D) + \chi_{K_3}(D)) m_{1,\mu,\lambda}(D) f_1 \Vert_{L^2}^2 \\
& \leq  2 \Vert \chi_{K_2}(D) m_{1,\mu,\lambda}(D) f_1 \Vert_{L^2}^2 + 2 \Vert \chi_{K_3}(D) m_{1,\mu,\lambda}(D) f_1 \Vert_{L^2}^2.
\end{split}
\end{align}
Invoking \eqref{eq:zero} together with the bounds from \cref{lem:ResultsRT21}, then implies for \(i=2,3\)
\begin{align*}
 \Vert \chi_{K_i}(D) m_{1,\mu,\lambda}(D) f_1 \Vert_{L^2}^2
 & =  \Vert \chi_{K_i}(D) \big( m_{1,\mu,\lambda}(D) f_1 - h_{1,i}(m_{i,\mu,\lambda}(D)f_i)\big)\Vert_{L^2}^2 \\
 & \leq  \Vert m_{1,\mu,\lambda}(D) f_1 - h_{1,i}(m_{i,\mu,\lambda}(D)f_i)\Vert_{L^2}^2 \\
 & \leq  \frac{C_0}{\gamma^{24 d}} \max\{(\mu^{-2}\delta + \lambda^{-1}\beta)^{1-\gamma},\mu^{-2}\delta + \lambda^{-1}\beta\}.
\end{align*}
Combined with \eqref{eq:triangle} this yields
\begin{align*}
 \Vert \chi_{K_2 \cup K_3}(D) m_{1,\mu,\lambda}(D) f_1 \Vert_{L^2}^2 \leq 4 \frac{C_0}{\gamma^{24 d}} \max\{(\mu^{-2}\delta + \lambda^{-1}\beta)^{1-\gamma},\mu^{-2}\delta + \lambda^{-1}\beta\}.
\end{align*}

We observe that by the choice of the parameters \(C_{1,2\mu,2\lambda} \setminus C_{1,2\mu,2\lambda'} \subset (K_2 \cup K_3) \cap C_{1,2\mu,2\lambda}\), and thus \(|m_{1,\mu,\lambda}(k) - m_{1,\mu,\lambda'}(k)| \leq \chi_{K_2 \cup K_3}(k) m_{1,\mu,\lambda}(k)\). Therefore,
\begin{align*}
 \Vert m_{1,\mu,\lambda}(D) f_1 - m_{1,\mu,\lambda'}(D)f_1\Vert_{L^2}^2 & \leq  \Vert \chi_{K_2 \cup K_3}(D) m_{1,\mu,\lambda}(D)f_1 \Vert_{L^2}^2 \\
& \leq  4 \frac{C_0}{\gamma^{24 d}} \max\{(\mu^{-2}\delta + \lambda^{-1}\beta)^{1-\gamma},\mu^{-2}\delta + \lambda^{-1}\beta\}.
\end{align*}

In conclusion, (for \(C_0 \geq C\))
\begin{align*}
 & \Vert f_1 - m_{1,\mu,\lambda'}(D)f_1 \Vert_{L^2}^2 +\Vert f_2 - m_{2,\mu,\lambda}(D)f_2 \Vert_{L^2}^2+\Vert f_3 - m_{3,\mu,\lambda}(D)f_3 \Vert_{L^2}^2 \\
 &\quad \leq  \Vert f_2 - m_{2,\mu,\lambda}(D)f_2 \Vert_{L^2}^2+\Vert f_3 - m_{3,\mu,\lambda}(D)f_3 \Vert_{L^2}^2 \\
  & \qquad +2 \Vert f_1 - m_{1,\mu,\lambda}(D)f_1 \Vert_{L^2}^2 + 2 \Vert m_{1,\mu,\lambda}(D) f_1 - m_{1,\mu,\lambda'}(D)f_1\Vert_{L^2}^2 \\
 & \quad \leq  2 C (\mu^{-2}\delta + \lambda^{-1} \beta) +8 \frac{C_0}{\gamma^{24 d}} \max\{(\mu^{-2}\delta + \lambda^{-1}\beta)^{1-\gamma},\mu^{-2}\delta + \lambda^{-1}\beta\} \\
   & \quad \leq 10 \frac{C_0}{\gamma^{24 d}} \max\{(\mu^{-2}\delta + \lambda^{-1}\beta)^{1-\gamma},\mu^{-2}\delta+\lambda^{-1}\beta\}.
\end{align*}
\end{proof}

Let us stress that the decomposition into \(K_2\) and \(K_3\) is in analogy to the two comparisons from \cref{prop:rigid_exact} in order to show that \(u_{11}\) is constant.
The set \(K_2\) resembles the comparison of \(u_{11}\) and \(u_{22}\) to show that \(u_{11}\) is constant in \(x_2\) and the set \(K_3\) resembles the comparison of \(u_{11}\) and \(u_{33}\) to show that $u_{11}$ does not depend on $x_3$.

Applying the previous result for all three directions simultaneously then yields the following corollary which will serve as the induction basis for the subsequent inductive bootstrap argument.

\begin{cor}[Induction Basis] \label{lem:InductionBasis}
Let \(\beta,\delta,\mu,\lambda>0\) and let \(C_{i,\mu,\lambda}\) be the cones in \eqref{eq:T3Cones} with corresponding multiplier \(m_{i,\mu,\lambda}(D)\), cf. \eqref{eq:CutOffMult}, for \(i=1,2,3\).
Further let \(f_i,h_{j,i}\) be functions as in \cref{lem:ResultsRT21}.
Let $C_0>0$, $\gamma \in (0,1)$ and $d\geq 0$ be the constants from \cref{lem:ResultsRT21}.  For \(\lambda_0 = \lambda > 0, \ \mu \in (0,\frac{1}{2\sqrt{2d^2+1}}), \ \lambda_1 \in (\frac{4d\mu\lambda_0}{\sqrt{2}\sqrt{1-4\mu^2}},\lambda_0)\) it holds that
 \begin{align*}
 \begin{split}
  \Vert f_1 - m_{1,\mu,\lambda_1}(D) f_1 \Vert_{L^2}^2 + \Vert f_2 - m_{2,\mu,\lambda_1}(D) f_2 \Vert_{L^2}^2 + \Vert f_3 - m_{3,\mu,\lambda_1}(D) f_3 \Vert_{L^2}^2 \\
  \leq \frac{30 C_0}{\gamma^{24 d}} \max\{(\mu^{-2}\delta + \lambda^{-1}\beta)^{1-\gamma},\mu^{-2}\delta + \lambda^{-1}\beta\}.
  \end{split}
 \end{align*}
\end{cor}

With \cref{lem:InductionBasis} in hand, we now iteratively further decrease the Fourier supports.
To this end, we will invoke the commutator bounds from \cref{lem:ResultsRT21}.

\begin{lem}[Iteration process]
\label{lem:iterate}
Let \(\beta,\delta>0\) and let \(C_{i,\mu,\lambda}\) be the cones in \eqref{eq:T3Cones} with corresponding multiplier \(m_{i,\mu,\lambda}(D)\), cf. \eqref{eq:CutOffMult}, for \(i=1,2,3\).
Further let \( \mu \in (0, 
\frac{1}{2 \sqrt{2d^2+1}})\), \(\lambda >0\), and let \(f_i,h_{j,i}\) be functions as in \cref{lem:ResultsRT21}.
Let $C_0>0$, $\gamma \in (0,1)$ and $d\geq 0$ be the constants from \cref{lem:ResultsRT21}.
Let \(\lambda_k > 0\) be a sequence for \(k \in \N\) with \(\lambda_0 = \lambda\) and \(\lambda_k \in (\frac{4d\mu \lambda_{k-1}}{\sqrt{2}\sqrt{1-4\mu^2}},\lambda_{k-1})\).
It then holds for every \(k \in \N\setminus\{0\}\)
\begin{align*}
 \Vert f_1 - m_{1,\mu,\lambda_k}(D)f_1 \Vert_{L^2}^2 + \Vert f_2 - m_{2,\mu,\lambda_k}(D)f_2 \Vert_{L^2}^2 + \Vert f_3 - m_{3,\mu,\lambda_k}(D)f_3 \Vert_{L^2}^2 \\
 \leq \left( \frac{30 C_0}{\gamma^{24 d}}\right)^k \max\{(\mu^{-2}\delta +\lambda^{-1}\beta)^{(1-\gamma)^k},\mu^{-2}\delta + \lambda^{-1}\beta\}.
\end{align*}

\end{lem}
\begin{proof}
 We prove the statement by induction on \(k\) with the induction basis given by \cref{lem:InductionBasis}.
 Assume that for some arbitrary but fixed \(k \in \N\) it holds
 \begin{align}
 \label{eq:induc_hyp}
  \begin{split}
  \Vert f_1 - m_{1,\mu,\lambda_k}(D) f_1 \Vert_{L^2}^2 & + \Vert f_2 - m_{2,\mu,\lambda_k}(D) f_2 \Vert_{L^2}^2 + \Vert f_3 - m_{3,\mu,\lambda_{k}}(D) f_3 \Vert_{L^2}^2 \\
  & \leq \left( \frac{30 C_0}{\gamma^{24 d}}\right)^k \max\{(\mu^{-2}\delta +\lambda^{-1}\beta)^{(1-\gamma)^k},\mu^{-2}\delta + \lambda^{-1}\beta\}.
  \end{split}
 \end{align}
Now for the induction step \(k \mapsto k+1\) we carry out the same argument as above to show
\begin{align}
\label{eq:induc_1}
\begin{split}
 \Vert f_1 - m_{1,\mu,\lambda_{k+1}}(D) f_1 \Vert_{L^2}^2 + \Vert f_2 - m_{2,\mu,\lambda_{k+1}}(D) f_2 \Vert_{L^2}^2 + \Vert f_3 - m_{3,\mu,\lambda_{k+1}}(D) f_3 \Vert_{L^2}^2 \\
 \leq \left( \frac{30 C_0}{\gamma^{24 d}}\right)^{k+1} \max\{(\mu^{-2}\delta +\lambda^{-1}\beta)^{(1-\gamma)^{k+1}},\mu^{-2}\delta + \lambda^{-1}\beta\}.
 \end{split}
\end{align}

We argue as in the induction basis and present the calculations only for the first term on the left hand side of \eqref{eq:induc_1}. By the triangle inequality,
\begin{align*}
 \Vert f_1 - m_{1,\mu,\lambda_{k+1}}(D) f_1 \Vert_{L^2}^2 \leq 2 \Vert f_1 - m_{1,\mu,\lambda_{k}}(D) f_1 \Vert_{L^2}^2 + 2 \Vert  m_{1,\mu,\lambda_k} (D) f_1 - m_{1,\mu,\lambda_{k+1}}(D) f_1 \Vert_{L^2}^2.
\end{align*}
The first contribution is already of the desired form. It thus remains to consider the second contribution. For this we consider an analogous argument as before:
 Let \(K_2^k := \{|k_2| > 4 d \mu \lambda_k\}, K_3^k := \{|k_3| > 4 d \mu \lambda_k\}\).
Then, since $ \chi_{K_j^k}(D) h_{1,j}(m_{j,\mu,\lambda_k}(D) f_j) = 0$ on $K_j^k$,
 \begin{align}
  \Vert  m_{1,\mu,\lambda_k} (D) f_1 & - m_{1,\mu,\lambda_{k+1}}(D) f_1 \Vert_{L^2}^2
   \leq  \Vert \chi_{K_2^k \cup K_3^k}(D) m_{1,\mu,\lambda_k}(D) f_1 \Vert_{L^2}^2 \nonumber \\
 & \leq  2 \sum_{j=2}^3 \Vert \chi_{K_j^k}(D) (m_{1,\mu,\lambda_k}(D) f_1 - h_{1,j}(m_{j,\mu,\lambda_k}(D) f_j)) \Vert_{L^2}^2 \nonumber \\
 & \leq  2 \sum_{j=2}^3 \Vert m_{1,\mu,\lambda_k}(D) f_1 - h_{1,j}(m_{j,\mu,\lambda_k}(D) f_j) \Vert_{L^2}^2 \nonumber \\
 & \leq  4 \sum\limits_{j=2}^{3} [ \Vert m_{1,\mu,\lambda_k}(D) f_1 - f_1 \Vert_{L^2}^2 + \Vert f_1 -  h_{1,j}(m_{j,\mu,\lambda_k}(D) f_j) \Vert_{L^2}^2] \nonumber \\
 & =  8 \Vert m_{1,\mu,\lambda_k}(D) f_1 - f_1 \Vert_{L^2}^2 + 4 \sum\limits_{j=2}^{3} \Vert f_1 -  h_{1,j}(m_{j,\mu,\lambda_k}(D) f_j) \Vert_{L^2}^2. \label{eq:induc_2}
 \end{align}

Again for the second right hand side contribution in \eqref{eq:induc_2} we use \cref{lem:ResultsRT21}:
\begin{align*}
 \Vert f_1 -  h_{1,j}(m_{j,\mu,\lambda_k}(D) f_j) \Vert_{L^2}^2 \leq \frac{C'^2}{\gamma^{24 d}} ( \Vert f_j - m_{j,\mu,\lambda_k}(D) f_j \Vert_{L^2}^2)^{1-\gamma}.
\end{align*}
Using the inductive hypothesis \eqref{eq:induc_hyp}, overall, we arrive at the following upper bound
\begin{align*}
  \Vert f_1 - m_{1,\mu,\lambda_{k+1}}(D) f_1 \Vert_{L^2}^2
  & \leq  2 \Vert f_1 - m_{1,\mu,\lambda_{k}}(D) f_1 \Vert_{L^2}^2 + 16 \Vert m_{1,\mu,\lambda_k}(D) f_1 - f_1 \Vert_{L^2}^2 \\
  & \quad + 8 \frac{C'^2}{\gamma^{24 d}} \sum_j ( \Vert f_j - m_{j,\mu,\lambda_k}(D) f_j \Vert_{L^2}^2)^{1-\gamma} \\
 & \leq 18 \Vert f_1 - m_{1,\mu,\lambda_{k}}(D) f_1 \Vert_{L^2}^2 \\
 & \quad + 16 \frac{C'^2}{\gamma^{24 d}} \big((\frac{30 C_0}{\gamma^{24 d}})^k \max\{(\mu^{-2}\delta + \lambda^{-1}\beta)^{(1-\gamma)^k},\mu^{-2}\delta + \lambda^{-1} \beta\}\big)^{1-\gamma}.
\end{align*}

Arguing symmetrically for \(f_2\) and \(f_3\) then yields:
\begin{align*}
& \Vert f_1 - m_{1,\mu,\lambda_{k+1}}(D) f_1 \Vert_{L^2}^2   + \Vert f_2 - m_{2,\mu,\lambda_{k+1}}(D) f_2 \Vert_{L^2}^2 + \Vert f_3 - m_{3,\mu,\lambda_{k+1}}(D) f_3 \Vert_{L^2}^2 \\
& \quad \leq  18 (\Vert f_1 - m_{1,\mu,\lambda_k}(D) f_1 \Vert_{L^2}^2 + \Vert f_2 - m_{2,\mu,\lambda_k}(D) f_2 \Vert_{L^2}^2 + \Vert f_3 - m_{3,\mu,\lambda_{k}}(D) f_3 \Vert_{L^2}^2) \\
 & \quad \quad + 48 \frac{C'^2}{\gamma^{24 d}} \big((\frac{30 C_0}{\gamma^{24 d}})^k \max\{(\mu^{-2}\delta + \lambda^{-1}\beta)^{(1-\gamma)^k},\mu^{-2}\delta + \lambda^{-1} \beta\}\big)^{1-\gamma}\\
\quad & \leq  18 \left( \frac{30 C_0}{\gamma^{24 d}}\right)^k \max\{(\mu^{-2}\delta +\lambda^{-1}\beta)^{(1-\gamma)^k},\mu^{-2}\delta + \lambda^{-1}\beta\} \\
 & \qquad + 48 \frac{C'^2}{\gamma^{24 d}} \big((\frac{30 C_0}{\gamma^{24 d}})^k \max\{(\mu^{-2}\delta + \lambda^{-1}\beta)^{(1-\gamma)^k},\mu^{-2}\delta + \lambda^{-1} \beta\}\big)^{1-\gamma}.
\end{align*}

Using that \(2C'^2 \leq C_0, 1-\gamma \in (0,1)\) and that \(\frac{C_0}{\gamma^{24 d}} \geq 3\),  leads to
\begin{align*}
 \Vert f_1 - m_{1,\mu,\lambda_{k+1}}(D) f_1 \Vert_{L^2}^2  & + \Vert f_2 - m_{2,\mu,\lambda_{k+1}}(D) f_2 \Vert_{L^2}^2 + \Vert f_3 - m_{3,\mu,\lambda_{k+1}}(D) f_3 \Vert_{L^2}^2 \\
 & \leq  18 \left( \frac{30 C_0}{\gamma^{24 d}}\right)^k \max\{(\mu^{-2}\delta +\lambda^{-1}\beta)^{(1-\gamma)^k},\mu^{-2}\delta + \lambda^{-1}\beta\} \\
 & \quad + 24 \frac{C_0}{\gamma^{24 d}} (\frac{30 C_0}{\gamma^{24 d}})^k \max\{(\mu^{-2}\delta + \lambda^{-1}\beta)^{(1-\gamma)^{k+1}},(\mu^{-2}\delta + \lambda^{-1} \beta)^{1-\gamma}\}\big) \\
& \leq  \frac{6C_0}{\gamma^{24 d}} (\frac{30 C_0}{\gamma^{24 d}})^{k} \max\{(\mu^{-2}\delta +\lambda^{-1}\beta)^{(1-\gamma)^{k+1}},\mu^{-2}\delta + \lambda^{-1}\beta\} \\
 & \quad + \frac{24 C_0}{\gamma^{24 d}}  (\frac{30 C_0}{\gamma^{24 d}})^{k} \max\{(\mu^{-2}\delta +\lambda^{-1}\beta)^{(1-\gamma)^{k+1}},\mu^{-2}\delta + \lambda^{-1}\beta\} \\
& =  (\frac{30 C_0}{\gamma^{24 d}})^{k+1} \max\{(\mu^{-2}\delta +\lambda^{-1}\beta)^{(1-\gamma)^{k+1}},\mu^{-2}\delta + \lambda^{-1}\beta\}.
\end{align*}
This concludes the proof.
\end{proof}

Now with the inductive procedure of reducing regions of Fourier space concentration in hand, we turn to the proof of the lower bound in \cref{thm:scaling_T3} and the argument for \cref{prop:rig_est}.

\begin{proof}[Proof of the lower bound in \cref{thm:scaling_T3} and proof of \cref{prop:rig_est}]

We argue in two steps, first fixing the free parameters and then exploiting the boundary conditions.

\emph{Step 1: Choice of parameters and proof of \cref{prop:rig_est}.}
We seek to invoke \cref{lem:iterate} expressing the bounds in terms of our energies, i.e. setting \(\delta = E_{el}^{per}(\chi;F):=\inf_{u \in \mathcal{D}_F} E_{el}^{per}(u,\chi), \beta = E_{surf}^{per}(\chi)\), and \(\lambda^{-1} = \mu^{-2} \epsilon\). Moreover, we use the notation
\begin{align}
\label{eq:en_total_periodic}
E_{\epsilon}^{per}(\chi;F):= E_{el}^{per}(\chi;F) + \epsilon E_{surf}^{per}(\chi).
\end{align}
By virtue of \cref{lem:FourierElastic} we obtain that for \(f_i := \chi_{i,i}\) indeed \(\sum_{j=1}^{3} \Vert \p_i f_i \Vert_{\dot{H}^{-1}}^2 \leq \delta\). It follows directly that also \(\sum_{i=1}^{3} \Vert \nabla f_i \Vert_{TV} \leq \beta\) and therefore the conditions in \cref{lem:ResultsRT21} are fulfilled. As a consequence, the above iteration in \cref{lem:iterate} is applicable.
With this in mind, we choose $\mu = \epsilon^{\alpha}$ for some $\alpha>0$ to be specified. Therefore,
\begin{align*}
 \lambda = \lambda_0 = \epsilon^{2\alpha-1}
\end{align*}
and thus since without loss of generality \(E_{\epsilon}^{per}(\chi;F) := \inf_{u \in \mathcal{D}_F} E_{el}^{per}(u,\chi) + \epsilon E_{surf}^{per}(\chi) \leq 1\) (having the upper bound from \cref{prop:upper} in mind), we deduce
\begin{align} \label{eq:ParameterOptimization}
 \sum_{j=1}^3 \Vert f_j - m_{j,\mu,\lambda_k}(D)f_j \Vert_{L^2}^2 \leq (\frac{30 C_0}{\gamma^{24d}})^{k} \epsilon^{-2\alpha} E_{\epsilon}^{per}(\chi;F)^{(1-\gamma)^k}.
\end{align}
We further choose
\begin{align*}
 \lambda_k = (2\sqrt{2d^2+1}\mu)^k \lambda_0 = M^k \epsilon^{(2+k)\alpha-1},
\end{align*}
where \(M=M(d) = 2\sqrt{2d^2+1} > 2\). This is admissible in the sense of the assumptions in \cref{lem:FirstComparison} since \(\frac{\lambda_k}{\lambda_{k-1}} = 2\sqrt{2d^2+1}\mu \in (\frac{1}{\sqrt{2}}\frac{4d\mu}{\sqrt{1-4\mu^2}},1)\).

Next, for \(\alpha = \alpha(\epsilon) \in (0,1)\) and \(\epsilon >0\) we choose \(k \in \N\) to be given by
\begin{align*}
 k := \left\lceil \frac{(1-2\alpha)|\log \epsilon| + \log 2}{\alpha |\log \epsilon| - \log M} \right\rceil \leq \left\lceil \frac{1+\frac{\log M}{|\log \epsilon|}}{\alpha - \frac{\log M}{|\log \epsilon|}} \right\rceil.
\end{align*}
This ensures that \(\lambda_k \leq \frac{1}{2}\).
In what follows, we will choose the parameters $\epsilon, \alpha$ such that \(\frac{\log M}{|\log \epsilon|}\leq \frac{\alpha}{2}\) which implies \(k \leq \frac{4}{\alpha}\). 
Exploiting the discrete Fourier transform, then yields
\begin{align*}
  \Vert f_j - m_{j,\mu,\lambda_k}(D)f_j \Vert_{L^2}^2 =  \sum_{\xi \in \Z^3 \setminus \{0\}} | \hat{f}_j(\xi)|^2 = \Vert f_j -  \langle f_j \rangle \Vert_{L^2}^2,
\end{align*}
and hence, by \eqref{eq:ParameterOptimization} results in the estimate
\begin{align*}
 \sum_{j=1}^3 \Vert f_j - \langle f_j \rangle \Vert_{L^2}^2 \leq \big(\frac{30 C_0}{\gamma^{24d}}\big)^{\frac{4}{\alpha}} \epsilon^{-2\alpha} E_{\epsilon}^{per}(\chi;F)^{(1-\gamma)^{\frac{4}{\alpha}}}.
\end{align*}

Using that \((1-\gamma)^{\frac{4}{\alpha}} \geq 1 - \frac{4}{\alpha}\gamma\), we set $\gamma := \frac{\alpha}{8} \in (0,1)$ which leads to the bound
\begin{align*}
 \sum_{j=1}^3 \Vert f_j -  \langle f_j \rangle \Vert_{L^2}^2 \leq \big( \frac{30C_0 C^{24d}}{(1-q)^{24d}} \big)^{\frac{C}{\alpha}} \alpha^{-\frac{24dC}{\alpha}} \epsilon^{-2\alpha} E_{\epsilon}^{per}(\chi;F)^{\frac{1}{2}}.
\end{align*}

Next, we fix the parameter $\alpha>0$: Observing that for any $\nu>0$ there exists $\alpha_0>0$ such that for $\alpha \in (0,\alpha_0)$ it holds that
\begin{align}
\label{eq:equiv}
 \alpha^{-\frac{24dC}{\alpha}} = \exp\big(24dC \log(\alpha^{-1}) \alpha^{-1}\big) \leq \exp\big(\frac{24dC}{\nu e} \alpha^{-1-\nu}\big) = \exp(C(\nu) \alpha^{-1-\nu}),
\end{align}
we choose \(\alpha = |\log\epsilon|^{-\frac{1}{2+\nu}}\) for all \(\epsilon \in (0,\epsilon_0)\) with $\epsilon_0>0$ still to be chosen. In particular, for $\epsilon_0>0$ sufficiently small, such that $|\log \epsilon|^{\frac{1}{2}} \geq 2 \log M$, \eqref{eq:equiv} holds and also  \(\frac{\log M}{|\log \epsilon|}\leq \frac{\alpha}{2}\) holds as required above.

As a consequence, for \(\nu':=\frac{\nu}{4+2\nu} \in (0,\frac{1}{2})\) we arrive at
\begin{align*}
  \sum_{j=1}^3 \Vert f_j -  \langle f_j \rangle \Vert_{L^2}^2 \leq \exp(c_{\nu'} |\log \epsilon|^{\frac{1}{2}+\nu'}) E_{\epsilon}^{per}(\chi;F)^{\frac{1}{2}} .
\end{align*}

\emph{Step 2: Conclusion.}
In order to conclude the estimate, we derive a lower bound for \(\sum_{j=1}^3 \Vert f_j -  \langle f_j \rangle \Vert_{L^2}^2\).
For this we recall that \(f_j = \chi_{j,j}\) is the \(j\)-th diagonal entry of the phase indicator \(\chi\) and hence \(\sum_{j=1}^3 \Vert f_j -  \langle f_j \rangle \Vert_{L^2}^2 = \Vert \chi -  \langle \chi \rangle \Vert_{L^2}^2\).
Thus, by the mean value condition in $\mathcal{D}_F^{per}$,
\begin{align*}
 | \langle \chi \rangle -F|^2 = | \langle \chi \rangle -\langle u \rangle|^2 \leq \int_{\T^3} |u-\chi|^2 dx \leq E_{el}^{per}(u,\chi),
\end{align*}
and, furthermore, as the left hand side is independent of \(u\),
\begin{align*}
 | \langle \chi \rangle -F|^2 \leq E_{el}^{per}(\chi;F).
\end{align*}

Overall this implies for the total energy \(E_\epsilon^{per}(\chi;F) = E_{el}^{per}(\chi;F) + \epsilon E_{surf}^{per}(\chi)\)
\begin{align*}
 \dist^2(F,\K) & \leq \int_{\T^3} |\chi-F|^2 dx \leq \int_{\T^3} |\chi- \langle \chi \rangle |^2 dx + \int_{\T^3} | \langle \chi \rangle -F|^2 dx \\
 & \leq \exp(c_{\nu'} |\log \epsilon|^{\frac{1}{2}+\nu'}) E_\epsilon^{per}(\chi;F)^{\frac{1}{2}} + E_{el}^{per}(\chi;F)\\
 & \leq \exp(c_{\nu'} |\log \epsilon|^{\frac{1}{2}+\nu'}) E_\epsilon^{per}(\chi;F)^{\frac{1}{2}} + E_{\epsilon}^{per}(\chi;F)^{\frac{1}{2}} \\
 & \leq 2 \exp(c_{\nu'} |\log \epsilon|^{\frac{1}{2}+\nu'}) E_\epsilon^{per}(\chi;F)^{\frac{1}{2}}.
\end{align*}

Finally, solving for \(E_\epsilon^{per}(\chi;F)\) shows the desired estimate
\begin{align*}
 E_\epsilon^{per}(\chi;F) \geq 2^{-2} \exp\big(-2c_{\nu'} |\log \epsilon|^{\frac{1}{2}+\nu'}\big) \dist^4(F,\K).
\end{align*}
The desired claim follows by an application of \cref{lem:per}.
\end{proof}

\appendix

\section{Branching, upper bound for the two-state problem for the divergence operator}
\label{sec:construc}

We complement our lower bounds for the compatible two-well problem from \cref{sec:two-state} by an upper bound in the case of the divergence operator acting on matrix fields as introduced in \cref{ex:div}. For simplicity, we only consider square matrices, i.e. \(m = d\). For earlier, closely related, three-dimensional constructions in the context of compliance minimization problems we refer to \cite{PW21}.
While the $\epsilon^{\frac{2}{3}}$ construction is by now rather ``standard'' \cite{KM92,KM94,CC15,OV10}, our argument does provide a slightly different perspective, in that, in arbitrary dimension, we can ensure boundary conditions on \emph{all} faces of the domain $\Omega=[0,1]^d$ (see the upper bound construction for \cite[Theorem 3]{RT22} for a similar construction for the gradient).

\begin{figure}
 \centering
\begin{tikzpicture}[thick, scale = 0.5]

\draw[fill = red!20] (0,0) -- (1.5,0) -- (1.5,7) -- (0.75,7) -- (0.75,10.5) -- (0,10.5) -- cycle;
  \draw[fill = red!20] (2.5,7) -- (2.5,10.5) -- (3.25,10.5) -- (3.25,7) -- cycle;
  \draw[fill = blue!20] (1.5,0) -- (1.5,7) -- (2,10.5) -- (2.5,10.5) -- (2.5,7) -- cycle;
  \draw[fill = blue!20] (0.75,7) -- (0.75,10.5) -- (1.25,10.5) -- cycle;
  \draw[fill = blue!20] (3.25,7) -- (3.25,10.5) -- (3.75,10.5) -- cycle;
  \draw[fill = blue!20] (3,0) -- (4.5,10.5) -- (5,10.5) -- (5,0) -- cycle;
  \draw[fill = orange!20] (0.75,7) -- (1.25,10.5) -- (2,10.5) -- (1.5,7) -- cycle;
  \draw[fill = orange!20] (1.5,0) -- (2.5,7) -- (3.25,7) -- (3.75,10.5) -- (4.5,10.5) -- (3,0) -- cycle;

 \draw (0,0) -- (5,0) -- (5,7) -- (0,7) -- (0,0);
 \draw[dashed] (2.5,0) -- (2.5,7);
 \draw (1.5,0) -- (1.5,7);
 \draw (1.5,0) -- (2.5,7);
 \draw (3,0) -- (4,7);

 \draw (0,7) -- (0,10.5) -- (5,10.5) -- (5,7);
 \draw (2.5,7) -- (2.5,10.5);
 \draw[dashed] (1.25,7) -- (1.25,10.5);
 \draw[dashed] (3.75,7) -- (3.75,10.5);
 \draw (0.75,7) -- (0.75,10.5);
 \draw (0.75,7) -- (1.25,10.5);
 \draw (3.25,7) -- (3.25,10.5);
 \draw (3.25,7) -- (3.75,10.5);
 \draw (1.5,7) -- (2,10.5);
 \draw (4,7) -- (4.5,10.5);

 \draw (-0.2,0) -- (-0.2,7);
 \draw (-0.25,0) -- (-0.15,0);
 \draw (-0.25,7) -- (-0.15,7);
 \node[left] at (-0.25,3.5) {$h$};

 \draw (0,-0.2) -- (5,-0.2);
 \draw (0,-0.25) -- (0,-0.15);
 \draw (5,-0.25) -- (5,-0.15);
 \node[below] at (2.5,-0.2) {$l$};

 \node at (0.75,3.5) {$A$};
 \node at (2.75,3.5) {$\tilde{A}$};
 \node at (4.25,3.5) {$B$};

 \draw[->] (3.25,1.75) -- (4.25,1.607) node[below] {$n$};
\end{tikzpicture}
\caption{Self-similar construction in \cref{prop:BranchingUpper}.}
\end{figure}
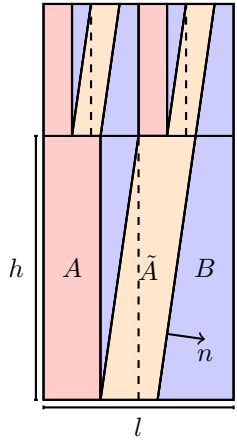

In deducing the upper bound for the divergence operator, we first provide a construction in a unit cell (\cref{lem:BranchingUnitCell}) and iterate this construction (\cref{prop:BranchingUpper}). This yields a construction which attains the boundary data in two directions. In order to attain these also on the remaining sides we use the flexibility of the wave cone for the divergence operator (\cref{prop:d_dim}). We remark that in two dimensions there would be no modification with respect to the gradient construction since there the curl and divergence only differ by a rotation of $90$ degrees.

In our unit cell branching construction, we do not work on the level of the potential, but directly consider the problem on the level of the wells. In this context, we recall the compatibility conditions for laminates formed by the divergence operator which is determined by the associated wave cone: For $M\in \R^{d\times d}$, we have that \(M \in \ker \AA(\xi)\) if and only it holds \(M \xi = 0\).

With this in hand, we introduce an auxiliary matrix which will play an important role in our construction:
Let \(A,B \in \R^{d \times d}\) be such that \((B-A)e_1 =0\) and let \(n = e_1 + \gamma_2 \nu\) for a unit vector \(\nu\) perpendicular to \( e_1\) and for some \(\gamma_2 \in \R, \ \gamma_2 \neq 0\).
We then define \(E_{\nu}\) by
 \begin{align} \label{eq:PerturbMatrix}
  E_{\nu} = \gamma_2 (B-A)\nu \otimes e_1
 \end{align}
and the associated ``perturbed'' matrix \(\tilde{A}_\nu = A + E_{\nu}\).
By construction, this matrix obeys the identities
\begin{align}\label{eq:PerturbMatrixProp}
 (\tilde{A}_\nu - A) \nu = 0, \ (B-\tilde{A}_\nu) n = 0.
\end{align}
This in particular allows for interfaces of \(\tilde{A}_\nu\) and \(A\) with normal \(\nu\) and of \(\tilde{A}_\nu\) and \(B\) with normal \(n\) which we will use in our branching construction below.

\begin{lem} \label{lem:BranchingUnitCell}
 For \(0<l<h \leq 1\), we define \(\omega = [0,l] \times [0,h] \times [0,1]^{d-2}\). Let \(A,B \in \R^{d \times d}\) be such that \((B - A)e_1 = 0\) and let \(F_\lambda = \lambda A + (1-\lambda)B\) for some \(\lambda \in (0,1)\). Then there exists \(u: \R^d \to \R^{d\times d}\) such that
 \begin{align*}
  \di u & =  0 \text{ in } \R^d ,\\
  u & = F_\lambda \text{ for } x_1 \in (-\infty,0) \cup (l,\infty).
  \end{align*}
  Furthermore, there exist \(\chi \in BV(\omega;\{A,B\})\) and a constant \(C = C(A,B) > 0\) such that for any \(\epsilon >0\) the localized energy can be bounded by
  \begin{align*}
  E_{\epsilon}(u,\chi;\omega) := \int_\omega |u-\chi|^2 dx + \epsilon \int_\omega |\nabla \chi| & \leq C (1-\lambda)^2\frac{l^3}{h} + 5 \epsilon h.
 \end{align*}
\end{lem}

\begin{proof}
 We consider the following partition of the domain $\omega$ into subdomains:
 \begin{gather*}
 \begin{aligned}
  \omega_1 & = \{x_1 \in (0,\frac{\lambda l}{2})\}, & \omega_2 & = \{x_1 \in (\frac{\lambda l}{2}, \frac{\lambda l}{2} + \frac{(1-\lambda)l}{2h} x_2)\}, \\
  \omega_3 & = \{x_1 \in (\frac{\lambda l}{2} + \frac{(1-\lambda)l}{2h}x_2, \lambda l +\frac{(1-\lambda)l}{2h} x_2)\}, & \omega_4 & = \{x_1 \in (\lambda l + \frac{(1-\lambda)l}{2h}x_2,l)\}.
  \end{aligned}
 \end{gather*}
Based on this we define
\begin{align*}
 u(x) & = \begin{cases}
         A & x \in \omega_1, \\
         B & x \in \omega_2 \cup \omega_4, \\
         A+E_{e_2} & x \in \omega_3,
        \end{cases} \quad
        \chi(x) = \begin{cases}
        A & x \in \omega_1 \cup \omega_3, \\
       B & x \in \omega_2 \cup \omega_4,             
                   \end{cases}
\end{align*}
where \(E_{e_2}\) is given in \eqref{eq:PerturbMatrix} for \(n = e_1 - \frac{(1-\lambda)l}{2h} e_2\).
We highlight that \(u\) is independent of \(x_k\) for \(k \geq 3\).
By definition of \(E_{e_2}\), the characterization of the wave cone  \eqref{eq:wave_cone} for the divergence operator and the remarks on laminates in \eqref{eq:PerturbMatrixProp}, this defines an divergence-free mapping.
Further, as \((B- F_\lambda)e_1 = (A-F_\lambda)e_1 = 0\), the exterior data are attained in $x_1 \in (-\infty,0)\cup (l,\infty)$.

To calculate the energy, we observe, that the only contribution to the elastic energy is given in \(\omega_3\). Hence,
\begin{align*}
 E_{el}(u,\chi;\omega) & := \int_{\omega} |u-\chi|^2 dx = \int_{\omega_3} |A+E_{e_2} - A|^2 dx = |E_{e_2}|^2 |\omega_3| = |(B-A)e_2|^2\frac{(1-\lambda)^2 \lambda l^3}{8 h}.
\end{align*}
As the surface energy is determined by the interfaces between \(\omega_j\) and \(\omega_k\), we obtain (\(l < h\))
\begin{align*}
 E_{surf}(\chi;\omega) := \int_{\omega}|\nabla \chi| = h + 2 \sqrt{\frac{(1-\lambda)^2 l^2}{4} + h^2} \leq h + 2 h \sqrt{\frac{(1-\lambda)^2}{4} + 1} \leq 5 h.
\end{align*}
This shows the claim.
\end{proof}

As for analogous constructions for the gradient, we will use this unit cell as a building block in order to achieve a self-similar construction attaining the boundary data on two directions.
For this to be admissible in the sense of an \(\A\)-free map, we rely on the following lemma.
It shows that for first order operators corners in which interfaces meet do not give rise to singularities.

\begin{lem} \label{lem:Corners}
 Let \(\A(D) = \sum_{j=1}^d A_j \p_j: C^\infty(\R^d;\R^n) \to C^\infty(\R^d;\R^m)\) be a homogeneous, linear operator of degree one with symbol \(\AA\) given in \eqref{eq:Symbol} and let \(\Omega_j\), \(j=1,\dots,l\), be a polygonal set (the set is defined as the intersection of half spaces) with outer unit normal \(n_j\)  such that
\begin{itemize} 
\item \(\R^d = \bigcup_{j=1}^l \Omega_j\), 
\item the two sets \(\Omega_j, \Omega_{j+1}\) have one common face (\(l+1 = 1\)), 
\item and such that they meet in one point, i.e. \(\bigcap_{j=1}^l \Omega_j = \{x_0\}\), cf. \cref{fig:GeometryCorners}.
\end{itemize}
Assume further that \(B_j \in \R^n\), \(j=1,\dots,l\), are such that \( B_j - B_{j+1} \in \ker \AA(n_j)\).
Then the map 
\begin{align*}
 u(x) = B_j \text{ for } x \in \Omega_j
\end{align*}
is \(\A\)-free.
\end{lem}

\begin{proof}
\begin{figure}
 \begin{tikzpicture}
  \foreach \x in {0,30,75,190,270}
    \draw (0,0) -- (\x:2cm);
  \fill (0,0) circle (2pt) node[above left] {$x_0$};
  \foreach \x/\y in {15/{\Omega_1},52.5/{\Omega_2},132.5/{\iddots},230/{\Omega_{l-1}},315/{\Omega_l}} 
    \node at (\x:1cm) {$\y$};
  \draw[->] (0,-1) -- (0.5,-1) node[below] {$n_{l-1}$};
 \end{tikzpicture}
\caption{Setting of \cref{lem:Corners}}
\label{fig:GeometryCorners}
\end{figure}

 First we note that \(u\) is indeed well-defined by the properties of \(\Omega_j\), further we notice that for \(M \in \R^n\) and \(U \subset \R^d\) a Lipschitz domain it holds for \(\varphi \in C^\infty_c(\R^d;\R^m)\)
 \begin{align*}
  \int\limits_U M \cdot (\A(D))^\ast  \varphi dx & = \sum_{k=1}^{d} \int\limits_U M \cdot \p_k(A_k^t \varphi) dx = \sum_{k=1}^d \int\limits_{\p U} M \cdot (A_k^t \varphi) n_k d\mathcal{H}^{d-1} = \int\limits_{\p U} \AA(n)M \cdot \varphi d\mathcal{H}^{d-1}.
 \end{align*}
With this it holds
\begin{align*}
 \int\limits_{\R^d} u \cdot (\A(D))^\ast \varphi dx = \sum_{j=1}^l \int\limits_{\Omega_j} B_j \cdot (\A(D))^\ast \varphi dx = \sum_{j=1}^l \int\limits_{\p \Omega_j} \AA(n_j) B_j \cdot \varphi d\mathcal{H}^{d-1}.
\end{align*}
Moreover on \(\p \Omega_j \cap \p \Omega_{j+1}\) it holds that \(n_{j+1} = - n_j \) and thus by the assumptions on \(B_j\)
\begin{align*}
 \int\limits_{\R^d} u \cdot (\A(D))^\ast \varphi dx = \sum_{j=1}^l \int\limits_{\p \Omega_j \cap \p \Omega_{j+1}} \AA(n_j)(B_j - B_{j+1}) \cdot \varphi d\mathcal{H}^{d-1} = 0.
\end{align*}
As \(\varphi\) was arbitrary the claim follows.
\end{proof}

With \cref{lem:Corners} in hand, we now iterate the unit cell-construction from \cref{lem:BranchingUnitCell}.

\begin{prop} \label{prop:BranchingUpper}
Let \(d,n \in \N\).
Let \(\Omega = [0,1]^d \), let \(A,B \in \R^n\) be such that \(B-A \in \Lambda_{\di}\), cf. \eqref{eq:wave_cone}, and let \(F_\lambda = \lambda A + (1-\lambda) B\) for some \(\lambda \in (0,1)\).
Let \(E_\epsilon\) be as in \eqref{eq:E-Total}. 
Then there exist \(u: \R^d \to \R^{d \times d}\) and \(\chi \in BV(\Omega;\{A,B\})\) with \(\di u = 0\) in \(\R^d\) and \(u = F_\lambda\) for \((x_1,x_2) \notin [0,1]^2\) such that for any \(\epsilon \in (0,1)\) and any \(N \in \N\)
  \begin{align*}
   E_{\epsilon}(u,\chi) \leq C ( \frac{1}{N^2} + \epsilon N)
  \end{align*}
  for some constant \(C = C(A,B,\lambda) > 0\).
\end{prop}

\begin{proof}
Without loss of generality we may assume \((B-A)e_1 = 0\), i.e. \(B-A \in \ker \AA(e_1)\) for \(\A(D) = \di\).
 Let \(\theta \in (\frac{1}{4},\frac{1}{2})\).
 We argue symmetrically in the upper and lower half of the cube, i.e. we give the construction of \(u\) on \([0,1] \times [\frac{1}{2},1] \times [0,1]^{d-2}\) and define $u$ on the lower half by symmetry.
 We define for \(N \in \N\) and for \(j \in \N_0\)
 \begin{align*}
  y_j = 1 - \frac{\theta^j}{2}, l_j = \frac{1}{2^jN}, h_j = y_{j+1}-y_j = \theta^j \frac{1-\theta}{2}.
 \end{align*}
Furthermore, let \(j_0 \in \N\) be the maximal \(j \in \N\) such that \(l_j < h_j\).
We set
\begin{align*}
 \omega_{j,k} = \Big((kl_j,y_j) + [0,l_j] \times [0,h_j] \Big)\times [0,1]^{d-2}, 
                 \end{align*}
                 for \(k \in \{0,1,\dots,N2^j-1\}, j \in \{0,1,\dots,j_0\}\); for \(k \in\{0,1,\dots,N 2^{j_0}-1\}, j = j_0+1\) we set
                 \begin{align*}
                 \omega_{j_0+1,k} = \Big((kl_{j_0},y_{j_0+1}) + [0,l_{j_0}] \times [0,\frac{\theta^{j_0+1}}{2}] \Big)\times [0,1]^{d-2}.
\end{align*}
Let \(u_j,\chi_j\) in \([0,l_j] \times [0,h_j] \times [0,1]^{d-2}\) be given by \cref{lem:BranchingUnitCell} for \(j=1,\dots,j_0\).
Further, in the layer \(j=j_0+1\) we interpolate with the desired boundary data by a cut-off argument:
To this end, we introduce the cut-off function \(\phi: [0,\infty) \to [0,1]\) and the profile \(h:[0,l_{j_0}] \to [0,\infty)\) by setting 
\begin{align*}
 \phi(t) = \begin{cases} 1 & t \in [0,\frac{1}{2}], \\
            -4t + 3 & t \in (\frac{1}{2},\frac{3}{4}), \\
            0 & t \geq \frac{3}{4},
           \end{cases} \quad
           h(t) = \begin{cases} (1-\lambda)t & t \in [0,\lambda), \\
                   \lambda (1-t) & t \in [\lambda, 1].
                  \end{cases}
\end{align*}
We consider the function \(\tilde{u}_{j_0+1}: [0,l_{j_0}] \times [0,\frac{\theta^{j_0+1}}{2}] \times [0,1]^{d-2}\) defined via
\begin{align*}
 \tilde{u}_{j_0+1}(x) = - \frac{2 l_{j_0+1}}{\theta^{j_0+1}} \phi'\big(\frac{2 x_2}{\theta^{j_0+1}}\big) h\Big(\frac{x_1}{l_{j_0+1}}\Big) ((A-B)e_2) \otimes e_1 + \phi\big(\frac{2 x_2}{\theta^{j_0+1}}\big) h'\Big(\frac{x_1}{l_{j_0+1}}\Big) (A-B) + F_\lambda.
\end{align*}
The associated phase indicator is defined by
\begin{align*}
 \chi_{j_0+1}(x) = \chi_{(0,\lambda l_{j_0})}(x_1) A + \chi_{(\lambda l_{j_0},l_{j_0})}(x_1) B.
\end{align*}
We note that \(\chi_{j_0+1}(x) = h'(\frac{x_1}{l_{j_0+1}}) (A-B) + F_\lambda\) and moreover for \(x_2 <  \frac{1}{2}\) it holds \(\tilde{u}_{j_0+1}(x) = \chi_{j_0+1}(x) \) and for \(x_2 > \frac{3}{4}\) correspondingly \(\tilde{u}_{j_0+1}(x) = F_\lambda\).
Furthermore, for \(x_1 \in \{0,l_{j_0+1}\}\), we know \(h(\frac{x_1}{l_{j_0+1}}) = 0\) and thus \((\tilde{u}_{j_0+1}(x) - F_\lambda) e_1 = \phi(\frac{2 x_2}{\theta^{j_0+1}}) h'\Big(\frac{x_1}{l_{j_0+1}}\Big) (A-B)e_1 =0 \).

With the help of this construction we meet the prescribed data for \(x_2 \geq 1\), and we can define \(u\) in the upper half of the full cube:
\begin{align*}
 u(x) = \begin{cases}
         u_j(x-(kl_j,y_j)) & x \in \omega_{j,k}, \\
         \tilde{u}_{j_0+1}(x - (kl_{j_0},y_{j_0+1})) & x \in \omega_{{j_0+1},k}.
        \end{cases}
\end{align*}
For the lower half of the cube we argue similarly, mirroring the unit cell construction of \cref{lem:BranchingUnitCell}, i.e. instead of \(E_{e_2}\) we consider \(E_{-e_2}\).
We define \(\chi\) in \([0,1]^d\) analogously.

We note, that this defines a divergence free mapping, as all the laminations are in compatible directions as \((B-A)e_1 = 0\) and by the choice of \(E_{e_2}\) in \eqref{eq:PerturbMatrix}.
\Cref{lem:Corners} shows, that we are divergence-free even thought interfaces meet in corners.
Moreover, we can bound the energy in the \(\omega_{j_0+1,k}\) cells for any \(k \in \{1,\dots,N2^{j_0}-1\}\):
\begin{align*}
 |\tilde{u}_{j_0+1}(x) - \chi_{j_0+1}(x)|^2  & = \Big|\frac{2 l_{j_0+1}}{\theta^{j_0+1}} \phi'\big(\frac{2x_2}{\theta^{j_0+1}}\big) h\Big(\frac{x_1}{l_{j_0+1}}\Big) ((A-B)e_2 \otimes e_1) \\
 & \qquad - (\phi\big(\frac{2x_2}{\theta^{j_0+1}}\big) - 1) h'\Big(\frac{x_1}{l_{j_0+1}}\Big) (A-B) \Big|^2\\ 
 &\leq C(A,B,\lambda)\Big( \frac{l_{j_0+1}^2}{\theta^{2j_0+2}} \phi'^2\big(\frac{2x_2}{\theta^{j_0+1}}\big) + 1 \Big),
\end{align*}
and, since \(l_{j_0+1} \geq h_{j_0+1}\) and \(\theta^{j_0+1} \sim h_{j_0+1}\),
\begin{align*}
 \int_{[0,l_{j_0+1}] \times [0,\frac{\theta^{j_0+1}}{2}] \times [0,1]^{d-2}} |\tilde{u}_{j_0+1} - \chi_{j_0+1}(x)|^2  dx & \leq C l_{j_0+1} (\int_0^1 \frac{l_{j_0+1}^2}{\theta^{j_0+1}} \phi'^2(t) dt + \theta^{j_0+1}) \\
 & \leq C ( \frac{l_{j_0+1}^3}{\theta^{j_0+1}} + l_{j_0+1} \theta^{j_0+1} ) \\
 & \leq C \frac{l_{j_0+1}^3}{h_{j_0+1}}.
\end{align*}
Furthermore, the surface energy is bounded by \(E_{surf}(\chi_{j_0+1};\omega_{j_0+1,k}) \leq C(A,B,\lambda) h_{j_0+1}\).

Overall, we have a function defined on \([0,1]^d\) and can extend it to be \(F_\lambda\) for \((x_1,x_2) \notin [0,1]^2\). The energy then can be bounded by
\begin{align*}
 E_{\epsilon}(u,\chi) & \leq 2 \sum_{j=0}^{j_0+1} \sum_{k=0}^{N2^j} E_{\epsilon}(u_j,\chi_j;\omega_{j,k}) \leq C \sum_{j=0}^{j_0+1} N 2^j ( \frac{l_j^3}{h_j} + \epsilon h_j) \\
 & \leq C \sum_{j=0}^{j_0+1} ( \frac{l_j^2}{h_j} + \epsilon \frac{h_j}{l_j}) = C \sum_{j=0}^{j_0+1} \frac{1}{N^2} (\frac{1}{4\theta})^j + \epsilon N (2\theta)^j \\
 & \leq C ( \frac{1}{N^2} + \epsilon N).
\end{align*}

\end{proof}

\begin{rmk}
We remark that in the situation of the divergence operator, there are situations with substantially more flexibility than for the gradient: If for the two wells $A,B\in \R^{d\times d}$ it does not only hold that \((B-A)e_1 = 0 \) but also that \((B-A)e_2=0\), there would not be any elastic energy contribution involved. In this situation, for the above construction, we would only have contributions to the surface energy, as then \(E_{e_2} = 0\) and hence \(A+E_{e_2} = A \in \{A,B\}\). In particular, for boundary data which are only attained on two directions this would yield a linear scaling law in $\epsilon$.
For curl free mappings as in gradient inclusions, this is not possible, as the direction of lamination is unique in that case, i.e. \(V_{\operatorname{rot},\lambda}\) is at most one-dimensional.
\end{rmk}

As a last auxiliary step towards the upper bound construction from \cref{thm:TwoWell}, in order to achieve the exterior data on \emph{all} sides of the unit cube, we adapt the branching construction similarly as in \cite{RT21}, as the construction from \cref{prop:BranchingUpper} does not yet satisfy \(F_\lambda\) at, e.g., \(x_3 =0\). Thus, we combine \cref{prop:BranchingUpper} with a further domain splitting for which we split \([0,1]^d\) into different regions. In each region, we prescribe a different direction for the branching construction from \cref{prop:BranchingUpper}.
 To this end, we use that the choice of \(e_2\) in the above results was arbitrary and we also can choose any other direction \(e_j\) for \(j \in \{2,\dots,d\}\). Combined with compatibility conditions at the resulting interfaces, this will allow us to deduce the desired branching construction.

 \begin{prop}
 \label{prop:d_dim}
 Under the same assumptions as in \cref{prop:BranchingUpper} there exist \(u: \R^d \to \R^{d \times d}\), \(\chi \in BV(\Omega;\{A,B\})\) and a constant \(C = C(A,B,\lambda) > 0\) such that \(u \in \mathcal{D}_{F_\lambda}\) for \(F_\lambda = \lambda A + (1-\lambda)B\) (\(\lambda \in (0,1)\)) and for any \(\epsilon \in (0,1)\) and \(N \in \N\) it holds
 \begin{align*}
  E_\epsilon(u,\chi) \leq C (\frac{1}{N^2} + \epsilon N).
 \end{align*}
\end{prop}

\begin{proof}
For simplicity, we first carry out the details for the case \(d=3\) and then only comment on the changes in the case of arbitrary dimension.
 We split \([0,1]^3\) into the following four parts and use different branching directions in each part: Let
 \begin{align*}
  \Omega_2^{\pm} &= \{x \in [0,1]^3 : \pm (x_2-\frac{1}{2}) \geq 0,\ \frac{1}{2}-|x_2-\frac{1}{2}| \leq \frac{1}{2}-|x_3-\frac{1}{2}|\}, \\
  \Omega_3^{\pm} &= \{x \in [0,1]^3 : \pm (x_3 - \frac{1}{2}) \geq 0,\ \frac{1}{2}-|x_3-\frac{1}{2}| \leq \frac{1}{2}-|x_2-\frac{1}{2}|\},
 \end{align*}
and consider the upper (\(\Omega_j^+\)) and lower (\(\Omega_j^-\)) halves separately as in the proof of \cref{prop:BranchingUpper}.

Next, we define \(u_2\) in \(\Omega_2^+\) and \(u_3\) in \(\Omega_3^+\) using \cref{prop:BranchingUpper}:
 The function \(u_2\) is given by the function from \cref{prop:BranchingUpper} above, whereas \(u_3\) is obtained from \(u_2\) by exchanging roles of \(e_2\) and \(e_3\), i.e. the branching is done in \(e_3\) direction and we use \(E_{e_3}\) instead of \(E_{e_2}\). For 
 \(\nu \in \{e_2, e_3\}\) the error matrix \(E_{\nu}\) is given in \eqref{eq:PerturbMatrix}.
 We then define the overall deformation \(u\) in $\Omega_2^+ \cup \Omega_3^+$ by
 \begin{align*}
  u(x) = \begin{cases}
          u_2(x) & x \in \Omega_2^{+},\\
          u_3(x) & x \in \Omega_3^{+}.
         \end{cases}
 \end{align*}
 This construction is depicted in \cref{fig:3dBranching}. As above $u$ is defined in the lower halves $\Omega_2^{-},\Omega_3^{-}$ by symmetry.

 We claim that this overall construction is divergence-free. In the individual regions $\Omega_2^+$ and $\Omega_3^+$ this follows by \cref{prop:BranchingUpper}. It thus remains to discuss the compatibility at the interface $x_2=x_3$. Since all other values of $u$ are given by (matching domains in which) $u\in \{A,B\}$, it suffices to discuss the compatibility of the error matrices $E_{e_2}$ and $E_{e_3}$ at this interface. To this end, we however note that \((E_{e_3}-E_{e_2}) \zeta = 0\) for all \(\zeta \in \vspan(e_2,e_3)\), i.e. also this  interface is admissible. This shows that \(u\) indeed defines a divergence free map.

The upper bound for the elastic and surface energies from \cref{prop:BranchingUpper} remains valid, thus yielding the claimed estimate which concludes the proof of the proposition.

In order to show the \(d\)-dimensional result, we split \([0,1]^d\) into \(2d-2\) regions \(\Omega_j^{\pm} := \{x \in [0,1]^d : \pm (x_j-\frac{1}{2}) \geq 0,\frac{1}{2}-|x_j-\frac{1}{2}| = \min_{2 \leq k \leq d} \frac{1}{2} - |x_k - \frac{1}{2}|\}\) for \(j = 2,\dots,d\) and argue as above.
\end{proof}

Finally, with \cref{prop:d_dim} in hand, we immediately obtain the proof of the upper bound construction from \cref{thm:TwoWell} for the divergence operator.

\begin{proof}[Proof of the upper bound in \cref{thm:TwoWell}]
In order to deduce the upper bound of \cref{thm:TwoWell}, we choose \(N \sim \epsilon^{-\frac{1}{3}}\), which shows the claim.
\end{proof}

\begin{figure}
 \begin{tikzpicture}[scale=0.5]
  \draw[->] (8,0,0) -- (9,0,0) node[below] {$e_1$};
  \draw[->] (8,0,0) -- (8,1,0) node[right] {$e_2$};
  \draw[->] (8,0,0) -- (8,0,1) node[below] {$e_3$};

  \draw (0,0,0) -- (5,0,0) -- (5,7,0) -- (0,7,0) -- (0,0,0); 
  \draw (0,0,0) -- (0,0,7) -- (5,0,7) -- (5,0,0); 
  \draw[gray] (0,0,7) -- (0,7,7) -- (0,7,0); 
  \draw[gray] (5,0,7) -- (5,7,7) -- (5,7,0); 
  \draw[gray] (0,7,7) -- (5,7,7); 

  \draw (0,7,0) -- (0,9.8,0) -- (5,9.8,0) -- (5,7,0); 
  \draw (2.5,7,0) -- (2.5,9.8,0);
  \draw[dashed] (1.25,7,0) -- (1.25,9.8,0);
  \draw[dashed] (3.75,7,0) -- (3.75,9.8,0);

  \draw (0,0,7) -- (0,0,9.8) -- (5,0,9.8) -- (5,0,7); 
  \draw (2.5,0,7) -- (2.5,0,9.8);
  \draw[dashed] (1.25,0,7) -- (1.25,0,9.8);
  \draw[dashed] (3.75,0,7) -- (3.75,0,9.8);

  \draw[gray] (0,0,9.8) -- (0,9.8,9.8) -- (0,9.8,0); 
  \draw[gray] (5,0,9.8) -- (5,9.8,9.8) -- (5,9.8,0); 
  \draw[gray] (0,9.8,9.8) -- (5,9.8,9.8); 
  \draw[gray] (2.5,9.8,0) -- (2.5,9.8,9.8) -- (2.5,0,9.8); 

  \draw (6,0,0) -- (6,9.8,0);
  \draw (5.95,0,0) -- (6.05,0,0);
  \draw (5.95,7,0) -- (6.05,7,0);
  \draw (5.95,9.8,0) -- (6.05,9.8,0);
  \node[right] at (6,3.5,0) {$h_1$};
  \node[right] at (6,8.4,0) {$\theta h_1$};

  \draw (6,0,0) -- (6,0,9.8);
  \draw (5.95,0,7) -- (6.05,0,7);
  \draw (5.95,0,9.8) -- (6.05,0,9.8);
  \node[below right] at (6,0,3.5) {$h_1$};
  \node[below right] at (6,0,8.4) {$\theta h_1$};

  \draw (0,10.5,0) -- (5,10.5,0);
  \draw (0,10.45,0) -- (0,10.55,0);
  \draw (5,10.45,0) -- (5,10.55,0);
  \node[above right] at (2.5,10.5,0) {$l_1$};

  \draw (0,11,0) -- (2.5,11,0);
  \draw (0,10.95,0) -- (0,11.05,0);
  \draw (2.5,10.95,0) -- (2.5,11.05,0);
  \node[above] at (1.25,11,0) {$\frac{l_1}{2}$};

  \draw (1.5,0,0) -- (1.5,7,0);
  \draw (1.5,0,0) -- (2.5,7,0);
  \draw (3,0,0) -- (4,7,0);
  \draw[dashed] (2.5,0,0) -- (2.5,7,0);

  \draw (1.5,0,0) -- (1.5,0,7);
  \draw (1.5,0,0) -- (2.5,0,7);
  \draw (3,0,0) -- (4,0,7);
  \draw[dashed] (2.5,0,0) -- (2.5,0,7);

  \draw (0.75,7,0) -- (0.75,9.8,0); 
  \draw (0.75,7,0) -- (1.25,9.8,0);
  \draw (1.5,7,0) -- (2,9.8,0);

  \draw (3.25,7,0) -- (3.25,9.8,0); 
  \draw (3.25,7,0) -- (3.75,9.8,0);
  \draw (4,7,0) -- (4.5,9.8,0);

  \draw (0.75,0,7) -- (0.75,0,9.8); 
  \draw (0.75,0,7) -- (1.25,0,9.8);
  \draw (1.5,0,7) -- (2,0,9.8);

  \draw (3.25,0,7) -- (3.25,0,9.8); 
  \draw (3.25,0,7) -- (3.75,0,9.8);
  \draw (4,0,7) -- (4.5,0,9.8);

  \fill[draw = none, fill= red, fill opacity = 0.1] (0,0,0) -- (5,0,0) -- (5,9.8,9.8) -- (0,9.8,9.8) -- cycle;
  \draw[dashed, red] (0,0,0) -- (0,9.8,9.8);
  \draw[dashed, red] (5,0,0) -- (5,9.8,9.8);
  \draw[red] (2.5,7,7) -- (2.5,9.8,9.8);

  \draw[opacity=0.3] (1.5,7,0) -- (1.5,7,7) -- (1.5,0,7);
  \draw[opacity=0.3] (2.5,7,0) -- (2.5,7,7) -- (2.5,0,7);
  \draw[opacity=0.3] (4,7,0) -- (4,7,7) -- (4,0,7);

  \draw[opacity=0.3] (0.75,0,7) -- (0.75,7,7) -- (0.75,7,0);
  \draw[opacity=0.3] (3.25,0,7) -- (3.25,7,7) -- (3.25,7,0);
  \draw[opacity=0.3] (0.75,0,9.8) -- (0.75,9.8,9.8) -- (0.75,9.8,0);
  \draw[opacity=0.3] (1.25,0,9.8) -- (1.25,9.8,9.8) -- (1.25,9.8,0);
  \draw[opacity=0.3] (2,0,9.8) -- (2,9.8,9.8) -- (2,9.8,0);
  \draw[opacity=0.3] (3.25,0,9.8) -- (3.25,9.8,9.8) -- (3.25,9.8,0);
  \draw[opacity=0.3] (3.75,0,9.8) -- (3.75,9.8,9.8) -- (3.75,9.8,0);
  \draw[opacity=0.3] (4.5,0,9.8) -- (4.5,9.8,9.8) -- (4.5,9.8,0);

  \draw[red] (1.5,0,0) -- (1.5,7,7);
  \draw[red] (1.5,0,0) -- (2.5,7,7);
  \draw[red] (3,0,0) -- (4,7,7);

  \draw[red] (0.75,7,7) -- (0.75,9.8,9.8); 
  \draw[red] (0.75,7,7) -- (1.25,9.8,9.8);
  \draw[red] (1.5,7,7) -- (2,9.8,9.8);
  \draw[red] (3.25,7,7) -- (3.25,9.8,9.8); 
  \draw[red] (3.25,7,7) -- (3.75,9.8,9.8);
  \draw[red] (4,7,7) -- (4.5,9.8,9.8);

  \fill[draw = none, fill = orange, fill opacity = 0.3] (1.5,0,0) -- (3,0,0) -- (4,7,7) -- (2.5,7,7) -- cycle; 
  \fill[draw = none, fill = orange, fill opacity = 0.3] (0.75,7,7) -- (1.5,7,7) -- (2,9.8,9.8) -- (1.25,9.8,9.8) -- cycle; 
  \fill[draw = none, fill = orange, fill opacity = 0.3] (3.25,7,7) -- (4,7,7) -- (4.5,9.8,9.8) -- (3.75,9.8,9.8) -- cycle; 
 \end{tikzpicture}
\caption{The three-dimensional branching construction to achieve the boundary data on all sides. The shaded regions are the diagonal interface at \(x_2 = x_3\) with the interfaces of \(A+E_{e_2}\) and \(A+E_{e_3}\) marked in orange.}
\label{fig:3dBranching}
\end{figure}
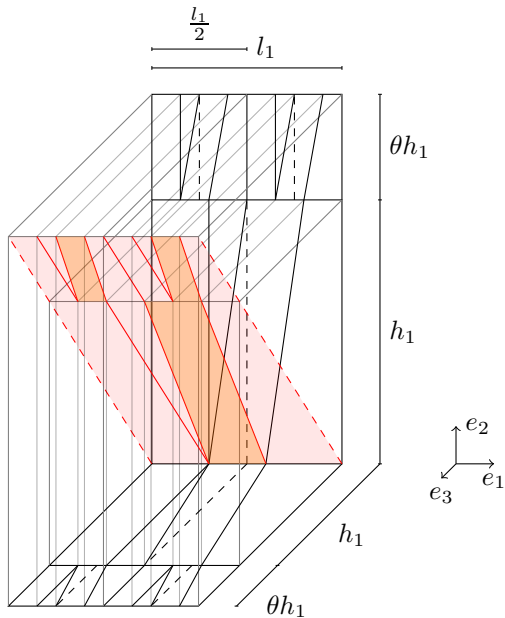

\begin{rmk}[Generalizations]
Building on the ideas from the gradient case and the ones from above one can formulate (rather restrictive) conditions, allowing for similar constructions for more general linear, constant coefficient differential operators. A key difficulty here consists in the ``hard form'' of the prescribed boundary conditions. When considering ``softer forms'' of these, as for instance in \cite{CKO99}, constructions for general constant coefficient operators with the desired boundary conditions would be feasible under much more general conditions by using Fourier theoretic arguments as in \cite{CO09}. We do not pursue these ideas here but postpone this to possible future work.
\end{rmk}

\section{On the role of the divergence operator} 
\label{sec:RoleDivergence}

Following \cite{ST21}, in this section we highlight the relevance of the divergence operator which also partially motivates our discussion of the scaling law for the $T_3$ problem from \cref{sec:IntroT3}. 

To this end, we first recall that considering any first order homogeneous constant coefficient differential operator \(\A(D) = \sum_{j=1}^d A^j\p_j: C^\infty(\R^d;\R^n) \to C^\infty(\R^d;\R^m)\), we can rewrite \(\A(D) u = \di \omega_1(u)\) with a linear map \(\omega_1: \R^n \to \R^{m \times d}\).
Indeed, let us define
\begin{align*}
 \omega_1: \R^n \to \R^{m \times d},\ \omega_1(x) = \big( \sum_{k=1}^n A^j_{ik} x_k \big)_{\substack{i=1,\dots,m,\\ j=1,\dots,d}}.
\end{align*}
Let \(u \in C^\infty(\R^d;\R^n)\), then it holds \(\A(D) u = (\di \circ \omega_1) (u)\) for the row-wise divergence.

Indeed this can be generalized for higher order operators $\A(D)u = \sum_{|\alpha|=k} A^\alpha \partial^\alpha u$, where $A^\alpha \in \R^{m \times n}$ are coefficient matrices, as follows.
For this we denote the space of symmetric $k$ tensors on $\R^d$ by $\text{Sym}(\R^d,k)$.
Let the $k$-th order divergence be given as
\begin{align} \label{eq:generalDiv}
 \di^k: C^\infty(\R^d;\R^m \otimes \text{Sym}(\R^d,k)) \to C^\infty(\R^d;\R^m), \, (\di^k u)_{j} := \sum_{1 \leq i_1 \leq \dots \leq i_k \leq d} \partial_{i_1} \dots \partial_{i_k} u_{j i_1 \dots i_k} .
\end{align}
For $k=1$ this is exactly the row-wise divergence as mentioned above.

\begin{rmk}
 This definition is natural in that sense that this operator (up to a sign) is the adjoint of the $k$-th derivative $D^k$.
\end{rmk}

The linear map $\omega_k: \R^n \to \R^{m} \otimes \text{Sym}(\R^d,k)$ then takes the form 
\begin{align} \label{eq:LinearTransform}
 (\omega_k(x))_{j i_1 \dots i_k} := \Big( A^{\sum_{l=1}^k e_{i_l}}x \Big)_j
\end{align}
and by this choice it holds $\A(D) u= (\di^k \circ \omega_k)(u)$ for any $u \in C^\infty(\R^d;\R^n)$ and $\ker \omega_k = I_{\A} = \bigcap_{|\alpha|=k} A^\alpha$.
In what follows, we omit the $k$ dependence of $\omega = \omega_k$ in the notation.

With this in hand, it is possible to bound the energy for a general homogeneous linear operator \(\A(D)\) (of order $k$) by the corresponding energy for the ($k$-th order) divergence (c.f. \cite[Appendix]{ST21} for the corresponding qualitative result in the case $k=1$).

\begin{prop}
\label{prop:div_reduction}
Let $d,n,k \in \N$, $\K \subset \R^n$, and let $\Omega \subset \R^d$ be a bounded Lipschitz domain.
Let $\A(D)$ be a $k$-th order homogeneous linear differential as in \eqref{eq:operator} and the elastic and surface energies be given by \eqref{eq:E-Elastic} and \eqref{eq:E-Surface}. Moreover let $\omega=\omega_k$ be the linear transformation in \eqref{eq:LinearTransform} and $\di^k$ the generalized $k$-th order divergence in \eqref{eq:generalDiv} with the corresponding energies $E_{el}^{\di^k}, E_{surf}^{\di^k}$.
Then there exist constants $C_1,C_2 > 0$ such that for any $\chi \in BV(\Omega;\K), F \in \R^{n}$
\begin{align*}
 E_{el}(\chi;F) & \geq C_1 E_{el}^{\di^k}(\omega(\chi); \omega(F)), \\
 E_{surf}(\chi) & \geq C_2 E_{surf}^{\di^k}(\omega(\chi)).
\end{align*}
Moreover if $\A(D)$ is cocanceling (and thus $\omega$ is injective), it also holds for all $u \in \mathcal{D}_F, \chi \in BV(\Omega;\K)$ (cf. \eqref{eq:admissible})
\begin{align*}
 E_{el}(u,\chi) + \epsilon E_{surf}(\chi) \sim E_{\epsilon}^{\di^k}(\omega(u),\omega(\chi)) + \epsilon E_{surf}^{\di^k}(\omega(\chi)).
\end{align*}
\end{prop}

\begin{proof}
In order to obtain the desired result, we use the pointwise bound \(|u-\chi| \geq C |\omega(u)-\omega(\chi)|\) and consider the adapted boundary data:
For any \(u \in \mathcal{D}_F\) with $\mathcal{D}_{F}$ denoting the set from \eqref{eq:admissible}, the composition \(\omega(u)\) satisfies \(\di^k \omega(u) = \A(D) u = 0\) in \(\R^d\) and \(\omega(u) = \omega(F)\) in \(\R^d \setminus \overline{\Omega}\). In other words, it holds that \(\omega(u) \in \mathcal{D}_{\omega(F)}^{\di^k}\) for the divergence operator and boundary data \(\omega(F)\).
For the elastic energy (denoting by \(E_{el}^{\di^k}, \mathcal{D}_{\omega(F)}^{\di^k}\) the energy and domain for the divergence operator) this implies
\begin{align}
E_{el}(u,\chi) & = \int_\Omega |u-\chi|^2 dx \geq C \int_\Omega|\omega(u)-\omega(\chi)|^2 dx = C E_{el}^{\di^k}(\omega(u),\omega(\chi)), \label{eq:ReductionDivLower} \\
 E_{el}(\chi;F) & \geq C \inf_{u \in \mathcal{D}_F} \int_\Omega | \omega(u) - \omega(\chi)|^2 dx \geq C \inf_{u : \omega(u) \in \mathcal{D}_{\omega(F)}^{\di^k}} \int_\Omega |\omega(u) - \omega(\chi)|^2 dx. \nonumber
\end{align}
We emphasize that, in general, this only yields lower bound \emph{inequalities} since \(\omega\) is possibly not injective and thus there may be deformations \(u\) with \(\A(D) u = 0\) and \(u \neq F\) outside \(\Omega\) but still fulfilling \(\omega(u) = \omega(F)\) outside \(\Omega\) (see the example in \cref{sec:ComparisonDivergenceProblem} below).
Replacing now \(\omega(u)\) by a general function \(w: \R^d \to \R^{m} \otimes \text{Sym}(\R^d,k))\) such that \(w \in \mathcal{D}_{\omega(F)}^{\di^k}\) yields
\begin{align*}
 E_{el}(\chi;F) \geq C \inf_{w \in \mathcal{D}_{\omega(F)}^{\di^k}} \int_\Omega |w - \omega(\chi)|^2 dx = C E_{el}^{\di^k}(\omega(\chi);\omega(F)).
\end{align*}
Furthermore, as \(|\nabla \chi| \geq c |\nabla (\omega(\chi))|\), we can also bound the surface energy
\begin{align} \label{eq:ReductionDivLowerSurf}
 E_{surf}(\chi) = \int_\Omega |\nabla \chi| \geq c \int_\Omega |\nabla (\omega(\chi))| = c E_{surf}^{\di^k}(\omega(\chi)).
\end{align}

In the case of a cocanceling operator $\omega$ is injective and we also have the bounds $|u-\chi| \leq C|\omega(u)-\omega(\chi)|, |\nabla \chi| \leq C |\nabla (\omega(\chi))|$, thus in \eqref{eq:ReductionDivLower} and \eqref{eq:ReductionDivLowerSurf} also the matching upper bounds hold, which concludes the proof.
\end{proof}

As a consequence, lower bounds for the divergence operator often also imply lower bounds for more general operators. A particular setting (see \cite{ST21}) for instance arises in the three state problem with \(\K=\{A_1,A_2,A_3\}\) being such that \(A_j - A_k \notin \Lambda_\A\). In this case also \(\omega(\K)\) consists of three states which is a result of the fact that the kernel of \(\omega\) is given by
\begin{align*}
 \ker(\omega) = \bigcap_{j=1}^d \ker A^j = \bigcap_{j=1}^d \ker \AA(e_j) = I_\A.
\end{align*}
In particular, if we find a \(T_3\) structure for a general linear, homogeneous, constant coefficient, first order differential operator \(\A(D)\) such that it is mapped to the \(T_3\) structure in \cref{sec:T3}, we can exploit the same lower bound as for the divergence operator. In addition to the relevance of the divergence operator for applications, this argument serves as an additional motivation for focusing particularly on the divergence operator in this article.

Moreover with \cref{prop:reduction_cocancelling} in mind, also for pairwise non super-compatible wells, we can assume without loss of generality that $I_\A = \{0\}$ and thus $\omega$ in injective.

\subsection{Comparison of the two-state problem for the divergence operator} \label{sec:ComparisonDivergenceProblem}

In this section, we discuss the comparison between the general two-state problem for linear, homogeneous differential operators and the one for the ($k$-th order) divergence operator. In particular, this yields yet another proof of the compatible case in \cref{thm:TwoWell}.

In the calculations from the first part of \cref{sec:RoleDivergence}, we notice that for an injective map $\omega$ we can also bound the quantities $E_{el}(u,\chi),E_{surf}(\chi)$ from above with the corresponding term in which $u,\chi$ are replaced by $\omega(u),\omega(\chi)$; hence for $I_\A = \{0\}$ it holds
\begin{align*}
 E_{\epsilon}(u,\chi) \sim E_{\epsilon}^{\di^k}(\omega(u),\omega(\chi)).
\end{align*}
We here emphasize that this only holds on the level of \emph{fixed} $u,\chi$ and that after the minimization in $u$, this does not necessarily yield a two-sided comparison of the energies any more. Indeed, while the lower bound estimates always hold (c.f. \cref{prop:div_reduction}), this may not be true for the upper bound estimates. In fact, even if $I_\A = \{0\}$, we can at the moment not exclude that there may be $w \in \mathcal{D}_{\omega(F)}^{\di^k} \setminus \omega(\mathcal{D}_F)$. We postpone a further discussion of this to future work.

The advantage of $I_\A = \{0\}$ is that we do not lose wells in that sense that for $I_\A = \{0\}$ also $\ker \omega = \{0\}$ and thus, $\omega$ is injective.
As seen above in \cref{sec:super} for two wells $A,B \in \R^n, A-B \notin I_\A$ we can restrict to $\tilde{A}(D)$ which fulfills $I_{\tilde{\A}} = \{0\}$.
This implies that for two compatible wells, which are not super-compatible, in deducing lower scaling bounds, we can use the corresponding lower bounds of the divergence operator as we do not lose information.

\begin{example}
 In concluding this section, we give an example of an operator which is not cocanceling.
 Considering \(d=2,n=3,m=1\) and 
\begin{align*}
 \A(D) u = \p_1 u_2 + \p_2 u_3,
\end{align*}
implies that \(\omega: \R^3 \to \R^{1\times2}, \omega(x) = (x_2 ,\ x_3)\) and \(\ker(\omega) = \vspan(e_1) = I_\A\).

The reduced operator $\tilde{\A}(D)$ would act on mappings taking values only in $\{0\} \times \R^2 \subset \R^3$.
\end{example}

\section*{Acknowledgements}

A.R. and C.T. gratefully acknowledge support through the Heidelberg STRUCTURES Excellence Cluster which
is funded by the Deutsche Forschungsgemeinschaft (DFG, German Research Foundation)
under Germany’s Excellence Strategy EXC 2181/1 - 390900948.

\bibliographystyle{alpha}
\bibliography{refs.bib}

\end{document}